\documentclass[12pt,a4paper]{amsart}

\makeatletter
\long\def\unmarkedfootnote#1{{\long\def\@makefntext##1{##1}\footnotetext{#1}}}
\makeatother

\usepackage{verbatim}
\usepackage{pgf,tikz,pgfplots}
\usepackage{tikz,pgfplots}
\usepackage{mathrsfs}
\usepackage[linecolor=black!50,backgroundcolor=black!5,bordercolor=black!50,textsize=tiny]{todonotes}
\usepackage{ulem}
\usepackage{amsmath,amssymb,amsthm}
\usepackage{enumerate,amssymb}
\usepackage{color}
\theoremstyle{plain}
\newtheorem{thm}{Theorem}[section]

\newtheorem{lemma}[thm]{Lemma}

\newtheorem{prop}[thm]{Proposition}
\newtoks\prt

\theoremstyle{definition}

\newtheorem{remark}[thm]{Remark}

\newtheorem{definition}[thm]{Definition}

\def\eqn#1$$#2$${\begin{equation}\label#1#2\end{equation}}

\numberwithin{equation}{section}

\def\div{\operatorname{div}}
\def\topi{\operatorname{im}_T}
\def\opartial{}
\def\cof{\operatorname{cof}}

\def\adj{\operatorname{adj}}

\def\Deg{\operatorname{Deg}}

\def\F{\mathcal{F}}
\def\diam{\operatorname{diam}}
\def\dist{\operatorname{dist}}

\def\epsilon{\varepsilon}
\def\ep{\varepsilon}
\def\en{\mathbb N}
\def\er{\mathbb R}

\def\ff{\varphi}

\def\haus{\mathcal H^{n-1}}
\def\hau2{\mathcal H^{2}}
\def\hauso{\mathcal H^{n-2}}
\def\im{\operatorname{im}}

\def\loc{\operatorname{loc}}

\def\mir1{\mathcal L_1}

\def\R{\widetilde{R}}

\def\Q{\mathcal{Q}}

\def\R{\widetilde{R}}

\def\rn{\mathbb R^n}
\def\spt{\operatorname{supp}}

\def\ue{\mathbf u}

\def\x{\widetilde{x}}
\def\y{\widetilde{y}}
\def\ve{\mathbf v}

\newdimen\vintkern
\def\vint{{-}\kern-\vintkern\int}

\newtoks\by
\newtoks\paper
\newtoks\book
\newtoks\jour
\newtoks\yr
\newtoks\pages
\newtoks\vol
\newtoks\publ
\def\ota{{\hbox\vol{???}}}
\def\cLear{\by=\ota\paper=\ota\book=\ota\jour=\ota\yr=\ota
\pages=\ota\vol=\ota\publ=\ota}
\def\endpaper{\the\by, {\the\paper},
\textit{\the\jour} \textbf{\the\vol} (\the\yr), \the\pages.\cLear}
\def\endbook{\the\by, \textit{\the\book}, \the\publ.\cLear}
\def\endprep{\the\by, \textit{\the\paper}, \the\jour.\cLear}
\def\endyearprep{\the\by, \textit{\the\paper}, \the\jour, (\the\yr).\cLear}
\def\name#1#2{#2 #1}
\def\nom{ \rm no. }

\long\def\clrred#1\endred{{\color{red}#1}}
\long\def\clrmagenta#1\endmagenta{{\color{magenta}#1}}
\long\def\clrpurple#1\endpurple{{\color{purple}#1}}
\long\def\clr#1\endblue{{\color{black}#1}}

\title{Weak limit of homeomorphisms in $W^{1,n-1}$ and (INV) condition}

\author{Anna Dole\v{z}alov\'a}
\address{Department of Mathematical Analysis, Charles University,
So\-ko\-lovsk\'a 83, 186~00 Prague 8, Czech Republic}
\email{\tt dolezalova@karlin.mff.cuni.cz; hencl@karlin.mff.cuni.cz; maly@karlin.mff.cuni.cz}

\author{Stanislav Hencl}

\author{Jan Mal\'y}



\thanks{The authors were supported 
by the grant GA\v{C}R P201/21-01976S. The first author was also supported in part by the Danube Region Grant no.
8X2043 of the Czech Ministry of Education, Youth and Sports and by the project Grant Schemes at CU, reg. no. CZ.02.2.69/0.0/0.0/19\_073/0016935.}

\date{\today}

\begin{document}

\begin{abstract}
Let $\Omega,\Omega'\subset\er^3$ be 
Lipschitz 
domains, 
let $f_m:\Omega\to\Omega'$ be a sequence of homeomorphisms with prescribed Dirichlet boundary condition and $\sup_m \int_{\Omega}(|Df_m|^2+1/J^2_{f_m})<\infty$. Let $f$ be a weak limit of $f_m$ in $W^{1,2}$. We show that $f$ is invertible a.e., more precisely it satisfies the (INV) condition of Conti and De Lellis and thus it has all the nice properties of mappings in this class. 

Generalization to higher 
dimensions 
and 
an 
example showing sharpness of the condition $1/J^2_f\in L^1$ are also given. 
Using this example we also show that unlike the planar case the class of weak limits and the class of strong limits of $W^{1,2}$ Sobolev homeomorphisms in $\er^3$ are not the same. 
\end{abstract}

\maketitle

\section{Introduction}
In this paper, we study classes of mappings that might serve as classes of deformations in Nonlinear Elasticity models. 
Let $\Omega\subset\rn$ be a domain set and let $f\colon\Omega\to\rn$ be a mapping.
Following the pioneering papers of Ball \cite{Ball} and Ciarlet and Ne\v{c}as \cite{CN} we ask if our mapping is in some sense injective as the physical `non-interpenetration of the matter' 
indicates
that a deformation should be one-to-one.
We are led to study nonlinear classes of mappings based on integrability of gradient minors
(\cite{Ball77}, \cite{Ball}, \cite{Sv}, \cite{GMS}, \cite{MTY}), 
on distortion
(\cite{Re}, \cite{IM}, \cite{HK})
or on finiteness of some energy functional
(\cite{Ball77}, \cite{FG}, \cite{CDL}, \cite{HeMo11}, \cite{HMC}). 
The list of citations is far from being representative and the reader is encouraged
to see also references therein.
Our aim is to 
study injectivity properties 
of the mapping $f$. 
One can follow the ideas of Ball \cite{Ball} and 
assume 
that our mapping has finite energy and that the energy functional $\int_{\Omega} W(Df)$ contains special terms
(like powers of $Df$, $\operatorname{adj} Df$ and $J_f$). 
Under quite strong assumptions, any mapping with finite energy and reasonable boundary data 
is a homeomorphism (\cite{Ball}, \cite{HK}).
However, is not realistic to insist on this requirement as  
in some real physical deformations cavitations or even fractures may occur. Thus we need conditions which guarantee that our mapping is injective a.e.\ but on some small set cavities may arise.  

This was nicely settled in the the work of M\"uller and Spector \cite{MS} where they studied a class of mappings that satisfy $J_f>0$ a.e.\ together with the (INV) condition (see also e.g.\  \cite{BHMC,HMC, HeMo11,MST,SwaZie2002,SwaZie2004,T}). Informally speaking, the (INV) condition means that the ball $B(x,r)$ is mapped inside the image of the sphere $f(S(a,r))$ and the complement $\Omega\setminus \overline{B(x,r)}$ is mapped outside $f(S(a,r))$ (see Preliminaries for the formal definition). 
From \cite{MS} we know that mappings in this class are one-to-one a.e. and that this class is weakly closed which makes it suitable for variational approach. Moreover, 
any mapping in this class has 
many desirable properties, it maps disjoint balls into essentially disjoint \textcolor{black}{sets}, $\deg(f,B,\cdot)\in\{0,1\}$ for a.e. ball $B$, its distributional determinant equals to the absolutely continuous part $J_f$ plus a countable sum of positive multiples of Dirac measures (these corresponds to created cavities) and so on.  

In all results in the previous paragraph the authors assume that $f\in W^{1,p}(\Omega,\rn)$ for some $p>n-1$. 
However in some real models for $n=3$ one often works with 
integrands containing the classical Dirichlet term $|Df|^2$ 
and thus this assumption is too strong. Therefore Conti and De Lellis \cite{CDL} introduced the concept of (INV) condition also for $W^{1,n-1}\cap L^{\infty}$ (see also \cite{BHMCR} and \cite{BHMCR2} for some recent work) and studied Neohookean functionals of the type
\eqn{energy2}
$$
\int_{\Omega}\bigl(|Df(x)|^2+\varphi(J_f(x))\bigr)\; dx
$$
for $n=3$, where $\varphi$ is convex, 
$\lim_{t\to 0+}\varphi(t)=\infty$ and $\lim_{t\to\infty}\frac{\varphi(t)}{t}=\infty$. 
They proved that mappings in the (INV) class that satisfy $J_f>0$ a.e. have nice properties like mappings in \cite{MS} but unfortunately this class is not weakly closed and therefore cannot be used in variational models easily. Our main aim is to fix this and to show that for suitable $\varphi$ the validity of the (INV) condition is preserved also for the weak limit.  
Note that for $p>n-1$ we know that $f$ is continuous on a.e. sphere and thus it is easy to define what is "inside of the image of the sphere $f(S(a,r))$" and 
we can define (INV) easily. In the situation $f\in W^{1,n-1}\cap L^{\infty}$ mappings do not need to be continuous on spheres and it is necessary to use topological degree for 
some classes of discontinuous mappings (in our case,
$W^{1,n-1}$ on a sphere) 
which was introduced by Brezis and Nirenberg \cite{BN} (see also \cite{CDL}). 

Let us note that homeomorphisms clearly satisfy the (INV) condition and 
their weak limits in $W^{1,p},\ p>n-1 ,$ also satisfy (INV) (see \cite[Lemma 3.3]{MS}). Moreover, cavitation can be written as a weak limit (even strong limit) of homeomorphisms.  
Therefore the class of weak limits of Sobolev homeomorphisms is a suitable class for variational models involving cavitation and we can expect some invertibility properties in this class. This is clear for $p>n-1$ because of the (INV) condition but it can fail in the limiting case of limit of $W^{1,n-1}$ homeomorphisms as shown e.g. in Bouchala, Hencl and Molchanova \cite{BHM}. 
The class of weak limits of Sobolev homeomorphisms was recently characterized in the planar case by Iwaniec and Onninen \cite{IO,IO2} and De Philippis and Pratelli \cite{DPP}. The situation in higher dimension is much more difficult and deserves further study. 

Our main result is the following theorem which shows that weak limits of $W^{1,n-1}$ homeomorphisms are invertible a.e. (and much better) under suitable integrability of $1/J_f$. 
We denote 
\eqn{energy_def}
$$
\F(f)=\int_{\Omega}\left(|D f|^{n-1}+\ff(J_f)\right)\,dx,
$$
where 
\eqn{varphi}
$$
\ff \text{ is a positive convex function on }(0,\infty)\text{ with }
\lim_{t\to0^+}\ff(t)=\infty,\ \varphi(t)=\infty\text{ for }t\leq 0
$$
and there is $A>0$ with
\eqn{varphi2}
$$
A^{-1}\ff(t)\le \ff(2t)\leq A\ff(t),\qquad t\in (0,\infty).
$$
We need to further assume that all $f_m$ have the same Dirichlet boundary condition.

\begin{thm}\label{main} 
Let $n\geq 3$, $a=\frac{n-1}{n^2-3n+1}$ and $\Omega,\Omega'\subset\rn$ be Lipschitz domains.
Let $\ff$ 
satisfy \eqref{varphi}, \eqref{varphi2} and 
$$
\ff(t)\geq \frac{1}{t^a}\text{ for every }t\in(0,\infty). 
$$
Let $f_m\in W^{1,n-1}(\Omega,\Omega')$, 
$m =0,1,2\dots$, 
be a sequence of 
homeomorphisms of 
$\overline\Omega$ onto $\overline\Omega'$ 
with $J_{f_m}> 0$ a.e. 
such that 
\eqn{key}
$$
\sup_m \F(f_m)<\infty. 
$$
Assume further that $f_m=f_0$ on $\partial\Omega$ for all $m\in\en$. 
Let $f$ be a weak limit of $f_m$ in $W^{1,n-1}(\Omega,\rn)$ 
, then $f$ satisfies the (INV) condition. 
\end{thm}

As a corollary this weak limit $f$ satisfies all the nice properties (see \cite[Section 3 and 4]{CDL}): it is 
one-to-one 
a.e., it maps disjoint balls into essentially disjoint \textcolor{black}{sets}, 
$\deg(f,B,\cdot)\in\{0,1\}$ for a.e. ball $B$, its distributional determinant equals to the absolutely continuous part $J_f$ plus a countable sum of positive multiples of Dirac measures and so on.  

In fact our result is even more general and instead of integrability of 
$1/J_f^a$ 
it is enough to assume that its distortion $K_f=|Df|^n/J_f$ is integrable with power $\frac{1}{n-1}$ (see Theorem 
\ref{distthm} below). This seems to be connected with the result of Koskela and Mal\'y \cite{KM} about the validity of Lusin $(N^{-1})$ condition for mappings of finite distortion.  

\textcolor{black}
{In our new paper \cite{DHMo} we use Theorem \ref{main} as a main step in showing that we can use Calculus of Variations approach in this context. We show that there is an energy minimizer of certain polyconvex functional in the class of weak limits of $W^{1,n-1}$ homeomorphisms. 
}

In our next result we improve the counterexample from \cite[Theorem 6.3.]{CDL}. 
Similarly to \cite{CDL} we show that the weak limit of homeomorphism\textcolor{black}{s} (that automatically satisfy (INV)) can fail to satisfy (INV). Moreover, the degree of the limit $f$ can be $-1$ on a set of positive measure even though $J_{f_m}>0$ a.e. and $\deg(f_m,B,\cdot)\in\{0,1\}$.
Let us note here that we automatically have $J_f\geq 0$ a.e. for the weak limit once $J_{f_m}>0$ a.e. (see Hencl and Onninen \cite{HO}) \textcolor{black}{and since our $\varphi$ tends to $\infty$ at $0$ we also have $J_f>0$ a.e. from Lemma 2.10 from \cite{DHMo}}. 
Primarily we show that at least in dimension $n=3$ our condition $J^{-2}_f\in L^1$ from Theorem \ref{main} is sharp (example in \cite{CDL} gives smaller integrability of $1/J_f$). We expect that our result is sharp in all dimensions but we have not pursued this as $n=3$ is the physically relevant dimension. 

\begin{thm}\label{example} Let $n=3$ and $a<2$. 
\textcolor{black}{Then there exist homeomorphisms 
$f_m$ of $\overline B(0,10)$ to $\overline B(0,10)$ such that $f_m\in W^{1,2}(B(0,10),B(0,10))$, $f_m$ is an identity mapping on $\partial B(0,10)$ with 
$J_{f_m}>0$ a.e. 
and 
$$
\sup_m \int_{\Omega}\left(|D f_m|^{n-1}+\frac{1}{(J_{f_m})^a}\right)\,dx<\infty, 
$$
whose weak limit $f$} does not satisfy the (INV) condition. 
\end{thm}

It was shown by Iwaniec and Onninen in \cite{IO2} that in the planar case $n=2$ the class of weak limits of homeomorphisms and the class of strong limits of homeomorphisms are the same for any $p\geq 2$. The same result for $1\leq p<2$ was later shown by De Philippis and Pratelli in \cite{DPP}. This is very useful in Calculus of Variations as we can approximate the minimizer of the energy not only in the weak but also in the strong convergence. We show that the situation is much more difficult in higher dimension and the result is not true in general. We show (see  Theorem \ref{distthm} (b) below) that \textcolor{black}{each strong limit of $W^{1,n-1}$ homeomorphisms
satisfies the (INV) condition}. Together with Theorem \ref{example} this implies the following result.

\begin{thm}\label{notthesame} Let $n=3$. There is a mapping $f\in W^{1,2}(B(0,10),B(0,10))$ which is a weak limit of 
Sobolev $\textcolor{black}{W}^{1,2}$
homeomorphisms $f_m$ 
of $\overline B(0,10)$ to $\overline B(0,10)$
with $f_m(x)=x$ on $\partial B(0,10)$ and $J_{f_m}>0$ a.e., but there are no homeomorphisms $h_m$ 
of $\overline B(0,10)$ to $\overline B(0,10)$
such that $h_m\to f$ strongly in $W^{1,2}(B(0,10),\er^3)$. 
\end{thm}

Again we expect that there \textcolor{black}{are similar examples} in $W^{1,n-1}$ in higher dimension. However, we do not see any simple way to generalize this counterexample to other Sobolev spaces $W^{1,p}$ for $p\neq n-1$.




\section{Preliminaries}

\subsection{Convention}
In what follows, we assume that Sobolev mappings are represented in the best way for our purposes.
For example, if we consider a trace of a Sobolev mapping $f$ on a $k$-dimensional surface $S$,
then we assume that $f$ is represented as the trace on $S$. If $f$ has a continuous representative
$\bar f$ on $S$, then we assume that $f=\bar f$ on $S$.

\subsection{Shapes \textcolor{black}{ and Figures}}
In what follows we will work with open sets of a simple geometric nature
called
\textit{shapes}. Our shapes will be of three types
\begin{enumerate}[(a)]
\item balls,
\item \textcolor{black}{full} cuboids,
\item a \textit{hollowed cuboid} is the difference 
$Q\setminus \overline B$, where $Q$ is a \textcolor{black}{full} cuboid and $B\subset\subset Q$
is a ball \textcolor{black}{(see the set $H$ on Fig. \ref{mriz} below)}.
\end{enumerate}
A \textit{figure} is the interior of a finite union of 
closed \textcolor{black}{full} cuboids with pairwise disjoint interiors; it may be disconnected.

\subsection{Boundary gradient and cofactors}
Let $K\subset\rn$ be a shape and $f$ be a smooth mapping on a neighbourhood of 
$\overline K$. 
Then we can derive a useful degree formula (cf.\ \eqref{weakdegree} below) 
involving $(\cof Df(x))\nu(x)$, 
$x\in \partial K$; here $\nu(x)$ denotes the exterior normal to $B$ at $x$
and $\cof A$ is the cofactor matrix of $A$
(\textcolor{black}{which is the transpose of 
the adjugate matrix $\adj A$ and which satisfies $\cof A\; A^T=(\det A)\mathrm{Id}$}). 
If $f$ is a Sobolev mapping only, we can find ``good shapes''
with the property that $f$ is Sobolev regular on their boundaries. 
Let $x\in\partial K$, in case of a cuboid we exclude points of edges. 
Then, instead 
of the expression $(\cof Df(x))\nu(x)$  we need a replacement relying on the tangential gradient
$D_{\tau}f$.
This can be given using some tools from the multilinear algebra, for
details see \cite{F} and \cite{MTY}.
We consider a linear subspace $\mathbb V$ of $\rn$ and a linear mapping
$L\colon \mathbb V\textcolor{black}{\to \mathbb{R}^n}$. Then the operator $\Lambda_{n-1}L$ is defined by
\begin{equation}\label{howLambda}
\Lambda_{n-1}L(\mathbf y_1\wedge\mathbf y_2\wedge\dots\wedge\mathbf y_{n-1})
=L\mathbf y_1\wedge L\mathbf y_2\wedge\dots\wedge L\mathbf  y_{n-1}.
\end{equation}
If the dimension $k$ of $\mathbb V$ is less that $n-1$, this operator is trivial. If
$k=n$, and $A$ is the matrix representing $L$, it can be shown that $\Lambda_{n-1}L$
is represented by $\cof A$. Therefore both sides of \eqref{howLambda} depend only
on values of $L$ on the linear hull of $ \mathbf y_1,\mathbf y_2,\dots\mathbf y_{n-1}$.
We may identify the wedge product with the cross product through the Hodge star operator.
Thus, in our case when $\mathbb V$ is the tangent space $T_x(\partial K)$,
$\Lambda_{n-1}T_x(\partial \textcolor{black}{K})$ is the one-dimensional space of multiples of $\nu(x)$
and $L=D_{\tau}f(x)$, 
the required expression $(\Lambda_{n-1}L)\nu$ can be computed 
(avoiding exterior algebra objects)
as $(\cof A)\nu$ where $A$ is a matrix
of any operator $\bar L\colon \rn\to\rn$ which extends $L$, i.e.\
$$
\bar L\mathbf y=L\mathbf y,\qquad \mathbf y\in T_x(\partial K).
$$
This also shows compatibility with the formula for smooth mappings where the extension
$\bar L$ appears naturally as the full gradient $Df(x)$.

\subsection{Degree for continuous mappings}


Let $\Omega\subset\rn$ be a bounded open set.
Given a \textcolor{black}{smooth map $f\colon\overline\Omega\to\rn$ and $y\textcolor{black}{_0}\in \rn\setminus f(\partial\Omega)$ such that $J_f(x)\neq 0$ for each $x\in \Omega\cap f^{-1}(y\textcolor{black}{_0})$,}
we can define the {\it topological degree} as
$$\deg(f,\Omega,y_0)=\sum_{\Omega\cap f^{-1}(y)} \operatorname{sgn}(J_f(x)).$$
By uniform approximation, 
this definition can be extended to an arbitrary continuous mapping 
$f\colon\overline\Omega\to\rn$ \textcolor{black}{and $y_0\in  \rn\setminus f(\partial\Omega)$}. Note that the degree depends only on
values of $f$ on $\partial \Omega$.

If $f\colon \overline\Omega\to\rn$ is a homeomorphism,
then either $\deg (f,\Omega,y)=1$ for all $y\in f(\Omega)$
($f$ is \textit{sense preserving}), or 
$\deg (f,\Omega,y)=-1$ for all $y\in f(\Omega)$
($f$ is \textit{sense reversing}). If, in addition,
$f\in W^{1,n-1}(\Omega,\rn)$, then this topological orientation
corresponds to the sign of the Jacobian. More precisely, we have

\begin{prop}[\cite{HM}]\label{p:top=anal} 
Let $f\in W^{1,n-1}(\Omega,\rn)$ be 
a homeomorphism on $\overline\Omega$ with $J_f>0$ a.e. Then 
$$
\deg(f,\Omega,y)=1,\qquad y\in f(\Omega).
$$
\end{prop}


\subsection{Degree for $W^{1,n-1}\cap L^{\infty}$ mappings}\label{degree}

If $K$ is a shape or a figure, 
$f\in W^{1,n-1}(\partial K)\cap C(\partial K)$, $|f(\partial K)|=0$, 
and $\ue\in C^1(\rn,\rn)$, then
\begin{equation}\label{weakdegree}
\int_{\rn}\deg(f,K,y)\operatorname{div} \ue(y)\; dy=
\int_{\partial K} (\ue\circ f)\cdot (\Lambda_{n-1} D_{\tau}f)\nu\; d\haus,
\end{equation}
see \cite[Proposition 2.1]{MS}.

Let $\mathcal M(\rn)=C_0(\rn)^*$ be the space of all signed Radon measures on $\rn$.
By \eqref{weakdegree} we see that $\deg(f,K,\cdot)\in BV(\rn)$ and
\begin{equation}\label{BVest}
\|D\deg(f,K,\cdot)\|_{\mathcal M(\rn)}\le 
C\|\Lambda_{n-1} D_{\tau}f\|_{L^1(\partial K)}\le C\|D_{\tau}f\|_{L^{n-1}(\partial K)}^{n-1}.
\end{equation}

Following \cite{CDL} (see also \cite{BN}) we need a more general version 
of the degree on 
the boundary of a shape 
which works for mappings in $W^{1,n-1}\cap L^{\infty}$ 
that are not necessarily continuous. 
Although only the three dimensional case 
on balls 
is discussed
on \cite{CDL}, the arguments pass in the general case as well. The definition is in fact based 
on \eqref{weakdegree}.

\begin{definition}
Let $K\subset\rn$ be a shape and let $f\in W^{1,n-1}(\partial K,\rn)\cap L^{\infty}(\partial K,\rn)$. Then we define 
$\Deg(f, K, \cdot)$ as the distribution satisfying
\eqn{qqq}
$$
\int_{\rn}\Deg(f,K,y)\psi(y)\; dy=
\int_{\partial K} (\ue\circ f)\cdot(\Lambda_{n-1} D_{\tau}f) \nu\; d\haus
$$
for every test function $\psi\in C_c^{\infty}(\rn)$ 
and every 
$C^{\infty}$ vector field $\ue$ on $\rn$ satisfying $\div \ue=\psi$.
\end{definition}

As in \cite{CDL} it can be verified that the right hand side does not depend on the way 
\textcolor{black}{$\psi$} is expressed as $\div \ue$. Indeed, this works if $f$ is smooth.
If we approximate $f$ by a sequence $(f_m)_m$ of smooth functions in the usual mollification way,
then $\ue\circ f_m\to \ue\circ f$ weakly* in $L^{\infty}$, $\Lambda_{n-1} D_{\tau}f_m\to \Lambda_{n-1} D_{\tau}f$
strongly in $L^1$, hence the right hand side of \eqref{weakdegree} converges well. In fact,
the distribution $\Deg(f, K,\cdot)$ can be represented as a $BV$ function by the following lemma.

\begin{lemma}\label{l:strongdeg}
Let $K$ be a shape.
Let $(f_m)_m$ be a sequence of 
continuous Sobolev mappings 
\textcolor{black}{which converges} to a limit function $f$
strongly in $W^{1,n-1}(\partial K,\rn)$ and \textcolor{black}{is bounded} in $L^{\infty}(\partial K,\rn)$. Then
$\Deg (f, K,\cdot)$ is an integer valued function in $BV(\rn)$ and 
$\deg(f_m, K,\cdot)\to \Deg(f, K,\cdot)$ strongly in $L^1(\rn)$.
\end{lemma}

\begin{proof}
Let $\psi$ be a smooth test function and $\ue$ be a smooth vector field such that $\div\ue=\psi$.
As above we observe that $\ue\circ f_m\to \ue\circ f$ weakly* in $L^{\infty}$ and $\Lambda_{n-1} D_{\tau}f_m\to \Lambda_{n-1} D_{\tau}f$
strongly in $L^1$. Hence we observe that $\deg(f_m, K,\cdot)\to \Deg(f, K,\cdot)$ in distributions.
By \eqref{BVest}, the sequence $\deg (f_m,\opartial K,\cdot)$ is bounded in $BV(\rn)$, 
so that the limit is in $BV$ as well and the convergence is weak* in $BV$.
By the compact embedding and \textcolor{black}{the} $L^{\infty}$ bound of $f_m$
we have $\deg (f_m,\opartial K,\cdot)\to \Deg (f,K,\cdot)$ in $L^1(\rn)$.
\textcolor{black}{It also} follows that $\Deg (f,\opartial K,\cdot)$ is integer valued.
\end{proof}

\begin{remark}
Let $K$ be a shape and $f\in W^{1,n-1}(\partial K)\cap C(\overline K)$.
If $|f(\partial K)|=0$, then $\Deg(f,K,y)=\deg(f,K,y)$ for 
a.e.\  $y\in\rn$. We use different symbols to distinguish and emphasize that
$\deg$ is defined pointwise on $\rn\setminus f(\partial K)$, whereas
$\Deg$ is determined only up to a set of measure zero.
\end{remark}

Assume that $f,g\in W^{1,n-1}(\partial K,\rn)\cap L^{\infty}(\partial K,\rn)$. 
From the embedding of $BV$ spaces into $L^{\frac{n}{n-1}}$ \cite[Theorem 3.47]{AFP}, \textcolor{black}{the} definition of BV norm  \cite[Definition 3.4]{AFP} and \eqref{qqq} (note that by approximation it must hold also for 
$\psi=\div\ue$, $\ue\in C_0^1(\rn)$) we obtain 
\eqn{odhad}
$$
\begin{aligned}
\bigl|\bigl\{y\in\rn:\ &\Deg(f,\opartial K,y)\neq \Deg(g,\opartial K,y)\bigr\}\bigr|^{\frac{n-1}{n}}
\leq \|\Deg(f,\opartial K,\cdot)-\Deg(g,\opartial K,\cdot)\|_{L^{\frac{n}{n-1}}}\\
&\leq C\bigl\|D (\Deg(f,\opartial K,\cdot)-\Deg(g,\opartial K,\cdot))\bigr\|_{\mathcal{M}(\rn)}\\
&=C\sup\Bigl\{\int_{\rn} \left(\Deg(f,\opartial K,y)-\Deg(g,\opartial K,y)\right) 
\operatorname{div} \ue(y)\; dy:\\
&\phantom{aaaaaaaaaaaaaaaaaaaaaaaaaaaaa}  \ue\in C^1_0(\rn),\ \|\ue\|_{L^{\infty}}\le 1 \Bigr\}\\
&\leq C \int_{\partial K\cap\{f\neq g\}}\bigl( |D_{\tau}f(x)|^{n-1}+|D_{\tau}g(x)|^{n-1}\bigr)d\haus(x).\\
\end{aligned}
$$

\subsection{(INV) condition}
Analogously to \cite{CDL} (see also \cite{MS}) we define the (INV) class.

\begin{definition}\label{image} 
Let $B\subset\rn$ be a ball and let $f\in W^{1,n-1}(\partial B,\rn)\cap L^{\infty}(\partial B,\rn)$. 
We define the topological image of $B$ under $f$, $\topi(f,B)$ 
as the set of all points where the density of the set $\{y\in \rn:\ \Deg(f,\opartial B,y)\neq 0\}$
is one. 
\end{definition}

\begin{definition}\label{inv} 
Let $f\in W^{1,n-1}(\Omega,\rn)\cap L^{\infty}(\Omega,\rn)$. We say that $f$ satisfies (INV) in the ball 
$B\subset\subset\Omega$ 
if 
\begin{enumerate}[(i)]
  \item its trace on $\partial B$ is in $W^{1,2}\cap L^{\infty}$;
	\item $f(x)\in \topi (f,B)$ for a.e.\ $x\in B$;
	\item $f(x)\notin \topi (f,B)$ for a.e.\ $x\in\Omega\setminus B$.
\end{enumerate}
We say that $f$ satisfies (INV) if for every $a\in\Omega$ there is $r_a>0$ such that for $\mathcal{H}^1$-a.e. $r\in (0,r_a)$ it satisfies (INV) in $B(a,r )$. 
\end{definition}

\begin{remark}\label{r:inv}
If $f$, in addition, satisfies $J_f>0$ a.e., then preimages of sets of zero measure
have zero measure and thus we can characterize the (INV) condition in a simpler way.
Namely, such a mapping satisfies the (INV) condition in the ball $B\subset\subset\Omega$
if and only if
\begin{enumerate}[(i)] 
  \item its trace on $\partial B$ is in $W^{1,2}\cap L^{\infty}$;
	\item $\Deg(f,B,f(x))\textcolor{black}{\neq 0}$ for a.e.\ $x\in B$;
	\item $\Deg(f,B,f(x))=0$ for a.e.\ $x\in\Omega\setminus B$.
\end{enumerate}
\end{remark}

\subsection{Estimates  of measure of preimages}\label{area}\

\begin{lemma}\label{Ninv} 
Let $\Omega\subset\rn$ be an open set of finite measure
and $f\in W_{\loc}^{1,1}(\Omega;\rn)$ satisfy $J_f\neq 0$ a.e. Then
for every $\ep>0$ there is $\delta>0$ such that for every measurable set $F\subset\rn$ we have
$$
|F|<\delta\implies |f^{-1}(F)|<\ep.
$$
\end{lemma}
\begin{proof}
Assume for contradiction that there are $\epsilon>0$ and $F_j$ with $|F_j|<\frac{1}{2^j}$ such that
$|f^{-1}(F_j)|\ge \ep$, $j=1,2,\dots$. Then the set
$$
E:=\bigcap_{k=1}^{\infty}\bigcup_{j=k}^{\infty}F_j
\text{ with }f^{-1}(E)=\bigcap_{k=1}^{\infty}\bigcup_{j=k}^{\infty}f^{-1}(F_j)
$$
satisfies $|E|=0$ but $|f^{-1}(E)|\ge \ep$.
We can find a set $A\subset f^{-1}(E)$ of full measure such that $J_f\neq 0$ on $A$ and such that change of variables formula 
$$
\int_A |J_f(x)|\; dx=\int_{f(A)}N(f,\Omega,y)\; dy. 
$$
holds on $A$ (see \cite{F} or the proof of \cite[Theorem A.35]{HK} for $\eta=\chi_{f(A)}$).
Now the left hand is positive as $J_f\neq 0$ and $|A|>0$ and the right  side is zero as $|f(A)|\subset E$ and $|E|=0$. This gives us a contradiction.  
\end{proof}

Let $\Omega\subset\rn$ be open, $A\subset\Omega$ be measurable and let $g\in W_{\loc}^{1,1}(\Omega;\rn)$
be one-to-one. 
Without any additional assumption we have (see e.g.\  \cite[Theorem A.35]{HK} for $\eta=\chi_{g(A)}$)
\eqn{area1}
$$
\int_{A}|J_g(x)| dx\le |g(A)|.
$$

\begin{lemma}\label{l:reverse}
Given $C_1<\infty$, 
there exists a function $\Phi\colon (0,\infty)\to(0,\infty)$ with 
$$
\lim_{s\to 0^+}\Phi(s)=0
$$
such that \textcolor{black}{the following holds}: Let $g\in W^{1,n-1}(\Omega,\rn)$ be a one-to-one mapping with $\|g\|_{L^{\infty}}\leq C_1$ and $\F(g)\le C_1$, where \textcolor{black}{$\F$ is as in \eqref{energy_def} with $\varphi$ satisfying} \eqref{varphi}. Then for each measurable set $A\subset\Omega$ we have
\begin{equation}\label{reverse}
\Phi(|A|)\le |g(A)|.
\end{equation}
\end{lemma}

\begin{proof}
Choose $t_0>0$ such that $\ff$ is decreasing on $(0,t_0)$ and 
write $t=\ff_L^{-1}(s)$ if 
\begin{equation}\label{partialinv}
\ff(t)=s \text{ and }0<t<t_0. 
\end{equation}
Then, either $\frac{|g(A)|}{|A|}\ge t_0$, or we use that $\ff$ is decreasing on $(0,t_0]$, area formula  \eqref{area1} and the Jensen inequality to obtain 
$$
\ff\Big(\frac{|g(A)|}{|A|}\Big)\leq \ff\Big(\frac{\int_A|J_g|}{|A|}\Big)
\le \vint_{A}\ff(J_g)\,dx\le \frac{C_1}{|A|}.
$$
This implies that 
\begin{equation}\label{alphacomp}
|g(A)|\ge \Phi(|A|), 
\end{equation} 
where
$$
\Phi(s)=
\begin{cases}
s\ff_L^{-1}\Big(\frac{C_1}{s}\Big),& s< \frac{C_1}{\ff(t_0)} ,\\
t_0s,& s\ge \frac{C_1}{\ff(t_0)} ,
\end{cases}
$$
and $\ff_L^{-1}$ is the left partial inverse function defined by \eqref{partialinv}.

\end{proof}

\textcolor{black}
{
We need the following observation from \cite[Lemma 2.3]{DHMo} to show that the limit mapping in Theorem \ref{main} satisfies $J_f\neq 0$.
\begin{lemma}\label{jfnonzero}
Let $\Omega\subset\rn$ be open, and let $f_k\in W^{1,1}(\Omega,\rn)$ be a sequence of homeomorphisms with $J_{f_k}>0$ a.e.\ such that $f_k\to f\in W^{1,1}(\Omega,\rn)$ pointwise a.e. Assume further that
$$
\sup_k\int_{\Omega}\varphi(J_{f_k}(x))\; dx<\infty,
$$ 
where $\varphi$ satisfies \eqref{varphi}. Then $J_f\neq 0$ a.e. 
\end{lemma}
}

\subsection{Minimizers of the tangential Dirichlet integral}\label{s:dirichlet}
In our main proof we have a sphere \textcolor{black}{(or cuboid)} $K$ in $\rn$ and on this sphere we have a small \textcolor{black}{$(n-2)$-dimensional} circle which is a boundary of an open spherical cap $S\subset K$. Our map $f$ is in $W^{1,n-1}$ so we can choose the sets so that $f$ is continuous on the \textcolor{black}{small} circle $\overline S\setminus S$. Our mapping $f$ can have a big oscillation on $S$ so we need to replace \textcolor{black}{it} by a reasonable mapping. We do this by choosing a minimizer of the tangential Dirichlet energy over this cap $S$ which has the same value on the circle $\overline S\setminus S$. In fact we need this even for more general shapes than spheres and circles.

Let $K\subset\rn$ be a shape. 
We say that a relatively open set $S\subset \partial K$ satisfies the \textit{exterior ball condition}
if for each $z\in \overline S\setminus S$ there exists a ball $B(z',r)$ with
$z'\in\partial K$ such that $z\in \partial B(z',r)$ and $B(z',r)\cap S=\emptyset$.

\begin{thm}\label{minimizers} 
 Let $K\subset\rn$ be a shape. Let $S\subset\partial K$
be a connected relatively open subset of $\partial K$ which does not contain
points of edges. 
Let $T$ be the relative boundary 
of $S$ with respect to $K$. Suppose that $\diam S<\frac{\textcolor{black}{r}}{4n}$ and that $S$ satisfies
the exterior ball condition. Let 
$f=(f^1,\dots,f^n)\in W^{1,n-1}(\partial K,\rn)$ be continuous on $T$. Then there exists 
a unique function $h=(h^1,\dots,h^n)\in C(\overline S)\cap W^{1,n-1}(S,\rn)$ such that each coordinate
$h^i$ minimizes $\int_{S}|D_{\tau}u|^{n-1}\,d\haus$ among all functions $u\in f^i+W_0^{1,n-1}(S)$.
We have $h=f$ on $T$, 
the function $h$ satisfies the estimate
\begin{equation}\label{osc}
\diam h(S)\leq \sqrt n\diam f(T)
\end{equation}
\textcolor{black}{and we have $\mathcal{L}^n(h(S))=0$. }
Moreover, \textcolor{black}{if $f_m$ are} continuous and converge to $f$ uniformly on $T$, then $h_m$ converge to $h$ uniformly on $S$, where $h_m$ are minimizers corresponding to boundary values $f_m$. 
\end{thm}

\begin{proof}
We give the proof for the case of a ball $K=B$. In case of a \textcolor{black}{full} cuboid everything is much simpler as $S$ is flat, and the same references for
properties of minimizers are valid. In case of a hollowed cuboid,
$S$ is a part of the boundary of a cuboid or of a sphere. \textcolor{black}{The part $\mathcal{L}^n(h(S))=0$ is proven in \cite[Theorem 2.16]{DHMo}. }

Choose $z=(z_1,\dots,z_n)\in S$. We may assume that $B=B(0,1)$ and that $z_n\ge \frac1{\sqrt n}$.
Let $\Pi$ be the projection $x\mapsto \hat x:=(x_1,\dots,x_{n-1})$. 
For each $x\in S$ we have $x_n>0$ and 
$$
|\hat x|\le |\hat z|+|\hat x-\hat z|\le \sqrt{1-\frac1n}+\frac1{4n}\le 1-\frac{1}{4n}.
$$
If $u\in W^{1,n-1}(S)$ and $\hat u(\hat x)=u(\hat x,\sqrt{1-|\hat  x|^2})$, then
$$
|D_\tau u|^2=|D\hat u|^2-(\hat x\cdot D\hat u)^2.
$$
Indeed, we can extend $u$ to the neighbourhood of $S$ in $\rn$ as 
$u(x)=u(\hat{x},\sqrt{1-|\hat x|^2})$ and 
then $Du=(D_1\hat u,\dots,D_{n-1}\hat u,0)$.
We clearly have 
$$
|Du|^2=|D_{\tau}u|^2+|D_{\nu}u|^2=|D_{\tau}u|^2+|(Du\cdot\nu)|^2,
$$
and for the unit ball we have $Du\cdot\nu =D\hat u \cdot\hat x$
as $\frac{\partial u}{\partial x_n}=0$. 

Note that 
$$
\xi\mapsto |\xi|^2-(\hat x\cdot \xi)^2,\qquad \xi\in \er^{n-1},
$$
is a positive definite quadratic form whenever $|\hat x|<1$ and 
$$
|\xi|^2-(\hat x\cdot \xi)^2\ge (1-|\hat x|^2)|\xi|^2.
$$
The functional
$$
\int_S|D_{\tau}u|^{n-1}\,d\haus=\int_{\Pi(S)}(|D\hat u|^2-(\hat x\cdot D\hat u)^2)^{\frac{n-1}{2}}\,d\hat x
$$
thus satisfies the axioms of Chapter 5 in \cite{HKM}.
The existence and uniqueness of the minimizer follows from \cite[Theorem 5.28]{HKM}.
The continuity up to the boundary follows from \cite[Theorem 6.6 and Theorem 6.31]{HKM}.
The oscillation estimate \eqref{osc} follows from the maximum principle \cite[Theorem 6.5]{HKM}.
The uniform convergence of a sequence of solutions
can be obtained from the comparison principle \cite[Lemma 3.18]{HKM},
namely, if $u$, $v$ are continuous scalar minimizers and $u\le v$ on $T$, then
$u\le v$ on $S$.
\end{proof}

\subsection{Mappings of finite distortion} Let 
$n\geq 2$ and $\Omega\subset\rn$ be a domain. 
The mapping $f \in W_{\loc}^{1,1}(\Omega, \R^{n})$ is said to be \textcolor{black}{a} mapping of finite distortion if $J_{f}(x) \ge  0$ a.e. in  $\Omega$, $J_{f} \in L_{\loc}^{1}(\Omega)$
and $Df(x)$ vanishes a.e. in the zero set of $J_{f}(x)$ (note that the last condition automatically holds if $J_f>0$ a.e.). 
With such a mapping $f$ we may associate the distortion function as 
$$
K_f(x) = 
\begin{cases}
\frac{|Df(x)|^n}{J_f(x)}&\text{ if }J_f(x)>0\\
1&\text{ if }J_f(x)=0.\\
\end{cases}
$$
See \cite{IM} and \cite{HK} and references given there for the introduction to the theory of mappings of finite distortion.

Let $f\in W^{1,1}(\Omega,\rn)$ be a mapping of finite distortion with 
$K_f\in L^{\frac{1}{n-1}}$ and $u\in C^1(\rn)$. Then \textcolor{black}{the following} crucial estimate from Koskela and Mal\'y \cite[(2.1)]{KM} holds (see \cite{KM} for detailed proof) 
\eqn{KM}
$$
\begin{aligned}
\int_{\Omega}|D (u\circ f(x))|\,dx&\le \int_{\Omega} |D u(f(x))|\ |Df(x)|\; dx\\
&\leq \int_{\Omega} |D u(f(x))|\ (K_f(x) J_f(x))^{\frac{1}{n}}\; dx\\
&\leq \Bigl(\int_{\Omega}|Du(f(x))|^n J_f(x)\; dx\Bigr)^{\frac{1}{n}}
\Bigl(\int_{\Omega} K^{\frac{1}{n}\frac{n}{n-1}}_f(x)\; dx\Bigr)^{\frac{n-1}{n}}\\
&\leq\|D u\|_{L^n(f(\Omega))}\|K^{\frac{1}{n-1}}_{f}\|_{L^{1}(\Omega)}^{\frac{n-1}{n}}.\\
\end{aligned}
$$
\subsection{Extension properties of Lipschitz domains}

It is well known that Lipschitz domains are Sobolev extension
domains, see Calder\'on \cite{C} and Stein \cite{S}. The Sobolev extension property holds even for 
so called uniform domains, see Jones \cite{J}. For nice recent progress in the field
of Sobolev extension see Koskela, Rajala and Zhang \cite{KRZ}.

Much less is known if we want to extend a Sobolev homeomorphism on $\overline\Omega$ ($\Omega\subset\rn$ Lipschitz) 
and require injectivity at least on a neighbourhood
of $\overline\Omega$. Such a result would simplify the proof of our 
main theorem. Unfortunately, we are aware only of planar result
and thus we bypass the absence of such a tool in a series of 
auxiliary results (Lemma \ref{l:ven}, Theorem \ref{t:ext},
Lemma \ref{otravne2}).
Note that the planar result due to Koski and Onninen \cite{KO}
deals in fact with a more difficult problem of extension from the boundary.
If we do not start from a function given on the interior, we cannot
use any kind of reflection.

\begin{lemma}\label{l:ven}
Let  $\Omega'\subset\rn$  be a Lipschitz domain. Then there exist a Lipschitz mapping $\ell\colon\overline{\Omega'}\to\rn$
and $\delta>0$
with the following properties:
\begin{enumerate}[\rm(a)]
\item $x\in\partial\Omega'\implies 
\ell(x)=x 
$,
\item $\dist(x,\partial\Omega')<
\delta 
\implies \ell(x)\notin\Omega'$.
\end{enumerate}
\end{lemma}

\begin{proof}
By the definition of a Lipschitz domain,
there exist open sets $U_i\subset\rn$, unit vectors $\ve_i\in\rn$, 
Lipschitz mappings $\Pi_i\colon U_i\to\partial\Omega'\cap U_i$,
$i=1,\dots,m$, and $R,\rho>0$ with the following properties
\begin{enumerate}[(i)]
\item for each $x\in U_i$ there exists $\lambda\in (-R,R)$ such that $x=\Pi_i(x)+\lambda \ve_i$,
\item for each $x\in U_i\cap \partial\Omega'$ and $t\in (0,2R)$ we have
$\Pi_i(x)=x$, $x+t\ve_i\in \rn\setminus\overline{\Omega'}$, $x-t\ve_i \in \Omega'$.
\item $\{x\in\rn\colon\dist(x,\partial\Omega')\le \rho\}\subset\bigcup_i U_i$.
\end{enumerate}
For each $z\in \overline{\Omega'}$ with $\dist(z,\partial\Omega')\le \rho$ find
$B_z=B(z,r_z)$ such that there exists $i=i(j)\in\{1,\dots,m\}$ with $\overline B(z,(m+1)r_z)\subset U_i$.
Using compactness  of $\{z\in\rn\colon\dist(z,\partial\Omega')\le \rho\}$
select finite covering of this sets by balls $B(z_j,r_j)$, $j=1,\dots,p$ with the property
that $r_j=r_{z_j}$ and find a smooth partition of unity $(\omega_j)_j$ 
on 
$\{z\in\rn\colon\dist(z,\partial\Omega')\le \rho\}$
such that $\{\omega_j>0\}=B(z_j,r_j)$, $j=1,\dots,p$.
Set $r=\min\{R/(m+1),\rho,r_1,\dots,r_p\}$ and find $\delta>0$ such that
for each $x\in U_i$, $i=1,\dots,m$, we have
$$
\dist(x,\partial\Omega)<\delta\implies |x-\Pi_i(x)|< r.
$$ 
Set
$$
\aligned
\eta_i(x)&=\sum_{j\colon i(j)=i}\omega_j,\qquad i=1,\dots,m,\\
\ell(x)&=x+m\sum_{i=1}^m \eta_i(x)(\Pi_i(x)-x) \text{ if }x\in\overline\Omega' \text{ and }
\dist(x,\partial\Omega')\le \delta
\endaligned
$$
and extend $\ell$ in a Lipschitz way to $\overline{\Omega'}$.
Then $\ell$ is a Lipschitz mapping which is identity on $\partial\Omega'$.

Fixing $x\in\overline{\Omega'}$ with $\dist(x,\partial\Omega')\le \delta$, we must prove
that $\ell(x)\notin\Omega'$.
We find $i_0\in\{1,\dots,m\}$ such that 
$$
\eta_{i_0}(x)\ge 1/m.
$$
We may assume that $i_0=1$ and that 
$x\in\spt\eta_i$ 
if and only if $i\in\{1,\dots,k\}$ for 
some $k\in \{1,\dots,m\}$. 
Write
$$
x_1:=x+m\eta_1(x)(\Pi_1(x)-x)=\Pi_1(x)+(m\eta_1(x)-1)(|\Pi_1(x)-x|)\ve_1.
$$
Since
$$
(m\eta_1(x)-1)(|\Pi_1(x)-x|)\le mr\le R,
$$
we have $x_1\notin\Omega'$. We have $x_1\in B(x,m\eta_1(x)r)$.
If $k=1$, we are done. 

Now, we proceed by induction.
Write 
$$
x_q=x+m\sum_{i\le q} \eta_i(x)(\Pi_i(x)-x), \qquad q\le k.
$$
By induction hypothesis we have $x_{q-1}\notin \Omega'$, Further,
$$
x_{q-1}\in B(x,m(\eta_1(x)+\dots+\eta_{q-1}(x))r)\subset B(x,mr).
$$
We find $j_q$ such that $i(j_q)=q$ and $|x-z_{j_q}|\le r_{j_q}$. Then 
$$
|x_{q-1}-z_{j_q}|\le mr+r_{j_q}\leq (m+1)r_{j_q}
$$ 
and thus we have $x_{q-1}\in U_q$.
This means that $x_{q-1}$ is of the form $\Pi_q(x_{q-1})+\lambda \ve_q$
with $\lambda<R$. Since $m\eta_q(x)(\Pi_q(x)-x)=\lambda'\ve_q$
with $0\le \lambda'\le mr\le R$, we have
$x_q=\Pi_q(x_{q-1})+(\lambda+\lambda')\ve_q$ with $\lambda+\lambda'<2R$
and it follows that $x_q\notin\Omega'$. 
We conclude that $\ell(x)=x_k\notin\Omega'$.
\end{proof}

\begin{thm}\label{t:ext}
Let $\Omega,\Omega'$ be Lipschitz domains and $f$ be a $W^{1,p}$-homeomorphism of
$\overline\Omega$ onto $\overline{\Omega'}$. Then there exist $\Omega_0\supset\overline\Omega$,
$\Omega_0'\supset\overline{\Omega'}$, and a continuous 
$W^{1,p}$-mapping $\tilde f\colon\Omega_0\to\Omega_0'$
such that $\tilde f=f$ on $\overline\Omega$ and $\tilde f$ maps  $\Omega_0\setminus\Omega$
to $\Omega_0'\setminus\Omega'$.
\end{thm}

\begin{proof}
We use Lemma \ref{l:ven} to $\Omega'$ and we keep the notation from Lemma \ref{l:ven}.
Let $f^*\colon\rn\to\rn$ 
be the usual $W^{1,p}$-extension of $f$, by its construction
it follows that $f^*$ is continuous. Find $\tau>0$ such that
$$
\dist(x,\overline\Omega)<\tau\implies \dist(f^*(x),\partial\Omega')<\delta.
$$
Set
$$
\aligned
\Omega_0 &=\{x\colon \dist(x,\overline\Omega)<\tau\},\\
\tilde f(x)&=
\begin{cases}
f^*(x),& x\in\overline\Omega\text{ or }f^*(x)\notin\Omega' ,\\
\ell(f^*(x)),& x\in \Omega_0\setminus\overline\Omega\text{ and }f^*(x)\in\Omega',
\end{cases}\\
\text{ and }\Omega'_{0}& =\tilde{f}(\Omega_0). \\
\endaligned
$$
It is easily verified that $\tilde f$ has the desired properties.
We use the chain rule (see e.g. \cite[Theorem 2.1.11]{Z} or \cite[Theorem 3.16]{AFP}) to prove the Sobolev regularity of the composition. 
\end{proof}

\section{Limit of homeomorphisms satisfies (INV)}

Recall that \textcolor{black}{our energy \eqref{energy2} is given by} 
$$
\F(f)=\int_{\Omega}(|D f|^{n-1}+\ff(J_f))\,dx,
$$
where $\ff$ is a positive convex function on
$(0,\infty)$ that satisfies \eqref{varphi} 
and  \eqref{varphi2}. 

\begin{thm}\label{distthm}
Let $n\geq 3$, $\Omega,\Omega'\subset\rn$ be 
Lipschitz
domains and
let $\ff$ satisfy \eqref{varphi}  and  \eqref{varphi2}.

Let $f_m\in W^{1,n-1}(\Omega,\Omega')$, 
$m=0,1,2,\dots,$ 
be a sequence of homeomorphisms.  
Let $f_m$ converge weakly in $W^{1,n-1}(\Omega,\rn)$ to a limit function $f$.
Assume further that we have either
\begin{enumerate}[\rm(a)]
\item
 $f_m$ are 
homeomorphisms of $\overline{\Omega}$ onto $\overline{\Omega'}$ \textcolor{black}{ with $J_{f_m}>0$ a.e. such that $f_m=f_0$ on $\partial\Omega$,
 for all $m\in\en$,
\begin{equation}\label{energy}
\F(f_m)\le C_1
\end{equation}
and 
\begin{equation}\label{distortion}
\|K_{f_m}^{\frac1{n-1}}\|_1\le C_1,
\end{equation}
or }
\item $f_m$ converge strongly in $W^{1,n-1}(\Omega,\rn)$  to $f$ \textcolor{black}{and $J_f>0$ a.e.}
\end{enumerate}

Then $f$ satisfies (INV).
\end{thm}

Our main theorem follows easily from this more general result. 

\begin{proof}[Proof of Theorem \ref{main}]
Assumption \eqref{key} clearly implies \eqref{energy} and by 
the Young 
inequality
$$
ab\leq \frac{1}{p}a^p+\frac{1}{p'}b^{p'}\text{ for }a\geq 0, b\geq 0, p>1
$$ 
used for $p=\frac{(n-1)^2}{n}$ (and thus $p'=\frac{(n-1)^2}{n^2-3n+1}$) we obtain 
$$
\int_{\Omega}K^{\frac{1}{n-1}}_{f_m}\,dx=
\int_{\Omega}|Df_m|^{\frac{n}{n-1}}\frac{1}{J^{\frac{1}{n-1}}_{f_m}}\,dx\leq \frac{1}{p}\int_{\Omega}|Df_m|^{n-1}\,dx+\frac{1}{p'}\int_{\Omega} \frac{1}{J_{f_m}^{\frac{n-1}{n^2-3n+1}}}\,dx. 
$$
The conclusion now follows from Theorem \ref{distthm}.
\end{proof}

\textcolor{black}{
\begin{remark}
Using the Young inequality with $p=\frac{(n-1)^2}{n(1-\varepsilon)}$ (and thus $p'=\frac{(n-1)^2}{n^2-3n+1+n\varepsilon}$) we obtain a similar inequality for lower powers of $K_{f_m}$, i.e., the counterexample from Theorem \ref{example} shows that assuming 
$$
\|K_{f_m}^{\frac{1-\varepsilon}{n-1}}\|_1\le C_1
$$
is also not enough to preserve the (INV) condition under weak limits.
\end{remark}
}

\begin{definition}\label{d:goodball}
Let $\Omega\subset\rn$ be open and let $f_m\in W^{1,n-1}(\Omega,\rn)$ be homeomorphisms 
that converge to 
a limit function $f$ weakly in $W^{1,n-1}(\Omega,\rn)$. 
We say that a shape $K\subset\subset\Omega$ is a 
\textit{good shape (in particular, good ball or good cuboid) \textcolor{black}{ with respect to $(f_m)_m$}} if the following properties are satisfied.
\begin{enumerate}[\rm(i)]
\item The trace of $f$ on $\partial K$ is in $W^{1,n-1}(\partial K,\rn)$. 
In what follows we assume that $f$ is represented to coincide with 
this trace on $\partial K$.

\item If $K$ is a \textcolor{black}{full} cuboid or a hollowed cuboid, the the trace of $f$ on each  $(n{-}2)$-dimensional 
edge $E$ of $K$ is in $W^{1,n-1}(E)$ and the trace 
representative of 
$f$ on the closed $(n-2)$-dimensional skeleton of $K$ is continuous. 
\item 
$|f_m(\partial K)|=0$ for all $m\in\en$. 
\item
There is a subsequence of $f_m$ such that 
the convergence $f_{m_k}\to f$ occurs weakly in $W^{1,n-1}(\partial K,\rn)$ and
$\haus$-a.e.\ on $\partial K$, 
(and therefore
$\deg(f_{m_k},K,\cdot)$ forms a bounded sequence in $BV$.)
\end{enumerate}
\end{definition}


\begin{lemma}\label{l:goodball}
Let $\Omega\subset\rn$ be open and let $f_m\in W^{1,n-1}(\Omega,\rn)$ be homeomorphisms 
that converge to 
a limit function $f$ weakly in $W^{1,n-1}(\Omega,\rn)$. 
Let $B(x_0,r_0)\subset \Omega$. Then $B(x_0,r)$ is a good ball \textcolor{black}{ with respect to $(f_m)_m$}
for a.e.\ $r\in (0,r_0)$.
\end{lemma}

\begin{proof}
By slicing analogous to the proof of the ACL property
we obtain that the trace of $f$ on $\partial B$ is in $W^{1,n-1}(\partial B(x_0,r),\rn)$ for a.e. $r>0$. 
Images of spheres by $f_m$ are disjoint as $f_m$ are one-to-one and thus 
$|f_m(\partial B(x_0,r))|=0$ for a.e. $r>0$. The fact that $\deg(f_m,B,\cdot)$ forms a bounded sequence in $BV$ follows from Section \ref{degree}. 

By the Fubini theorem and by 
the Fatou theorem 
\eqn{Fatou}
$$
\int_0^{r_0}\liminf_{m\to\infty}\Big(\int_{\partial B(x_0,r)} |D_{\tau}f_m|^{n-1}\,d\haus\Big)\; dr\leq 
\liminf_{m\to\infty}\int_{B(x_0,r_0)}|Df_m|^{n-1}
\leq C_1. 
$$
The last inequality implies that for a.e. $r$ 
$$
\liminf_{m\to\infty}\Big(\int_{\partial B(x_0,r)} |D_{\tau}f_m|^{n-1}\,d\haus\Big)<\infty
$$
and we can choose a subsequence for which
the limes inferior  turns to the limit. 
Thus, we have a bounded sequence in $W^{1,n-1}(\partial B(x_0,r))$ and we select a weakly convergent subsequence. Since $W^{1,n-1}$ is compactly embedded into $L^{n-1}$ we obtain that this subsequence converge to $f$ in $L^{n-1}$. Up to a subsequence we can thus assume that it converges to $f$ pointwise $\haus$-a.e.\ on $\partial B$.  
\end{proof}

\begin{lemma}\label{l:mriz}
Let $\Omega\subset\rn$ be a bounded open and let $f_m\in W^{1,n-1}(\Omega,\rn)$ be homeomorphisms 
that converge to 
a limit function $f$ weakly in $W^{1,n-1}(\Omega,\rn)$. 
Let $\delta>0$.
Then there exist partitions 
$$
\aligned
&t_1^0<t_1^1<\dots<t_1^{m_1},\\
&t_2^0<t_2^1<\dots<t_2^{m_2},\\
&\hdots\\
&t_n^0<t_n^1<\dots<t_1^{m_n}
\endaligned
$$
such that 
$$
\overline\Omega\subset (t_1^0,t_1^{m_1})\times\dots\times(t_n^0,t_n^{m_n}),
$$
each $t_i^{j}-t_i^{j-1}<\delta$ and 
each 
$$
Q=(t_1^{j_1-1},t_1^{j_1})\times (t_1^{j_2-1},t_1^{j_2})\times\dots\times
(t_1^{j_n-1},t_1^{j_n})
$$
with $1\le j_i\le m_i$, $i=1,\dots,n$, is a good cuboid \textcolor{black}{ with respect to $(f_m)_m$} provided that
$\overline Q\subset\Omega$.
\end{lemma}

\begin{proof}
The proof is analogous to the proof of Lemma 
\ref{l:goodball} with the additional difficulty that we must take care of
the $W^{1,n-1}$-regularity (which implies continuity for
a suitable representative by the Morrey estimates) on 
the $(n{-}2)$-dimensional edges. Therefore we select the partition points
$t_i^j$ subsequently for $i=1,2,\dots,n$ in such a way that \textcolor{black}{the Sobolev regularity on
the intersections of $\{x\in\Omega\colon x_i=t_i^j\}$
with all $\{x\in\Omega\colon x_i=t_{i'}^{j'}\}$
for all $i'<i$ and $j'\in\{1,\dots,m_{i'}\}$ is controled}.
\end{proof}


In the main proof we assume that the (INV) condition fails. Hence we can find a ball $B\subset\Omega$ such 
\textcolor{black}
{ either something from outside of $B$ is mapped into the topological image of the ball or something from inside of $B$ is mapped outside of topological image, i.e. 
}
that the set
$$
\{x\in \Omega\setminus B:\ f(x)\in \im_T(f,B)\} (\text{or }\{x\in B:\ f(x)\notin \im_T(f,B)\}) \text{ has positive measure}. 
$$
At the same time we need to show that also 
\textcolor{black}
{ something from inside of $B$ is mapped inside the topological image and something from outside of $B$ is mapped outside, i.e. that 
}
the following set have positive measure
$$
\{x\in \Omega\setminus B:\ f(x)\notin \im_T(f,B)\} (\text{or }\{x\in B:\ f(x)\in \im_T(f,B)\}). 
$$
This second condition seems to be believable but unfortunately we need the following technical Lemma to show its existence (note that we replace $f(x)\in \im_T(f,B)$ by $\Deg(f,B,f(x))\neq 0$ in \eqref{goal}). The main idea to show this is simple, we extend our $f,f_m:\Omega\to\Omega'$ to mappings $f^*,f_m^*:\Omega_1\to\Omega'_1$ with $\Omega_1\supsetneq\Omega$ so that other conditions holds for these extensions. Now it is not difficult to see that for many points
$$
x\in \Omega_1\setminus\Omega\text{ we have }f(x)\notin \im_T(f,B). 
$$ 
\textcolor{black}{Moreover, we do another important observation there. In the proof of the main theorem we assume that (INV) condition fails and thus either $(ii)$ or $(iii)$ of Definition \ref{inv} fail. We show that if $(ii)$ fails for some ball then $(iii)$ fails for some other shape. It follows that we can assume in the proof of main theorem that $(iii)$ fails.}

\begin{lemma}\label{otravne2} 
Let $n\geq 3$, $\Omega,\Omega'\subset\rn$ be 
Lipschitz domains, let $\ff$ satisfy \eqref{varphi} and \eqref{varphi2}. 
Let $f_m\in W^{1,n-1}(\Omega,\Omega')$, $m=0,1,2,\dots$, be a sequence of 
homeomorphisms of 
$\overline\Omega$ onto $\overline\Omega'$ such that
 $f_m=f_0$ on $\partial\Omega$, $J_{f_m}>0$ a.e., and 
\begin{equation}\label{en}
\F(f_m)\le C_1, 
\end{equation}
which converges weakly in $W^{1,n-1}(\Omega,\rn)$ 
to a limit function $f$. 
Assume that $f$ does not satisfy (INV). 

Then we can find domains $\Omega_1\supset\Omega, \Omega_1'\subset\rn$, 
such that $f_{\textcolor{black}{m}}$ extend to  $W^{1,n-1}$-homeomorphisms 
$f_{\textcolor{black}{m}}^*\colon\overline\Omega_1\to\overline\Omega_1'$
with $f_{\textcolor{black}{m}}^*=f_0^*$ on $\overline\Omega_1\setminus\Omega$, 
\begin{equation}\label{en2} 
\sup_{\textcolor{black}{m}\in\en} \int_{\Omega_1}\ff(J_{f^*_{\textcolor{black}{m}}})\,dx<\infty\text{ and }
\sup_{\textcolor{black}{m}\in\en}\int_{\Omega_1}K^{\frac{1}{n-1}}_{f^*_{\textcolor{black}{m}}}\,dx<\infty,
\end{equation}
and we can find a good shape $K\subset\subset\Omega_1$
for the sequence $f^*_{\textcolor{black}{m}}$ and the limit function
$f^*$
such that 
both sets 
\eqn{goal}
$$
\{x\in  \Omega_1 \setminus K:\ \Deg(f^*,K,f^*(x))\neq 0\}\text{ and }
\{x\in  \Omega_1 \setminus K:\ \Deg(f^*,K,f^*(x))= 0\}
$$
have positive measure. 
\end{lemma}

\begin{proof}
Denote $\mathbb H=\{x\in\rn\colon x_1<0\}$.
Denote the reflection
$(x_1,x_2\dots,x_n)\mapsto(-x_1,x_2,\dots,x_n)$ by $R$.
Since domains $\Omega,\Omega'$ are Lipschitz and we can find a joint
localization of both, there
exist \textcolor{black}{$k$ pairs of} open sets $U_i,V_i\subset\rn$ and \textcolor{black}{of}
bilipschitz mappings $\Phi\textcolor{black}{_i}\colon U_i\to\rn$, $\Psi\textcolor{black}{_i}\colon V_i\to\rn$ \textcolor{black}{with $i\in\{1,\dots,,k\}$}
such that 
\begin{enumerate}[(i)]
\item the sets $U_i$ cover $\partial\Omega$,
\item $f_0(U_i)\subset  V_i$,
\item for each $x\in \overline U_i$  
we have 
$x\in\Omega$ iff $\Phi\textcolor{black}{_i}(x)\in \mathbb H$,
\item for each $x\in \overline V_i$
we have 
$x\in\Omega'$ iff $\Psi\textcolor{black}{_i}(x)\in \mathbb H$,
\item $x\in \Phi\textcolor{black}{_i}(U_i)\setminus\mathbb H\implies R(x)\in\Phi\textcolor{black}{_i}(U_i)$,
\item $x\in \Psi\textcolor{black}{_i}(V_i)\cap\mathbb H\implies R(x)\in\Psi\textcolor{black}{_i}(V_i)$. \\
\end{enumerate}
Then we can construct ``Lipschitz reflections'' near $\partial\Omega$ and 
$\partial\Omega'$,
$$
\aligned
R_i^{\Phi}=\Phi_i^{-1}\circ R\circ\Phi_i,\\
R_i^{\Psi}=\Psi_i^{-1}\circ R\circ\Psi_i.
\endaligned
$$
\textcolor{black}{Let us fix $i\in\{1,\dots,\textcolor{black}{k}\}$. Then for any $m$ we can extend} $f\textcolor{black}{_m}$, $j=0,1,\dots$,
to a Sobolev homeomorphism 
$f^*\textcolor{black}{_m}\colon \overline{\Omega\cup U_i}\to \rn$ setting
$$
f^*\textcolor{black}{_m}(x)=
\begin{cases}
f\textcolor{black}{_m}(x),& x\in\overline\Omega,\\
R_i^{\Psi\textcolor{black}{_i}}( f_0 (R_i^{\Phi}(x))),& x\in U_i\setminus\overline\Omega.
\end{cases}
$$
Also we use the limit function
\eqn{hvezda}
$$
f^*=
\begin{cases}
f(x),& x\in\overline\Omega,\\
f_0^*(x),& x\in U_i\setminus\overline\Omega.
\end{cases}
$$
The Sobolev regularity and continuity are preserved by composition with
the bilipschitz mappings.
We use the property \eqref{varphi2} of $\ff$  and $\F(f_0)<\infty$  to verify
that
$$
\sup_{\textcolor{black}{m}\in\en} \int_{\Omega_1}\ff(J_{f^*\textcolor{black}{_m}})\,dx<\infty\text{ and }
\sup_{\textcolor{black}{m}\in\en}\int_{\Omega_1}K^{\frac{1}{n-1}}_{f^*\textcolor{black}{_m}}\,dx<\infty.
$$
Now we look for a good shape $K$ as in the statement of the theorem.
Since $f$ does not satisfy (INV) on $\Omega$ we can use Lemma \ref{l:goodball} 
and Remark \ref{r:inv} 
to find an (arbitrarily 
small) good ball $B(c,r)\subset \Omega$ such that either 
$$
\{x\in \Omega\setminus B:\ \Deg(f,B,f(x))\neq 0\}
$$
or 
\eqn{defA}
$$
A:=\{x\in B:\ \Deg(f,B,f(x))= 0\}
$$
have positive measure. Since $B$ is small we can assume that 
$3\sqrt n\diam B<\dist(B,\partial\Omega)$.

In the first case 
we find $y_0\in\partial\Omega'$ such that $y$ is a boundary point of
the convex hull $\widehat{\Omega'}$ of $\Omega'$,
find
$i\in\{1,\dots,\textcolor{black}{k}\}$ such that $y_0\in V_i$ and
extend $f\textcolor{black}{_m}$ and $f$ to $\overline\Omega_1$ as in \eqref{hvezda}, where $\Omega_1=\Omega\cup U_i$. 
Notice
that for all $y\in f(\Omega_1)\setminus\overline{\widehat{\Omega'}}$ we have $\Deg(f,B,y)=0$. This can be obtained by approximation. Namely,
by the Mazur lemma there exist convex combinations $g\textcolor{black}{_m}$
of $f\textcolor{black}{_m}$ such that $g\textcolor{black}{_m}\to f$ strongly in $W^{1,n-1}(\Omega,\rn)$.
All the functions $g\textcolor{black}{_m}$ have values in $\widehat{\Omega'}$ 
and therefore
$$
\deg(g\textcolor{black}{_m},B,y)=0,\qquad y\notin \widehat{\Omega'}.
$$
Now we use Lemma \ref{l:strongdeg} 
and \eqref{weakdegree}
to deduce that
\eqn{ctverecek}
$$
\Deg(f^*,B,f^*(x))=0\qquad \text{for a.e. }x\in\Omega_1\setminus (f_0^*)^{-1}
(\widehat{\Omega'}).
$$ 
Hence
our conclusion holds for this ball $B$, but of course for $\Omega_1$ instead of $\Omega$.

The second case, namely that
the set $A$ has positive measure, is more tricky.
By Theorem \ref{t:ext}
there exist $\Omega_0\supset\overline\Omega$,
$\Omega_0'\supset\overline\Omega'$, and a continuous 
$W^{1,n-1}$-mapping $\tilde f_0\colon\Omega_0\to\Omega_0'$
such that $\tilde f_0=f_0$ on $\overline\Omega$ and $\tilde f_0$ maps $\Omega'\setminus\Omega$
to $\Omega_0'\setminus\Omega_0$. We set
$$
\tilde f_m(x)=\begin{cases}
\tilde f_0(x),&  x\in \Omega_0\setminus\Omega,\\
f_m(x),& x\in \Omega\\
\end{cases}
\text{ and }
\tilde f(x)=\begin{cases}
\tilde f_0(x),&  x\in \Omega_0\setminus\Omega,\\
f(x),& x\in \Omega.
\end{cases}
$$

We use Lemma \ref{l:mriz} to construct a  partition \textcolor{black}{
$\Q$ of a neighbourhood
of  $\overline{\Omega}$ into good cuboids with respect to the extended functions.
Moreover,
we may assume that each cuboid $Q\in\Q$ intersects $\overline\Omega$ and is so small that it satisfies}
\begin{equation}\label{howQ}
\diam Q<\diam B
\quad\text{ and }\quad 
Q\cap\partial\Omega\ne\emptyset\implies Q\subset\textcolor{black}{\Omega_0}\cap U_i
\text{ for some $i$}.
\end{equation}
\textcolor{black}{
We define} the figure $F$ as (see Fig. \ref{mriz})
$$
\overline F=\bigcup_{Q\in\Q}\overline Q.
$$

\begin{figure}[h t p]
\phantom{a}
\vskip 220pt
{\begin{picture}(0.0,0.0) 
     \put(-190.2,0.2){\includegraphics[width=0.90\textwidth]{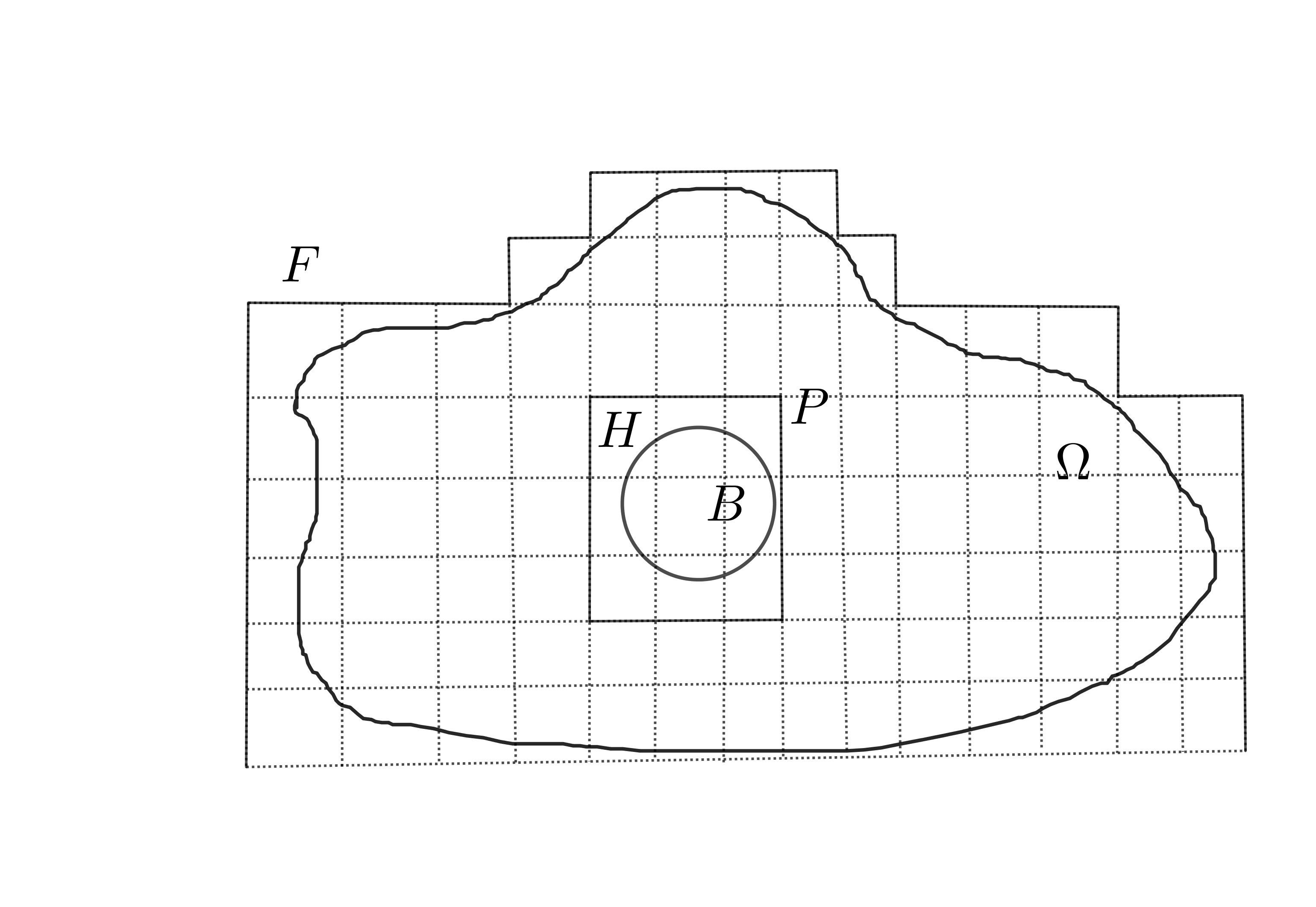}}
  \end{picture}
  } 
\vskip -50pt	
\caption{We cover $\Omega$ by a set of good cuboids $F$ and $B$ by \textcolor{black}{full} cuboids $P$.}\label{mriz}
\end{figure}

Using $\dist(B,\partial\Omega)>3\sqrt n\diam B$ and \eqref{howQ} 
we find $\Q'\subset\Q$ such that the figure $P$ with
$$
\overline P:=\bigcup_{Q\in\Q'}\overline Q
$$
is itself a \textcolor{black}{full} cuboid and 
$$
\overline B\subset P\subset \subset\Omega.
$$
We will consider the hollowed cuboid 
$$
H=P\setminus B.
$$
Denote 
$$
\Q''=\{Q\in\Q\colon Q\cap P=\emptyset\}.
$$

We have \textcolor{black}{(see Fig. \ref{mriz})}
\begin{equation}\label{deg1}
\deg (\tilde f,F,y)=1,\qquad y\in \Omega'.
\end{equation}
Indeed,
$\textcolor{black}{\tilde{f}}=\tilde f_0$ on $\partial F$, $f_0$ is a sense preserving homeomorphism on $\Omega$ 
and 
$$
\deg (\tilde f_0,F,y)=\deg(f_0,\Omega,y)+\deg(\tilde f_0,F\setminus\overline \Omega,y).
$$
Here we use the additivity property of the degree 
and the fact that $\deg(\tilde f_0,F\setminus\overline \Omega,y)=0$
as 
$$
y\notin \tilde f_0(F\setminus\overline \Omega).
$$
\textcolor{black}{
Note that the additivity property of the degree defined by
\eqref{qqq} follows from the fact that the normals
on boundary parts of adjacent surfaces are opposite and 
thus cancellation occurs. }
Further by  \eqref{defA}
$$
\Deg (f,B,y)=0,\qquad y\in f(A).
$$
Now, by \eqref{deg1}
\eqn{suma}
$$
1=\Deg(\tilde f,F,y)=\Deg(f,B,y)+\Deg(f,H,y)+
\sum_{Q\in \Q''}\Deg(\tilde f,Q,y),
\qquad \text{ for a.e. } y\in f(A).
$$

By \eqref{suma} there exists a 
good shape $K$ such that $K\cap B=0$ and $\deg(\tilde f,B,y)\neq 0$
for 
a.e.\ 
$y\in f(A)$, namely either $K=H$ or $K\in \Q''$.
If $K\subset\subset\Omega$, we can proceed as in the preceding case 
(see \eqref{ctverecek})
and add a suitable $U_i$ to $\Omega$.
Let $K\cap\partial\Omega\ne\emptyset$. Then we find 
$i$ such that $K\subset\subset U_i$ and as in the preceding case
extend $f\textcolor{black}{_m}$ as $f\textcolor{black}{_m}^*$
and $f$ as $f^*$
to $\overline{\Omega\cup U_i}$ as in \eqref{hvezda}. 
We \textcolor{black}{now claim that} that
\begin{equation}\label{zbyva}
\Deg(f^*,K,f(x))\neq 0 \quad \text{for a.e.}\quad x\in A.
\end{equation}
We start with showing
that 
\begin{equation}\label{zbyva2}
\Deg(f^*,K,y)=\Deg(\tilde f,K,y)\neq 0\quad \text{for a.e.}\quad y\in f(A).
\end{equation}
To this end we first use a homotopy 
$$
h(y,t)=\Psi_i^{-1}\Bigl(
\Psi_i(\tilde f_0(y))+t\bigl(\Psi_i(f^*_0(y))-\Psi_i(\tilde f_0(y))\bigr)\Bigr)
$$
to \textcolor{black}{prove} that 
\begin{equation}\label{homotopy}
\deg(\tilde f_0,K,y)=\deg(f^*_0,K,y),\qquad y\in\Omega'.
\end{equation}
Let $\psi$ be a smooth function supported in $\Omega'$
and $\ue$ be a smooth function satisfying $\div\ue=\psi$.
Then by \eqref{homotopy} we have
$$
\aligned
&\int_{\partial K\setminus \Omega} 
(\ue\circ f^*_0)\cdot(\Lambda_{n-1} D_{\tau}f^*_0) \nu\; d\haus
\\&\qquad
=
\int_{\partial K} 
(\ue\circ f^*_0)\cdot(\Lambda_{n-1} D_{\tau}f^*_0) \nu\; d\haus
-\int_{\partial K\cap \Omega} 
(\ue\circ f_0)\cdot(\Lambda_{n-1} D_{\tau}f_0) \nu\; d\haus
\\&\quad
=
\int_{\partial K} 
(\ue\circ \tilde f_0)\cdot(\Lambda_{n-1} D_{\tau}\tilde f_0) \nu\; d\haus
-\int_{\partial K\cap \Omega} 
(\ue\circ f_0)\cdot(\Lambda_{n-1} D_{\tau}
f_0 
) \nu\; d\haus
\\&\qquad
=
\int_{\partial K\setminus \Omega} 
(\ue\circ \tilde f_0)\cdot(\Lambda_{n-1} D_{\tau} 
\tilde  f_0 
) \nu\; d\haus.
\endaligned
$$
Hence
$$
\aligned
&\int_{\partial K} 
(\ue\circ f^*)\cdot(\Lambda_{n-1} D_{\tau}f^*) \nu\; d\haus
\\&\qquad
=
\int_{\partial K\setminus\Omega} 
(\ue\circ f^*_0)\cdot(\Lambda_{n-1} D_{\tau}f^*_0) \nu\; d\haus
+\int_{\partial K\cap \Omega} 
(\ue\circ f)\cdot(\Lambda_{n-1} D_{\tau}f 
) \nu\; d\haus
\\&\quad
=
\int_{\partial K\setminus\Omega} 
(\ue\circ \tilde f_0)\cdot(\Lambda_{n-1} D_{\tau}\tilde f_0) \nu\; d\haus
+\int_{\partial K\cap \Omega} 
(\ue\circ f)\cdot(\Lambda_{n-1} D_{\tau}f) \nu\; d\haus
\\&\qquad
=
\int_{\partial K} 
(\ue\circ \tilde f)\cdot(\Lambda_{n-1} D_{\tau}
\tilde f
) \nu\; d\haus.
\endaligned
$$
\textcolor{black}{We thus proved \eqref{zbyva2}. Then,}
by Lemma \ref{Ninv} we have
$$
\Deg(f^*,K,f(x))=\Deg(\tilde f,K,f(x))\neq 0\quad \text{ for a.e. }x\in A.
$$
which establishes \eqref{zbyva}.

\begin{figure}[h t p]
\phantom{a}
\vskip 220pt
{\begin{picture}(0.0,0.0) 
     \put(-210.2,0.2){\includegraphics[width=0.90\textwidth]{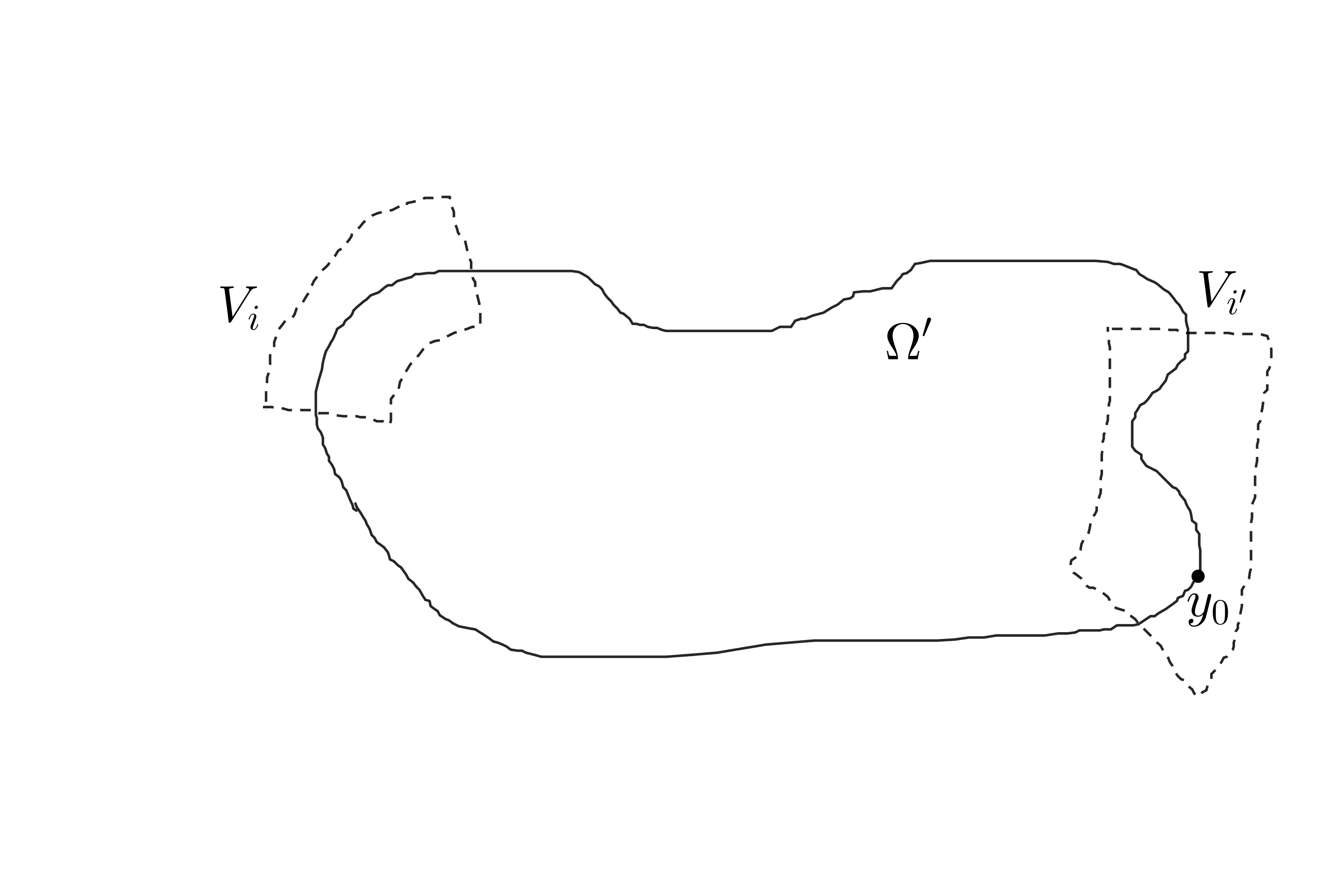}}
  \end{picture}
  } 
\vskip -50pt	
\caption{We add to $\Omega'$ two disjoint sets $V_i$ and $V_i'$.}\label{pridej2kusy}
\end{figure}

Now, we need to extend the function to a still larger set 
$\Omega_1$ (see Fig. \ref{pridej2kusy}). Similarly to the first step, we find $i'\in\{1,\dots,\textcolor{black}{k}\}$
and a point $y_0\in \partial\Omega'\cap \partial\widehat{\Omega'}$
such that $y_0\in V_{i'}$ (recall that $\widehat{\Omega'}$ denotes the convex hull of $\Omega'$). We set
$\Omega_1=\Omega\cup U_i\cup U_{i'}$ and extend $f\textcolor{black}{_m}$ and $f$
to $f^*\textcolor{black}{_m}$ and $f^*$ using ``Lipschitz reflection''
on both $U_i$ and $U_{i'}$. To make it possible, we require
in addition that 
$V_{i}\cap V_{i'}=\emptyset$. This can be achieved if 
the covering of the boundary is chosen fine enough. 
If we consider the strongly converging convex combinations
$g\textcolor{black}{_m}^*$, we observe that $g\textcolor{black}{_m}^*(x)\in \widehat{\Omega'}$
if $x\in \Omega$ and $g\textcolor{black}{_m}^*(x)\in V_{i'}$ if $x\in U_{i'}\setminus\Omega$.
Therefore $g\textcolor{black}{_m}^*(x)\notin V_{i'}\setminus \widehat{\Omega'}$ for 
$x\in K$
and thus 
$$
\Deg(f^*,K,f(x))=0\qquad\text{for a.e. } x\in U_{i'}\setminus (f^*_0)^{-1}
(\widehat{\Omega'}).
$$
\end{proof}

\begin{proof}[Proof of Theorem \ref{distthm}]

\textcolor{black}{{\underline{Step 1. Outline of the proof}:} 
We give here a short informal summary of the proof first. Case (b) is proven by a simplified version of the proof of Case (a), as thanks to the strong convergence we do not need the assumption on integrability of the distortion and Jacobian of $f_m$.}

\textcolor{black}{ 
We start by assuming that $f$ violates the (INV) condition and that "something from outside is mapped inside the topological image". We find a good shape $K$ with respect to $(f_m)_m$ such that
$$
U\setminus K=\{x\in\Omega\setminus K\colon \Deg (f,K,f(x))\ne 0\}
$$
is of positive measure. (In the most simple case, $K$ may be one of the balls which violate the (INV) condition.) Those are the points which originally were outside of $K$ but $f$ mapped them into the topological image of $K$. We cover the boundary of $K$ by a $(n-2)$-dimensional "cage" or "skeleton" made of parts of $(n-2)$-dimensional circles. On this skeleton our functions are Hölder continuous. On the rest of the boundary of $K$ we replace them by $g_m$ and $g$ which are continuous. One can think of it as of having prescribed deformation of the skeleton and $g_m$ and $g$ being a suitable continuous extensions of it on $\partial K$. The differences between the topological images of $f_m$ and $g_m$ (or $f$ and $g$) create bubbles of some kind, through which the material can leave the topological image of $K$ or enter it from the outside (see Figure \ref{defFj}). The neck of such bubble must be getting thinner and thinner as $m$ grows, since in the end the topological image "skips" it completely (see Figure \ref{defUV}). We find two balls $B_U$ and $B_V$ of the same sizes outside of $K$ such that a big parts of them lie in $U$ and $V$, respectively. Since most of $B_U$ is then mapped inside the topological image of $K$ but $B_V$ is mapped outside of it, the lines connecting these two balls must pass through the thin neck of the bubble. That gives a contradiction with our assumption on the integrability of the distortion, as the necks are getting smaller and smaller, but the material of the lines cannot be deformed that much.
}

\textcolor{black}{{\underline{Step 2. Finding a good shape $K$}:} }
\textcolor{black}{We assume for contradiction that $f$ does not satisfy the (INV) condition.}
Assume first Case (a).
Since (INV) fails for $f$, by Lemma \ref{otravne2}
we may assume (passing if necessary to a different domain and different mapping) that there is a good shape $K$ \textcolor{black}{ with respect to $(f_m)_m$} 
such that 
both sets $U\setminus K$ and $V\setminus K$ have positive measure,
where
\eqn{UV}
$$
U=\{x\in\Omega\colon \Deg (f,K,f(x))\ne 0\},\quad V=\{x\in\Omega\colon \Deg (f,K,f(x))= 0\}.
$$

In Case (b),  we find a good ball $K$
such that (INV) is violated on $K$.
Since either (ii) or (iii) from Definition \ref{inv} fails, 
by Remark \ref{r:inv} 
we have 
that either $U\setminus K$ has positive measure,
or $V\cap K$ has positive measure. We will handle the former case, the latter one being similar.

\textcolor{black}{{\underline{Step 3. Finding a skeleton of $\partial K$}:} 
Now, we handle Cases (a) and (b)} together. 
Since $f_m$ converge weakly in $W^{1,n-1}$ and $W^{1,n-1}$ is compactly embedded into $L^{n-1}$ on each ball $B\subset\subset\Omega$,
we obtain that $f_m$ converge to $f$ in $L^{n-1}$ at least locally. 
Up to a subsequence we can thus assume that $f_m\to f$ pointwise a.e. \textcolor{black}
{Using Lemma \ref{jfnonzero} we thus obtain that $J_f\neq 0$ a.e. (for Case (a), as we assume it in Case (b)).}  
Passing if necessary to a subsequence we find a constant $C_2$ such that
$$
\int_{\partial K}(|D_{\tau} f|^{n-1}+|D_{\tau} f_m|^{n-1})\,d\haus < C_2,\qquad m\in\en.
$$
Choose $\ep>0$ small enough whose exact value is specified later. 
Find 
$\rho\in (0,\,\frac{1}{16n}r_0)$ 
such that for each $z\in\partial K$ we have
\eqn{ahoj}
$$
\int_{\partial K\cap B(z,2\rho)}|D_{\tau} f|^{n-1}\,d\haus
<
\ep^{n-1}.
$$
Now, we distinguish \textcolor{black}{three possibilities} according to the form of the shape $K$. \textcolor{black}{We define sets $T_j\subset \partial K$ which form a "skeleton" of $\partial K$. Their key property will be that the diameter of their image under $f$ is small, namely
\eqn{morrey0}
$$
\diam f(T_j)\le C_3\ep.
$$
}

\textbf{First, let $K$ be a ball.}
For each $z\in \partial \textcolor{black}{K}$ we find $\rho_z\in (\rho,2\rho)$ such that 
$$
\aligned
\rho\int_{\partial K\cap \partial B(z,\rho_z)}|D_{\tau} f|^{n-1}\,d\hauso
<
\ep^{n-1}.
\endaligned
$$ 
Analogously to the definition of the good ball we can also assume that $f_m\to f$ occurs $\hauso$-a.e. on $\partial K\cap \partial B(z,\rho_z)$ and that 
$$
\liminf_{m\to\infty}\|f_m\|_{W^{1,n-1}(\partial K\cap \partial B(z,\rho_z))}<\infty. 
$$
It follows that up to a subsequence (see e.g. \cite[Lemma 2.9]{MS})
\eqn{funiformly0}
$$
f_m\to f\text{ weakly in }W^{1,n-1}\text{ and also uniformly on }\partial K\cap \partial B(z,\rho_z). 
$$
Note that on the $(n-2)$ dimensional space $\partial K\cap \partial B(z,\rho_z)$ we have embedding into H\"older functions $W^{1,n-1}\hookrightarrow C^{0,1-\frac{n-2}{n-1}}$ and thus $f$ is continuous there 
and we have the estimate 
\eqn{morrey}
$$
\diam f(\partial K\cap \partial B(z,\rho_z))
\leq C (\rho_z)^{1-\frac{n-2}{n-1}}\Bigl(\int_{\partial K\cap \partial B(z,\rho_z)}|D_{\tau} f|^{n-1}\,d\hauso\Bigr)^{\frac{1}{n-1}}
\le C_3 \ep.
$$
Using \textcolor{black}{a} Vitali type covering, we find $B_j=B(z_j,\rho_j)$ such that $\rho_j=\rho_{z_j}$,
$$
\partial K\subset \bigcup_jB(z_j,\rho_j)
$$
and the balls $B(z_j,\frac15\rho_j)$ are pairwise disjoint. Here $j=1,\dots,j_{\max}$.
 Note that the multiplicity
of the covering is estimated by a constant $N_1$ depending 
only on the dimension since $\rho_z\in (\rho,2\rho)$ for every $z$. 
Furthermore, the balls in the Vitali covering theorem are chosen inductively so we can also assume using \eqref{funiformly0} 
that for a subsequence (chosen in a diagonal argument) 
\eqn{funiformly}
$$
f_m\to f\text{ weakly in }W^{1,n-1}\text{ and uniformly on } 
\partial K\cap \partial B(z_j,\rho_j)\text{ for each }j.
$$ 
Given $j$, denote 
$$
S_j=\partial K\cap B_j\setminus \bigcup_{l<j}\overline B_l.
$$
Note that $S_j$ obviously satisfies the exterior ball condition of Subsection \ref{s:dirichlet}. 
Let $T_j$ denote the relative boundary of $S_j$ with respect to $\partial K$. \textcolor{black}{From \eqref{morrey} we have
\eqref{morrey0}.
}

\textbf{If $K$ is a \textcolor{black}{full} cuboid,} similarly to the proof of Lemma 
\ref{l:mriz} we find partitions of each face of $K$
to $(n{-}1)$-dimensional \textcolor{black}{(full)} cuboids $S_j$ such that,
denoting the relative boundaries of $S_j$ with respect to $K$ by $T_j$,
we have
$\diam S_j<\rho$
and 
$$
\rho\int_{T_j}|D_{\tau} f|^{n-1}\,d\hauso\le \ep^{n-1}.
$$
We can also assume that $f_m\to f$ occurs $\hauso$-a.e. on 
$T_j$ and that 
$$
\liminf_{m\to\infty}\|f_m\|_{W^{1,n-1}(T_j)}
<\infty. 
$$
It follows that up to a subsequence (see e.g. \cite[Lemma 2.9]{MS})
\eqn{funiformly1}
$$
f_m\to f\text{ weakly in }W^{1,n-1}(T_j)\text{ and also uniformly on }T_j. 
$$
By embedding we also have continuity and H\"older estimates
similar to \eqref{morrey} of $f$ on
$T_j$, in particular \textcolor{black}{\eqref{morrey0}}.

\textbf{If $K$ is a hollowed cuboid,}
we \textcolor{black}{construct the skeleton} of flat and round parts of the boundary
combining the methods used for a ball and a cuboid, \textcolor{black}{obtaining sets $T_j$ with the desired property \eqref{morrey0}}.

\textcolor{black}{\underline{Step 4. Replacing $f$ by $g$ with similar degree}:} 
Now we consider the shapes together.
For each $j$ we define $h_j$ on $S_j$ such that $h_j$ minimizes 
coordinate-wise 
the 
tangential $(n-1)$-Dirichlet integral among functions with boundary data $f$ on $T_j$ (see Theorem \ref{minimizers}). We define $h_j=f$ on $\partial K\setminus S_j$. 
Also we define the function $g$ on $\partial K$ as $g=h_j$ on 
each $\overline S_j$.
Set (see Fig. \ref{defFj})
$$
\aligned
F&=\{y\in \Omega'\colon \Deg(f,K,y)\ne \deg(g,K,y) \},\\
F_j&=\{y\in  \Omega'\colon \Deg(f,K,y)\ne \textcolor{black}{\Deg}(h_j,K,y) \}.
\endaligned
$$
\begin{figure}
\phantom{a}
\vskip 220pt
{\begin{picture}(0.0,0.0) 
     \put(-170.2,0.2){\includegraphics[width=0.90\textwidth]{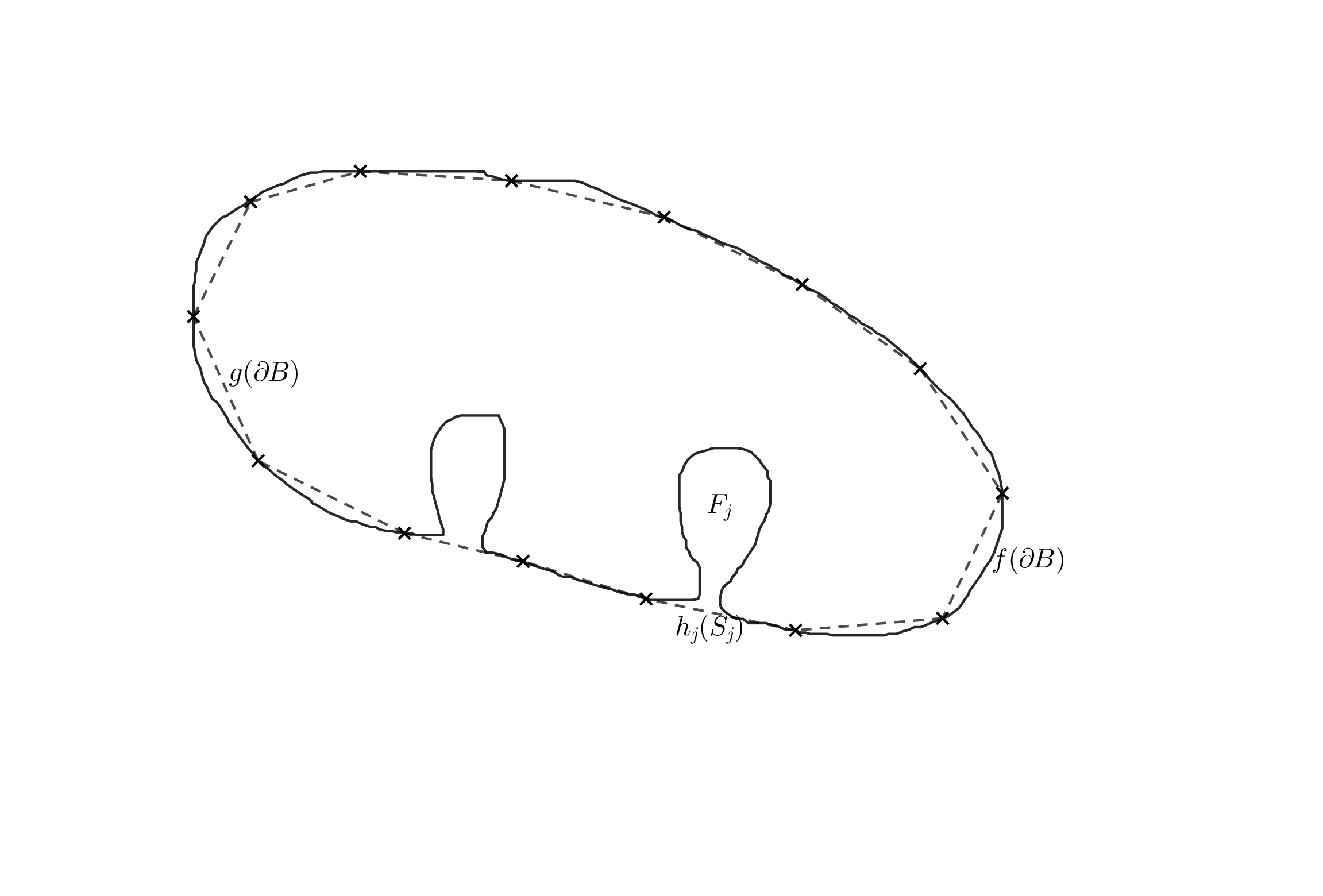}}
  \end{picture}
  }
\vskip -50pt	
\caption{2D representation of the sets $F_j$. $T_j$ corresponds to points on $f(\partial K)$ (of course in $\rn$ they are $(n-2)$-dimensional), $h_j$ is represented by dashed lines connecting these points (of course these are minimizers of $(n-1)$-energy in higher dimensions and not lines) and $F_j$ is created ``between'' $h_j(S_j)$ and $f(S_j)$.}\label{defFj}
\end{figure}

Then 
$$
y\in \bigcup_jF_j \quad \text{for a.e. }y\in F
$$
(this can be viewed e.g. by using \eqref{weakdegree}) 
and, by \eqref{odhad}, \eqref{ahoj}, and the minimizing property
$\int_{S_j}|D_{\tau}h_j|^{n-1}\,d\haus\leq C\int_{S_j}|D_{\tau}f|^{n-1}\,d\haus$ we have
$$
\aligned
\sum_{j}|F_j|&\le C
\sum_{j}\Big(\int_{S_j}(|D_{\tau} f|^{n-1}+|D_{\tau} h_j|^{n-1})\,d\haus\Big)^{\frac{n}{n-1}}
\\&
\le C\sum_{j}\Big(\int_{S_j}|D_{\tau} f|^{n-1}\,d\haus\Big)^{\frac{n}{n-1}}\\
&\le C\ep
\sum_{j}\int_{S_j}|D_{\tau} f|^{n-1}\,d\haus
\le CC_2\ep.
\endaligned
$$

\textcolor{black}{\underline{Step 5. Concluding the proof for Case (b)}:} 
Now, we distinguish the cases again.
Assume (b). 
Since $J_f\neq 0$ a.e. we can choose $\ep$ small enough so that
using Lemma \ref{Ninv} we obtain 
\eqn{Fg2}
$$
|f^{-1}(F)|\le \Bigl|f^{-1}\Bigl(\bigcup_{j}F_j\Bigr)\Bigr|\le \kappa.
$$
\textcolor{black}{We can find $\delta=\delta(\kappa)>0$ such that there exists a set $Z\subset \Omega$ such that $J_f>2\delta$ on $\Omega\setminus Z$ and $|Z|<\kappa/2$. Since (up to a subsequence) $D_{f_m}\to D_f$ pointwise a.e. we obtain $J_{f_m}\to J_f$ pointwise a.e., and hence we can find $m$ big enough such that $J_{f_m}>\delta$ on $\Omega\setminus Z'$, where $|Z'|<\kappa$.
Then we can pass to a subsequence so that we have $J_{f_m}>\delta$ on $\Omega\setminus Z'$ and using \eqref{area1} we obtain 
\eqn{Fg3}
$$
|f_m^{-1}(F)|\leq |Z'|+|f_m^{-1}(F\setminus f_m(Z'))|\le \kappa + \Bigl|f_m^{-1}\Bigl(\bigcup_{j}F_j \setminus f_m(Z')\Bigr)\Bigr|\le \kappa + \frac{CC_2\varepsilon}{\delta}\leq 2\kappa
$$
for all $m\in\en$ when $\epsilon$ is chosen small enough.}
Fix $m\in\en$ and note that 
$$
\aligned
\text{ for every }x\in U\setminus K\text{ we have } \deg(f_m,K,f_m(x))= 0\text{ since }f\textcolor{black}{_m}\text{ is a homeomorphism}. 
\endaligned
$$
Using the definitions of $U$ \eqref{UV}
$$
\begin{aligned}
U \setminus K&\subset\{\deg(f_m,K,f_m(x))=0,\ \Deg(f,K,f(x))\neq 0\}
\\
&\subset \{\deg(f_m,K,f_m(x))\ne \Deg(f,K,f_m(x))\}
\cup \{\Deg(f,K,f_m(x))\ne \textcolor{black}{\deg}(g,K,f_m(x))\}
\\&\cup \{\deg(g,K,f_m(x))\ne \deg(g,K,f(x))\}
\cup \{\Deg(g,K,f(x))\ne \Deg(f,K,f(x))\}.
\end{aligned}
$$
We already know by \eqref{Fg2} and \eqref{Fg3} that 
$$
\{\Deg(f,K,f_m(x))\ne \Deg(g,K,f_m(x))\}\cup \{\Deg(g,K,f(x))\ne \Deg(f,K,f(x))\}<\textcolor{black}{3}\kappa.
$$
Now, since the components of $\rn\setminus g(\partial K)$ are open and $f_m\to f$
a.e.,
we can \textcolor{black}{ assume that $m$ is so large that }
$$
|\{\deg(g,K,f_m(x))\ne \deg(g,K,f(x))\}|<\kappa.
$$
Finally, since $f_m\to f$ strongly, 
for $m$ large enough we have by Lemma \ref{l:strongdeg}
$$
|\{\deg(f_m,K,\cdot)\ne \Deg(f,K,\cdot)\}|<\Phi(\kappa),
$$
so that  by  \eqref{reverse} 
$$
| \{\deg(f_m,K,f_m(\cdot))\ne \Deg(f,K,f_m(\cdot))\}|<\kappa.
$$ 
Altogether,
$|U \setminus K|<\textcolor{black}{5}\kappa$. Since $|U\setminus K|>0$ we can choose $\kappa$ small enough, so that 
we have a contradiction. \textcolor{black}{The case $|V\cap K|>0$ is done analogously.}

\textcolor{black}{\underline{Step 6. Finding balls $B_U$ and $B_V$ which are mostly in $U$ and $V$}:
From now on, we consider only Case (a). }
We can use Lemma \ref{otravne2} and \eqref{UV} and we can thus assume that 
both $U\setminus K$ and $V\setminus K$ have positive measure. 

Let $\Pi$ be \textcolor{black}{the} orthogonal projection onto the hyperplane $\{x\in\rn:\ x_1=0\}$. 
Assume that $\kappa\in(0,\frac16)$ is so small that for each
ball $B$ and each measurable set $E$ we have
\eqn{howdelta1}
$$
|B\setminus E|<5\kappa|B|\implies |\Pi(E)|\ge \frac78|\Pi(B)|.
$$ 
Using Lebesgue density arguments, we find $r>0$ small enough and balls $B_U=B(x_U,r)$ and $B_V=B(x_V,r)$
such that 
\eqn{Udelta1}
$$
\aligned
|B_U\setminus U|\le \kappa|B_U|,\\ |B_V\setminus V|\le \kappa|B_V|
\endaligned
$$
and the convex hull of $\overline B_U\cup \overline B_V$ is contained in $\Omega\setminus K$.
We may assume that $x_V-x_U$ is a multiple of $\mathbf e_1$ (see Fig. \ref{defUV}).

	\begin{figure}
\phantom{a}
\vskip 200pt
{\begin{picture}(0.0,0.0) 
     \put(-170.2,0.2){\includegraphics[width=0.90\textwidth]{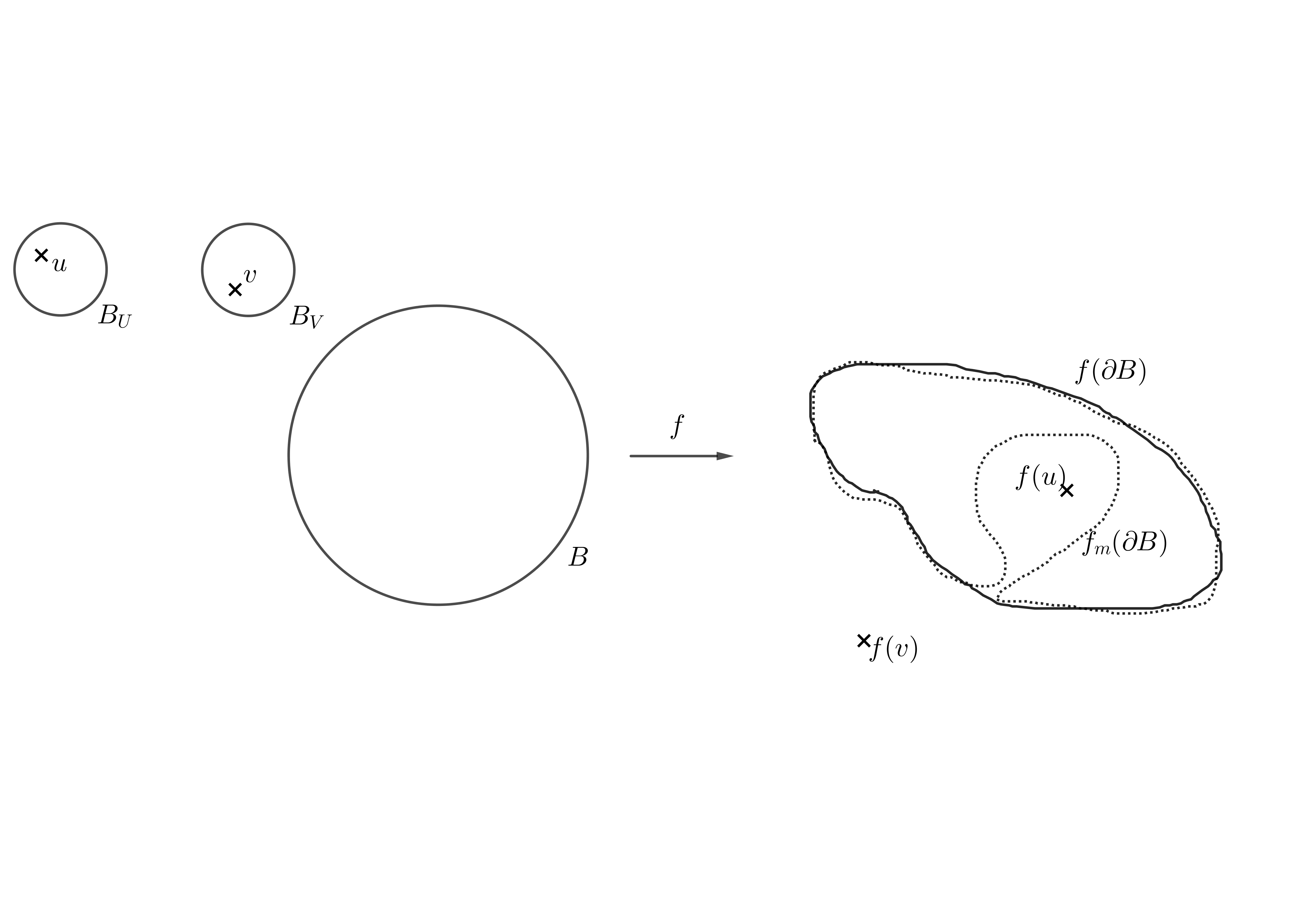}}
  \end{picture}
  }
\vskip -70pt	
\caption{Definition of $B_U$ and $B_V$. Images of most of the points $u\in B_U$ lies inside $f(\partial B)$. On the other hand, for $f_m(\partial B)$ (see dotted line) both points $f(u)$ and $f(v)$ are outside for high enough $m$ and most of the points in $B_U$ and $B_V$.}\label{defUV}
\end{figure}

Choosing $\ep$ small enough we can \textcolor{black}{assume} 
using Lemma \ref{Ninv} that
\eqn{Fg}
$$
|f^{-1}(F)|\le \Bigl|f^{-1}\Bigl(\bigcup_{j}F_j\Bigr)\Bigr|\le \kappa|B_U|. 
$$

\textcolor{black}{\underline{Step 7. Replacing $f_m$ by $g_m$ with similar degree}:} 
Find a compact set $H\subset\Omega'\setminus g(\partial K)$ such that
\eqn{defK}
$$
\Omega'\setminus H<\Phi(\kappa|B_U|).
$$
For each $m\in\en$ and $j\in\{1,\dots,j_{\max}\}$ 
let $g_{m,j}$ be defined in $S_j$ as the 
coordinate-wise minimizer of the 
$(n-1)$-Dirichlet integral among functions with boundary data $f_m$ on $T_j$.
We define $g_{m,j}$ as
$f_m$ on $\partial K\setminus S_j$. We also define
$g_m$ on $\partial K$ as $g_{m,j}$ on each $\overline{S_j}$.

Since $f_m\to f=g$ uniformly on $T_j$ by \eqref{funiformly}
(or \eqref{funiformly1}), we have $g_m\to g$ uniformly on $\partial K$ using Theorem \ref{minimizers}.
Hence we find $m\in\en$ such that $g_m(\partial K)$ 
does not intersect 
$H$ and 
\eqn{W'}
$$
\deg (g_m,K,\cdot)=\deg (g,K,\cdot)\quad\text{in }H.
$$

Also, we require
$$
|f_m-f|=|f_m-g|< \ep\quad\text{on all }T_j.
$$
Similarly as in Fig. \ref{defFj} (but using $f_m$ instead of $f$) we define 
$$
\aligned
E&=\{y\in \Omega'\colon \deg(f_m,K,y)=0\ne \deg(g_m,K,y) \},\\
E_j&=\{y\in  \Omega'\colon \deg(f_m,K,y)=0\ne \deg(g_{m,j},K,y) \}.
\endaligned
$$
Then
$$
y\in \bigcup_jE_j \quad \text{for a.e. }y\in E.
$$
Using \eqref{odhad} and the minimizing property
$\int_{S_j}|D_{\tau} g_{m,j}|^{n-1}\,d\haus\leq C\int_{S_j}|D_{\tau} f_m|^{n-1}\,d\haus$, we obtain
$$
|E_j|^{1-\frac1n}\le C\int_{S_j}|D_{\tau} f_m|^{n-1}\,
d\haus.
$$

\textcolor{black}{\underline{Step 8. Not that many big bubbles where $f_m$ and $g_m$ have different degree}:} 
Choose $a>0$ and set 
$$
\aligned
J^+=\{j\colon \int_{S_j}|D_{\tau} f_m|^{n-1}\,d\haus>a\}, \\
J^-=\{j\colon \int_{S_j}|D_{\tau} f_m|^{n-1}\,d\haus\le a\}. 
\endaligned
$$
Hence
\eqn{rr}
$$
\aligned
\sum_{j\in J^-}|E_j|&\le C\sum_{j\in J^-}\Big(\int_{S_j}|D_{\tau} f_m|^{n-1}\,d\haus\Big)^{\frac{n}{n-1}}
\\&\le Ca^{\frac1{n-1}}\sum_{j\in J^-}\int_{S_j}|D_{\tau} f_m|^{n-1}\,d\haus
\\&
\le Ca^{\frac1{n-1}}\int_{\partial K}|D_{\tau} f_m|^{n-1}\,d\haus\le C_4 a^{\frac1{n-1}},
\endaligned
$$
where $C_4=CC_2$. We fix $a$ such that 
\eqn{rrr}
$$
C_4a^{\frac1{n-1}}\le\Phi( \kappa|B_U|).
$$
We set 
$$
W=f_m^{-1}\Big(\bigcup_{j\in J^-}E_j\Big).
$$
and using \eqref{alphacomp}, \eqref{rr} and \eqref{rrr} we obtain 
\eqn{Ej}
$$
|W|<\kappa|B_U|.
$$

We have
\eqn{defM}
$$
\# J^+\le M:=\frac{C_2}{a}.
$$

\textcolor{black}{\underline{Step 9. A big part of $B_U$ is mapped into big bubbles, a big part of $B_V$ stays away from them}:} 
Now, consider the situation in $B_U$ and $B_V$.
Set 
$$
\aligned
X&=\{x\in B_U\setminus W\colon \deg(g_m,K,f_m(x))\ne 0\},\\
Y&=\{x\in B_V\setminus W\colon \deg(g_m,K,f_m(x))= 0\}.
\endaligned
$$
Using definition of $X$, definition of $U$ \eqref{UV} and \eqref{W'}
$$
\begin{aligned}
B_U\setminus X
&\subset W\cup (B_U\setminus U) 
\cup\{x\in B_U\colon \deg(g_m,K,f_m(x))= 0,\ \Deg(f,K,f(x))\neq 0\}\\
&\subset 
W\cup (B_U\setminus U) 
\cup\{x\in B_U\colon \deg(g,K,f_m(x))= 0,\ \Deg(f,K,f(x))\neq 0\}\cup \{f_m(x)\notin H\}\\
&\subset W\cup (B_U\setminus U)\cup\{\deg(g,K,f(x))\ne \Deg(f,K,f(x))\} \cup\\
&\phantom{\subset}\cup\{\deg(g,K,f_m(x))\ne \deg(g,K,f(x))\}\cup \{f_m(x)\notin H\}.\\
\end{aligned}
$$
Then by \eqref{Udelta1}, \eqref{Ej} and  \eqref{Fg} 
$$
\aligned
|W\cup (B_U\setminus U) |<2\kappa|B_U|\text{ and }\\
| \{\deg(g,K,f(x))\ne \Deg(f,K,f(x))\}|<\kappa|B_U|.
\endaligned
$$
Since the set $\{y\colon \deg(g,K,y)=0\}$ is open and 
$f_m\to f$ a.e., we can take 
$m$ so large that 
$$
|\{\deg(g,K,f_m(x))\ne \deg(g,K,f(x))\}|<\kappa|B_U|.
$$
Finally using \eqref{defK} and \eqref{reverse} (for $f_m$ since $\int \ff(J_{f_m}))\leq C_1$) we obtain 
$$
|\{f_m(x)\notin H\}|\leq \kappa|B_U|
$$
and all these inequalities together give us 
$$
|B_U\setminus X|\le 5\kappa|B_U|.
$$
Similarly using  
$$
\begin{aligned}
B_V\setminus Y
&\subset W\cup (B_V\setminus V) 
\cup\{x\in B_V\colon \deg(g_m,K,f_m(x))\neq 0,\ \Deg(f,K,f(x))= 0\}\\
&\subset W\cup (B_V\setminus V) \cup\{\deg(g,K,f(x))\ne \Deg(f,K,f(x))\} \cup\\
&\phantom{\subset}\cup\{\deg(g,K,f_m(x))\ne \deg(g,K,f(x))\}\cup \{f_m(x)\notin H\}.\\
\end{aligned}
$$
we obtain
$$
|B_V\setminus Y|\le 5\kappa|B_V|.
$$

\textcolor{black}{\underline{Step 10. Concluding the proof for Case (a)}:} 
By \eqref{howdelta1} we have
$$
\textcolor{black}{|}\Pi(B_U\cap X)\textcolor{black}{|}>\frac78\textcolor{black}{|}\Pi(B_U)\textcolor{black}{|},\quad \textcolor{black}{|}\Pi(B_V\cap Y)\textcolor{black}{|}>\frac78\textcolor{black}{|}\Pi(B_V)\textcolor{black}{|},
$$
so that 
\eqn{bigprojection}
$$
|P|>\frac34|\Pi (B_V)|,
$$
where
$$
P=\Pi(B_U\cap X)\cap \Pi(B_V\cap Y).
$$

Consider the segment parallel to the $x_1$-axis that connects $x'\in \textcolor{black}{B_U\cap}X$
with  $x''\in \textcolor{black}{B_V\cap}Y$. We have
$$
\deg(g_{m},K,f_m(x'))\ne\deg(g_m,K,f_m(x''))=0.
$$

\begin{figure}
\phantom{a}
\vskip 220pt
{\begin{picture}(0.0,0.0) 
     \put(-190.2,0.2){\includegraphics[width=0.90\textwidth]{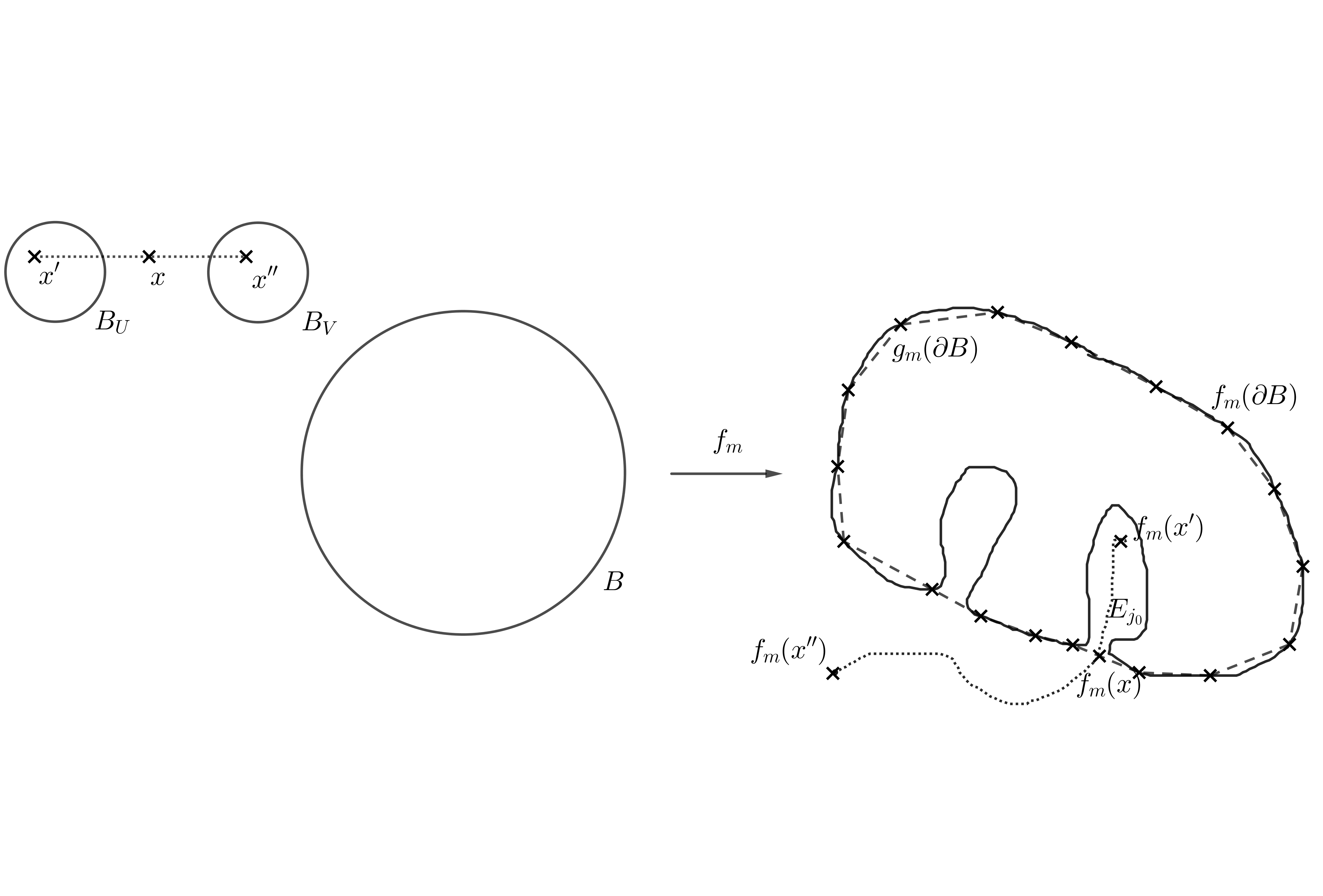}}
  \end{picture}
  }
\vskip -50pt	
\caption{2D representation of the segment $[x',x'']$ and its image (see dotted curves). $f_m(\partial K)$ is a full curve, $T_j$ corresponds to points on $f_m(\partial K)$, $g_m$ is represented by dashed lines connecting these points.}\label{defx}
\end{figure} 

Since $x''$, $x'\notin W=f_m^{-1}(\bigcup_{j\in J^-}E_j)$ 
there exists $j\in J^+$ such that (see Fig. \ref{defx})
$$
\deg(g_{m,j},K,f_m(x'))\ne\deg(g_{m,j},K,f_m(x''))).
$$
Hence there exists $x$ between $x''$ and $x'$ such that
$f_m(x)\in\partial E_j$ (see Fig. \ref{defx}).  
Since $\partial E_j\subset f_m(\overline S_j)\cup g_m(\overline S_j)$ and 
$f_m(x)\notin f_m(\partial K)$ as $x\in \Omega\setminus\overline K$ and $f_m$ is a homeomorphism, 
it follows that
$f_m(x)\in  g_m(\overline S_j) $. 
Using \eqref{bigprojection} and \eqref{defM} we can fix $j_0\in J^+$ such that
$$
\Bigl| \Pi\bigl(\{x\in\Omega\setminus K\colon f_m(x)\in g_m(\overline S_{j_0})\}\bigr)\Bigr|\ge \frac3{4M}|\Pi(B_U)|.
$$
Note that $\diam g_m(\overline S_{j_0})<C\ep$ by \textcolor{black}{\eqref{morrey0}} and Theorem \ref{minimizers}. 
Now, choose $\beta>0$ and use $\ep>0$ so small that the 
$W^{1,n}_0(\Omega')$-capacity of $g_m(\overline S_{j_0})$ in $\Omega'$ 
is smaller than $\beta^n$. It follows that we can find smooth $u\in W_0^{1,n}(\Omega')$ such that $u$ has compact support, $u\equiv 1$ on $g_m(\overline S_{j_0})$ and 
$$
\int_{\Omega'}|Du|^n\,dy\le \beta^n.
$$
It is clear that for each $a\in \Pi\bigl(\{x\in\Omega\setminus K\colon f_m(x)\in g_m(\overline S_{j_0})\}\bigr)$, where $f_m$ is absolutely continuous on the segment $\Pi^{-1}(a)\cap \Omega$, we have 
$$
\int_{\Pi^{-1}(a)\cap \Omega}|Du\circ f_m|\geq 1 
$$
since the function is changing value from $0$ to $1$. 
Therefore, by \eqref{KM}, 
\eqn{zaver}
$$
\frac3{4M}|\Pi(B_U)|\le \int_{\Omega}|D(u\circ f_m)|\,dx\le 
\|D u\|_{L^n(\Omega')}\|K^{\frac{1}{n-1}}_{f_m}\|_{L^1(\Omega)}^{\frac{n-1}{n}}\le \beta C_1^{\frac{n-1}{n}}.
$$
Given $\beta>0$, in the course of the construction we derive $\ep$, then $\rho$ and $m$.
On the other hand, $B_U$, $\kappa$, $a$, $M$ and thus all the left hand side of \eqref{zaver}
do not depend on $\beta$. Thus, by a suitable choice of $\beta$ we obtain a contradiction.
\end{proof}

\section{Counterexample - sharpness of the condition $\frac{1}{J^2_f}\in L^1$}

We use the notation $A\lesssim B$ for $A\leq C\cdot B$, where $C$ is a positive constant which may depend on the dimension $n$ and exponents $a$ and $p$, but not on $\varepsilon$ nor any of the variables. By $A\approx B$ we mean $A\lesssim B$ and $B\lesssim A$.

We first recall some elementary inequalities that we use often in this section. 
For every $y\in[0,1]$ and $p\in(\frac{1}{2},1)$ we have 
$$
1-y^p\textcolor{black}{\leq}1-y
$$
and since the function $y^p$ is concave and its derivative is $p$ at $1$
$$
y^p\leq 1+p(y-1).
$$
Therefore for every $p\in(\frac{1}{2},1)$ we have
\eqn{elementary}
$$
1-y^p\approx 1-y\text{ for every }y\in[0,1]. 
$$
We also use the fact that 
\eqn{gonio}
$$
\sin(\alpha)\approx \alpha\text{ on }[0,\pi/2], \sin(\alpha)\approx \alpha (\pi-\alpha)\text{ on }[0,\pi] \text{ and }\cos(\pi/2-\alpha)\approx\alpha \text{ on }[0,\pi/2].
$$
Note that for $\alpha\in(0,\pi)$ we have the following elementary estimate
\eqn{trojka}
$$
\frac{1}{\sin \alpha}\lesssim \frac{1}{\alpha(\pi-\alpha)}=\frac{1}{\pi}
\left(\frac{1}{\alpha}+\frac{1}{\pi-\alpha}\right). 
$$

\begin{proof}[Proof of Theorem \ref{example}]
{\bf Step 1. Geometrical explanation:} 
We fix a parameter $\varepsilon>0$ small enough, we construct a homeomorphism $f_{\varepsilon}$ and later we choose $f_m$ as $f_{\varepsilon}$ for $\varepsilon=1/m$. 
We define the mapping from spherical coordinates $(r,\alpha,\beta)$ to spherical coordinates. We first define it on $B(0,2)$, i.e. for
$r\in (0,2)$, $\alpha\in (0,\pi)$ and $\beta\in (-\pi,\pi)$. \textcolor{black}{Then we extend it to $B(0,10)\setminus B(0,2)$ so that $f(x,y,z)=(x,y,-z)$ on $\partial B(0,10)$ and then we compose it with a proper reflection.} The mapping has the form 
$$
f_{\varepsilon}\left((r,\alpha,\beta)\right)=\left(\tilde{r}(r,\alpha,\varepsilon),\tilde{\alpha}(r,\alpha,\varepsilon),\beta \right), 
$$ 
i.e. it is enough to define it in the $xz$-plane and then rotate the picture around the \textcolor{black}{$z$-axis} both in the domain and in the target. 

To improve the readability we first give the informal idea about the behaviour of the mapping using pictures and later we give exact formulas. 
In Figure \ref{fm} we show the behaviour of $f_{\varepsilon}$ for $\varepsilon=1/m$ on different spheres in the $xz$-plane.
The outer sphere $\partial B(0,2)$ is mapped onto some drop-shape with $[0,0,0]$ at the very top and this shape is actually the same for all $\varepsilon>0$. 
The behaviour on spheres inside is described for spheres of radius $\frac{1}{2}$ and $\frac{3}{2}$. 
Each sphere $\partial B(0,r)$ inside is divided into two parts - the inner part $I_r$ denoted in a dotted curve and the outer part $O_r$ denoted by a full curve. The boundary between these two regions $W$ is denoted by the thin blue dashed curve and is very important for the behaviour of our map. The image $f_m(O_r)$ is some outer half-drop (denoted by a full curve on the right part of the picture) and the image $f_m(I_r)$ is some inner half-drop (denoted by a dotted curve) so that the image $f_m(B(0,r))$  looks like a "horseshoe". These horseshoes are nested, i.e.  
$f_m(B(0,{r_1}))\subset f_m(B(0,{r_2}))$ for $r_1<r_2$, so that the whole map $f_m$ could be a homeomorphism. Let us describe what happens for $\varepsilon\to 0+$, that is, $m\to\infty$. The tips of all horseshoes (the upper two parts) are approaching the point $[0,0,0]$ on the very top. At the same time $W$ (boundary between inner and outer parts of spheres) is changing drastically but only on $B(0,1)$. The small "pie" on the bottom has very small angle which disappear\textcolor{black}{s} as $\varepsilon\to 0+$ so in the limit there are no outer parts $O_r$ for $r<1$. \textcolor{black}{It is actually possible to do so }with bounded $W^{1,2}$ energy - on each $\partial B(0,r)$, $0<r<1$, we map something like 2D ball or radius $\delta$ (in fact a small spherical cap) to something like 2D ball of radius $1$ with energy $\int_{B^2(0,\delta)}|Dh|^2\approx \hau2(B^2(0,\delta)) |\frac{1}{\delta}|^2\approx 1$. 
\begin{figure}
\phantom{a}
\begin{tikzpicture}
\draw[very thick] (0,0) circle(4);
\draw[dotted] (0,0) circle(2.08);
\draw[orange, very thick] (0,-1.0)arc[start angle=-90, end angle=-86,radius=1] ;
\draw[orange, very thick] (0,-1.0)arc[start angle=-90, end angle=-94,radius=1] ;
\draw[orange, dotted, very thick] (0,0)circle(1);
\draw[green, very thick] (0,-3.0)arc[start angle=-90, end angle=0,radius=3] ;
\draw[green, very thick] (0,-3.0)arc[start angle=-90, end angle=-180,radius=3] ;
\draw[green, dotted, very thick] (0,0)circle(3);
\draw (0,-0.1)--(0,0.1);
\draw (-0.1, 0)--(0.1, 0);
\draw [dashed,blue,  domain=1.04:2, samples=40] 
 plot ({2*\x*(sin(pi*(2-\x) r))}, {2*(\x)*(cos(pi*(2-\x) r))});
 \draw [dashed,blue,  domain=1.04:2, samples=40] 
 plot ({-2*\x*(sin(pi*(2-\x) r))}, {2*(\x)*(cos(pi*(2-\x) r))});
 \draw [dashed,blue,  domain=0:1.04, samples=40] 
 plot ({2*\x*(sin((pi-0.12*\x) r))}, {2*(\x)*(cos((pi-0.12*\x) r))});
  \draw [dashed,blue,  domain=0:1.04, samples=40] 
 plot ({-2*\x*(sin((pi-(0.12*\x)) r))}, {2*(\x)*(cos((pi-(0.12*\x)) r))});
  \draw [dashed,red,  domain=1.04:2, samples=40] 
 plot ({2*\x*(sin((pi*(2-\x)-0.6*\x*(2-\x)/2) r))}, {2*(\x)*(cos((pi*(2-\x)-0.6*\x*(2-\x)/2) r))});
 \draw [dashed,red,  domain=1.04:2, samples=40] 
 plot ({-2*\x*(sin((pi*(2-\x)-0.6*\x*(2-\x)/2) r))}, {2*(\x)*(cos((pi*(2-\x)-0.6*\x*(2-\x)/2) r))});
 \draw [dashed,red,  domain=0:1.04, samples=40] 
 plot ({2*\x*(sin((pi-0.12*\x-0.6*\x/2) r))}, {2*(\x)*(cos((pi-0.12*\x-0.6*\x/2) r))});
  \draw [dashed,red,  domain=0:1.04, samples=40] 
 plot ({-2*\x*(sin((pi-(0.12*\x)-0.6*\x/2) r))}, {2*(\x)*(cos((pi-(0.12*\x)-0.6*\x/2) r))});

\node at (3.3,1) {$W$};
\node at (2.2,-0.8) {$\tilde{W}$};

\node at (0.3,0.6) {$B_{\frac{1}{2}}$};
\node at (0.4,2.6) {$B_{\frac{3}{2}}$};
\node at (3.3,-1.5) {$B_2$};

\draw[->, thick] (4.4,0)--(6,0);
  \node[above] at (5.3, 0) {$f_m$};

    \foreach \x in {9}
    \foreach \y in {0.82}
\draw[very thick, orange]
(\x+0*\y,4+0*\y) to[out=-138,in=50] (\x-1*\y,4-1*\y) to[out=-130,in=90] (\x-2.5*\y,4-5*\y) to[out=-90,in=160]  (\x-1*\y,4-7.8*\y) to[out=-20,in=180] (\x+0*\y,4-8*\y) ;
    \foreach \x in {9}
    \foreach \y in {0.82}
\draw[ very thick, orange]
(\x+0*\y,4+0*\y) to[out=-42,in=-230] (\x+1*\y,4-1*\y) to[out=-50,in=90] (\x+2.5*\y,4-5*\y) to[out=-90,in=20]  (\x+1*\y,4-7.8*\y) to[out=-160,in=0] (\x+0*\y,4-8*\y) ;

    \foreach \x in {9}
    \foreach \y in {0.7}
\draw[very thick, orange, dotted]
(\x+0*\y,4+0*\y) to[out=-138,in=50] (\x-1*\y,4-1*\y) to[out=-130,in=90] (\x-2.5*\y,4-5*\y) to[out=-90,in=160]  (\x-1*\y,4-7.8*\y) to[out=-20,in=180] (\x+0*\y,4-8*\y) ;
    \foreach \x in {9}
    \foreach \y in {0.7}
\draw[ very thick, orange, dotted]
(\x+0*\y,4+0*\y) to[out=-42,in=-230] (\x+1*\y,4-1*\y) to[out=-50,in=90] (\x+2.5*\y,4-5*\y) to[out=-90,in=20]  (\x+1*\y,4-7.8*\y) to[out=-160,in=0] (\x+0*\y,4-8*\y) ;

\draw[draw=white, fill=white] (6.1,1.78) rectangle ++(5,3);

    \foreach \x in {9}
    \foreach \y in {0.9}
\draw[ very thick, green]
(\x+0*\y,4+0*\y) to[out=-138,in=50] (\x-1*\y,4-1*\y) to[out=-130,in=90] (\x-2.5*\y,4-5*\y) to[out=-90,in=160]  (\x-1*\y,4-7.8*\y) to[out=-20,in=180] (\x+0*\y,4-8*\y) ;
    \foreach \x in {9}
    \foreach \y in {0.9}
\draw[ very thick, green]
(\x+0*\y,4+0*\y) to[out=-42,in=-230] (\x+1*\y,4-1*\y) to[out=-50,in=90] (\x+2.5*\y,4-5*\y) to[out=-90,in=20]  (\x+1*\y,4-7.8*\y) to[out=-160,in=0] (\x+0*\y,4-8*\y) ;
 
    \foreach \x in {9}
    \foreach \y in {0.6}
\draw[ very thick, green, dotted]
(\x+0*\y,4+0*\y) to[out=-138,in=50] (\x-1*\y,4-1*\y) to[out=-130,in=90] (\x-2.5*\y,4-5*\y) to[out=-90,in=160]  (\x-1*\y,4-7.8*\y) to[out=-20,in=180] (\x+0*\y,4-8*\y) ;
    \foreach \x in {9}
    \foreach \y in {0.6}
\draw[very thick, green, dotted]
(\x+0*\y,4+0*\y) to[out=-42,in=-230] (\x+1*\y,4-1*\y) to[out=-50,in=90] (\x+2.5*\y,4-5*\y) to[out=-90,in=20]  (\x+1*\y,4-7.8*\y) to[out=-160,in=0] (\x+0*\y,4-8*\y) ;

\draw[draw=white, fill=white] (6.1,2) rectangle ++(5,3);

    \foreach \x in {9}
    \foreach \y in {-1}
\draw[ very thick, green, dotted]
(\x+1.64*\y,2) to[out=42,in=130] (\x+1.45*\y,2.09) to[out=-42,in=130] (\x+1.37*\y,2) ;
    \foreach \x in {9}
    \foreach \y in {1}
\draw[ very thick, green, dotted]
(\x+1.64*\y,2) to[out=138,in=50] (\x+1.45*\y,2.09) to[out=-138,in=50] (\x+1.37*\y,2) ;

    \foreach \x in {9}
    \foreach \y in {-1}
\draw[ very thick, orange, dotted]
(\x+1.65*\y,1.78) to[out=42,in=130] (\x+1.55*\y,1.78) ;
    \foreach \x in {9}
    \foreach \y in {1}
\draw[ very thick, orange, dotted]
(\x+1.65*\y,1.78) to[out=138,in=50] (\x+1.55*\y,1.78) ;

    \foreach \x in {9}
    \foreach \y in {1}
\draw[ very thick]
(\x+0.012*\y,4+0*\y) to[out=-138,in=50] (\x-1*\y,4-1*\y) to[out=-130,in=90] (\x-2.5*\y,4-5*\y) to[out=-90,in=160]  (\x-1*\y,4-7.8*\y) to[out=-20,in=180] (\x+0*\y,4-8*\y) ;
    \foreach \x in {9}
    \foreach \y in { 1}
\draw[ very thick]
(\x-0.012*\y,4+0*\y) to[out=-42,in=-230] (\x+1*\y,4-1*\y) to[out=-50,in=90] (\x+2.5*\y,4-5*\y) to[out=-90,in=20]  (\x+1*\y,4-7.8*\y) to[out=-160,in=0] (\x+0*\y,4-8*\y) ;

\node at (9.8,4.2) {$[0,0,0]$};
\node at (12.3,-1.5) {$f_m(B_2)$};


\end{tikzpicture}
\caption{Mapping $f_{m}$ and its behaviour on spheres of radius $\frac{1}{2}$, $\frac{3}{2}$ and $2$.}\label{fm}
\end{figure}
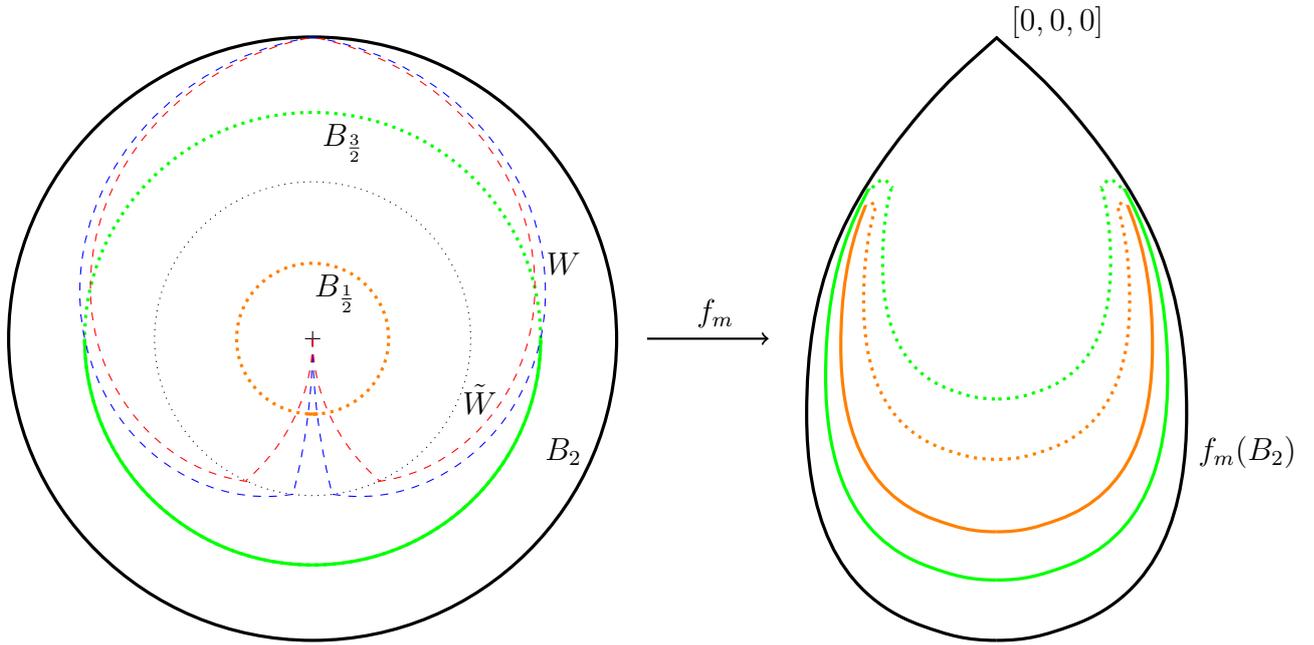

\begin{figure}
\begin{tikzpicture}
\draw[very thick] (0,0) circle(4);
\draw[dotted] (0,0) circle(2);
\draw[orange, dotted, very thick] (0,0)circle(1);
\draw[green, very thick] (0,-3.0)arc[start angle=-90, end angle=0,radius=3] ;
\draw[green, very thick] (0,-3.0)arc[start angle=-90, end angle=-180,radius=3] ;
\draw[green, dotted, very thick] (0,0)circle(3);
\draw (0,-0.1)--(0,0.1);
\draw (-0.1, 0)--(0.1, 0);
\draw [dashed,blue,  domain=1.0:2, samples=40] 
 plot ({2*\x*(sin(pi*(2-\x) r))}, {2*(\x)*(cos(pi*(2-\x) r))});
 \draw [dashed,blue,  domain=1.0:2, samples=40] 
 plot ({-2*\x*(sin(pi*(2-\x) r))}, {2*(\x)*(cos(pi*(2-\x) r))});
 \draw [dashed,blue,  domain=0:1.0, samples=40] 
 plot ({2*\x*(sin((pi) r))}, {2*(\x)*(cos((pi) r))});

\node at (3.3,1) {$W$};

\node at (0.3,0.6) {$B_{\frac{1}{2}}$};
\node at (0.4,2.6) {$B_{\frac{3}{2}}$};
\node at (3.3,-1.5) {$B_2$};

\draw[->, thick] (4.4,0)--(6,0);
  \node[above] at (5.3, 0) {$f$};

    \foreach \x in {9}
    \foreach \y in {0.9}
\draw[very thick, green]
(\x+0*\y,4+0*\y) to[out=-138,in=50] (\x-1*\y,4-1*\y) to[out=-130,in=90] (\x-2.5*\y,4-5*\y) to[out=-90,in=160]  (\x-1*\y,4-7.8*\y) to[out=-20,in=180] (\x+0*\y,4-8*\y) ;
    \foreach \x in {9}
    \foreach \y in {0.9}
\draw[ very thick, green]
(\x+0*\y,4+0*\y) to[out=-42,in=-230] (\x+1*\y,4-1*\y) to[out=-50,in=90] (\x+2.5*\y,4-5*\y) to[out=-90,in=20]  (\x+1*\y,4-7.8*\y) to[out=-160,in=0] (\x+0*\y,4-8*\y) ;

    \foreach \x in {9}
    \foreach \y in {0.7}
\draw[ very thick, orange, dotted]
(\x+0*\y,4+0*\y) to[out=-138,in=50] (\x-1*\y,4-1*\y) to[out=-130,in=90] (\x-2.5*\y,4-5*\y) to[out=-90,in=160]  (\x-1*\y,4-7.8*\y) to[out=-20,in=180] (\x+0*\y,4-8*\y) ;
    \foreach \x in {9}
    \foreach \y in {0.7}
\draw[very thick, orange, dotted]
(\x+0*\y,4+0*\y) to[out=-42,in=-230] (\x+1*\y,4-1*\y) to[out=-50,in=90] (\x+2.5*\y,4-5*\y) to[out=-90,in=20]  (\x+1*\y,4-7.8*\y) to[out=-160,in=0] (\x+0*\y,4-8*\y) ;

    \foreach \x in {9}
    \foreach \y in {0.6}
\draw[ very thick, green, dotted]
(\x+0*\y,4+0*\y) to[out=-138,in=50] (\x-1*\y,4-1*\y) to[out=-130,in=90] (\x-2.5*\y,4-5*\y) to[out=-90,in=160]  (\x-1*\y,4-7.8*\y) to[out=-20,in=180] (\x+0*\y,4-8*\y) ;
    \foreach \x in {9}
    \foreach \y in {0.6}
\draw[very thick, green, dotted]
(\x+0*\y,4+0*\y) to[out=-42,in=-230] (\x+1*\y,4-1*\y) to[out=-50,in=90] (\x+2.5*\y,4-5*\y) to[out=-90,in=20]  (\x+1*\y,4-7.8*\y) to[out=-160,in=0] (\x+0*\y,4-8*\y) ;

    \foreach \x in {9}
    \foreach \y in {1}
\draw[ very thick]
(\x+0.012*\y,4+0*\y) to[out=-138,in=50] (\x-1*\y,4-1*\y) to[out=-130,in=90] (\x-2.5*\y,4-5*\y) to[out=-90,in=160]  (\x-1*\y,4-7.8*\y) to[out=-20,in=180] (\x+0*\y,4-8*\y) ;
    \foreach \x in {9}
    \foreach \y in { 1}
\draw[ very thick]
(\x-0.012*\y,4+0*\y) to[out=-42,in=-230] (\x+1*\y,4-1*\y) to[out=-50,in=90] (\x+2.5*\y,4-5*\y) to[out=-90,in=20]  (\x+1*\y,4-7.8*\y) to[out=-160,in=0] (\x+0*\y,4-8*\y) ;

\node at (9,-1.2) {$f_T(B_{\frac{1}{2}})$};
\node at (12.3,-1.5) {$f(B_2)$};


\end{tikzpicture} 
\caption{Limit mapping $f$ and its behaviour on spheres of radius $\frac{1}{2}$, $\frac{3}{2}$ and $2$.}\label{flimit}
\end{figure}
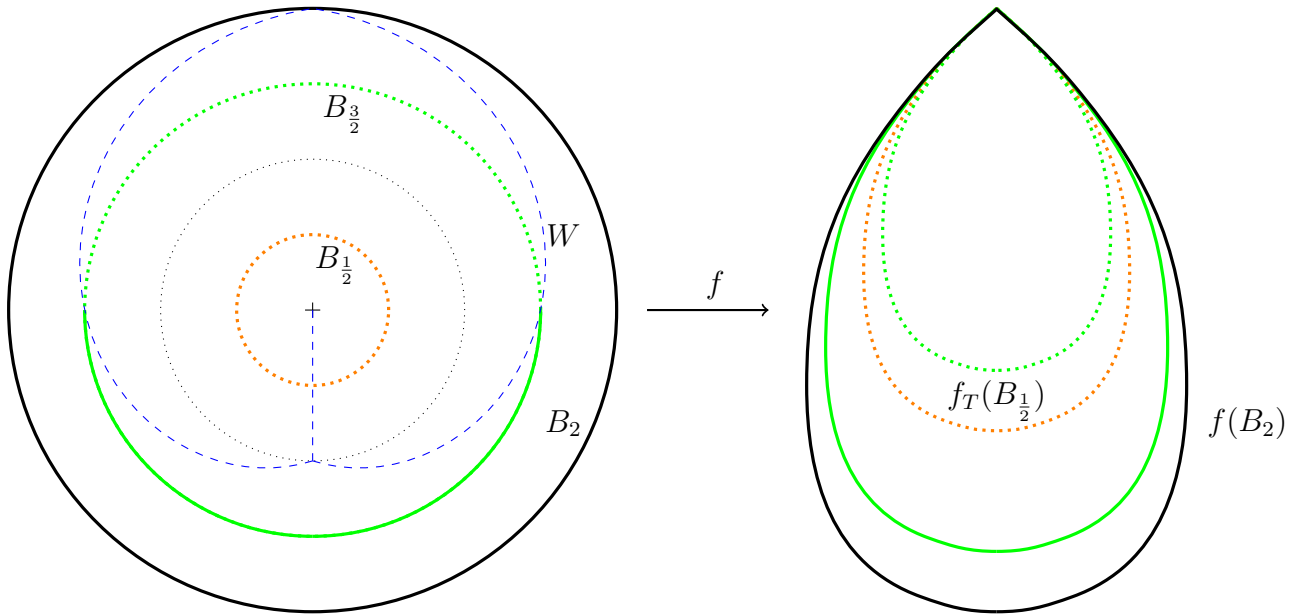

The behaviour of limit mapping is depicted in Figure \ref{flimit}. We will show that $f_m$ forms a bounded sequence in $W^{1,2}$ (and also that $\int \frac{1}{J_f^a}$ is bounded) so there is a subsequence which converges weakly to the pointwise limit $f$. 
All "horseshoes" $f(B(0,r))$ have two tips that go up to the point $[0,0,0]$. Let us describe in details the behaviour of $f$ on $\overline{B(0,\frac{1}{2})}$ and why the limit fails \textcolor{black}{to satisfy} the (INV) condition there. The boundary $\partial B(0,\frac{1}{2})$ has only inner part $I_{\frac{1}{2}}$ and there is no outer part so the image $f(\partial B(0,\frac{1}{2}))$ consists only from the dotted orange drop on the right-hand side of the picture. It follows that $f_T(B(0,\frac{1}{2}))$ is equal to the inner part of this (rotated) drop and it is not difficult to check that the degree actually equals $-1$ there as  we have changed the orientation of the sphere. However for $x\in B(0,\frac{1}{2})$ we know that $f(x)$ does not belong to to $f_T(B(0,\frac{1}{2}))$ as it is mapped outside of this drop, in fact for $f_m$ we had the outer drop $f_m(O_r)$ and $f_m(B(0,r))$ lies between $f_m(O_r)$ and $f_m(I_r)$ so in the limit outside of $f(I_r)$ (which is the limit of $f_m(I_r)$). 

{\bf Step 2. Formal definitions:} 
We first define the set $W$ between the inner and outer parts of $\partial B(0,r)$ and then we divide $B(0,2)$ into different regions accordingly. 
We set $r_\varepsilon^1=1+\frac{\varepsilon}{\pi-\varepsilon}$ and
$$
S_\varepsilon=\begin{cases}
\pi-\varepsilon r, & 0<r<r_\varepsilon^1,\\
(2-r)\pi,& r_\varepsilon^1<r<2\\
\end{cases}
$$
and our $W$ is defined as (see the blue curve in Fig. \ref{areas})
$$
W:=\Bigl\{(r,\alpha): \alpha=S_{\varepsilon}\Bigr\}. 
$$
This formula 
 corresponds to the blue curve on the right half of Fig. \ref{areas} while the blue curve on the left side is created by rotation around the $z$-axes. Note that $r_\varepsilon^1\to 1$ as $\varepsilon\to 0$.
\begin{figure}[h t p]
\begin{tikzpicture}
\draw (0,0) circle(4);
\draw (0,0) circle(2.08);
\draw[dotted] (0,0) circle(0.8);
\draw (0,-0.1)--(0,0.1);
\draw (-0.1, 0)--(0.1, 0);
\draw [blue,  domain=1.04:2, samples=40] 
 plot ({2*\x*(sin(pi*(2-\x) r))}, {2*(\x)*(cos(pi*(2-\x) r))});
 \draw [blue,  domain=1.04:2, samples=40] 
 plot ({-2*\x*(sin(pi*(2-\x) r))}, {2*(\x)*(cos(pi*(2-\x) r))});
 \draw [blue,  domain=0:1.04, samples=40] 
 plot ({2*\x*(sin((pi-0.12*\x) r))}, {2*(\x)*(cos((pi-0.12*\x) r))});
  \draw [blue,  domain=0:1.04, samples=40] 
 plot ({-2*\x*(sin((pi-(0.12*\x)) r))}, {2*(\x)*(cos((pi-(0.12*\x)) r))});
 
 \draw [red,  domain=1.04:2, samples=40] 
 plot ({2*\x*(sin((pi*(2-\x)-0.6*\x*(2-\x)/2) r))}, {2*(\x)*(cos((pi*(2-\x)-0.6*\x*(2-\x)/2) r))});
 \draw [red,  domain=1.04:2, samples=40] 
 plot ({-2*\x*(sin((pi*(2-\x)-0.6*\x*(2-\x)/2) r))}, {2*(\x)*(cos((pi*(2-\x)-0.6*\x*(2-\x)/2) r))});
 \draw [red,  domain=0:1.04, samples=40] 
 plot ({2*\x*(sin((pi-0.12*\x-0.6*\x/2) r))}, {2*(\x)*(cos((pi-0.12*\x-0.6*\x/2) r))});
  \draw [red,  domain=0:1.04, samples=40] 
 plot ({-2*\x*(sin((pi-(0.12*\x)-0.6*\x/2) r))}, {2*(\x)*(cos((pi-(0.12*\x)-0.6*\x/2) r))});
 
\node at (1,1) {$A_1$};
\node at (3.3,1) {$W$};
\node at (1,3) {$A_2$};
\node at (1,-3) {$D_2$};
\draw (0.3,-1.7)--(4,3.8);
\draw (2.55,-0.9)--(5,3);
\draw (0,-1.7)--(-5,3);
\node at (4.2,3.8) {$B$};
\node at (5.25,3) {$C$};
\node at (-5.3,3) {$D_1$};
\end{tikzpicture}
\caption{Definition of $W$ and different areas.}\label{areas}
\end{figure}
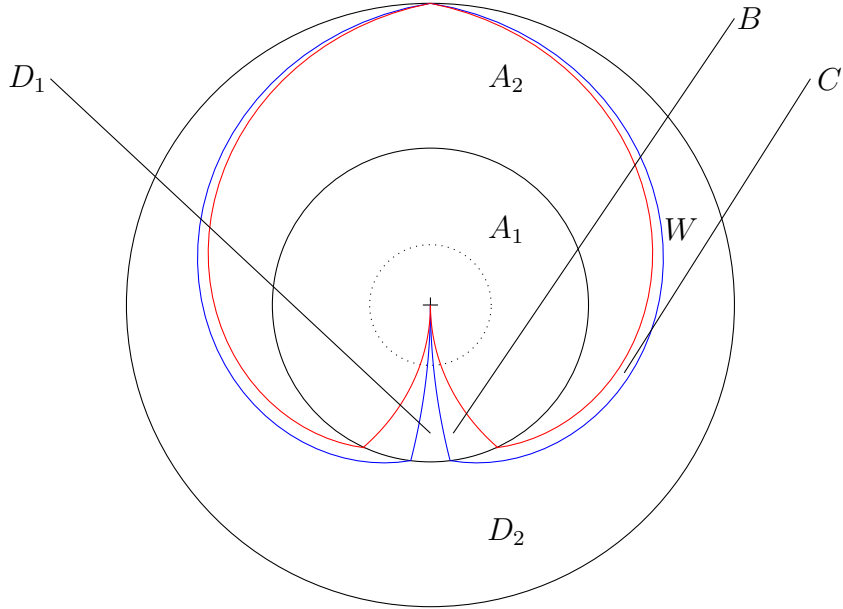
Given $a<2$ we fix $\textcolor{black}{p\in \left(\frac{1}{2}, 1\right)}$ such that
\eqn{choosep}
$$
a(1-3p)>-1.
$$
Now we define the "thickness between the blue curve and the red curve" as  
$$
\delta(\varepsilon, r)=
\begin{cases}
\varepsilon^{1/p} r & 0<r<r_\varepsilon^1,\\
c_0\varepsilon^{1/p}(2-r)^\lambda &r_\varepsilon^1<r<2,\\
\end{cases}
$$
where 
\eqn{chooselambda}
$$
\lambda =\frac{2}{1+a-3ap} \geq \frac{2}{1+a-3a/2}=\frac{2}{1-a/2}> 2 \text{ and }c_0=\frac{\pi}{\pi-\varepsilon}\cdot\frac{(\pi-\varepsilon)^\lambda}{(\pi-2\varepsilon)^\lambda}\approx 1,
$$
so that $\delta$ is continuous at $r_{\epsilon}^1$.  
Finally, we define the red curve on the picture as 
$$
\tilde{W}:=\Bigl\{(r,\alpha): \alpha=\tilde{S}_{\varepsilon}\Bigr\}. 
\text{ where }
\tilde{S}_\varepsilon=S_\varepsilon-\delta(\varepsilon, r). 
$$
Now we can define the regions in Fig. \ref{areas} as
$$
\begin{aligned}
A_1&:=\left\{(r,\alpha):\ r\in\left(0,r_\varepsilon^1\right),\ \alpha<\tilde{S}_{\varepsilon}\right\},\\
B&:=\left\{(r,\alpha):\ r\in\left(0,r_\varepsilon^1\right),\ \tilde{S}_{\varepsilon}<\alpha<S_{\varepsilon}\right\},\\
D_1&:=\left\{(r,\alpha):\ r\in\left(0,r_\varepsilon^1\right),\ S_{\varepsilon}<\alpha\right\},\\
A_2&:=\left\{(r,\alpha):\ r\in\left(r_\varepsilon^1,2\right),\ \alpha<\tilde{S}_{\varepsilon}\right\},\\
C&:=\left\{(r,\alpha):\ r\in\left(r_\varepsilon^1,2\right),\ \tilde{S}_{\varepsilon}<\alpha<S_{\varepsilon}\right\}\text{ and }\\
D_2&:=\left\{(r,\alpha):\ r\in\left(r_\varepsilon^1,2\right),\ S_{\varepsilon}<\alpha\right\}.\\
\end{aligned}
$$
Note that we always define only the part of the region in the right part of Fig. \ref{areas} and the corresponding left-part is created by rotation around the $z$-axes (or mirroring). 

Our mapping $f_{\varepsilon}:(r, \alpha, \beta)\mapsto (\tilde{r}, \tilde{\alpha}, \tilde{\beta})$ is defined as 
\eqn{defzobr}
$$
\begin{aligned}
    \tilde{r}&=R_{\varepsilon}(r, \alpha)\cos(T_{\varepsilon}(r, \alpha))\\
    \tilde{\alpha}&=R_{\varepsilon}(r, \alpha)T_{\varepsilon}(r, \alpha)\\
    \tilde{\beta}&=\beta,
\end{aligned}
$$
where we define $R_{\varepsilon}$ and $T_{\varepsilon}$ below. 
\textcolor{black}{Informally speaking, we deform a sphere into a horseshoe with inner and outer part. Were those half-circles, it would be natural to parametrize them in polar coordinates. However, as we work with half-drops, we use another way. Our $R_{\varepsilon}\in [0,1]$ could be viewed as some "radius of the drop in the image" and $T_{\varepsilon}\in[0,\frac{\pi}{2}]$ corresponds to some "angle or parametrization of the boundary of the drop", but instead of using $[R_{\varepsilon}\cos T_{\varepsilon},R_{\varepsilon}\sin T_{\varepsilon}]$ as in the case of polar coordinates we use $[R_{\varepsilon}\cos T_{\varepsilon},R_{\varepsilon}T_{\varepsilon}]$ as it fits us better.}
We want to keep our formulas as simple as possible: we define the functions piecewise on regions $A_1, A_2, B, C, D_1, D_2$. We keep $R_\varepsilon$ to be the same as our limit mapping on $A_1$ and $D_2$ and very close on $A_2$ and $D_1$. We use $B$ and $C$ to continuously connect the values on these regions (by a linear convex combination). 

We define our $R_{\varepsilon}\in[0,1]$ as 
$$
R_\varepsilon=
\begin{cases}
\frac{2-r}3, & \text{ on }A_1,\\
\sqrt{\frac{\pi-2\varepsilon}{\pi-\varepsilon}}\cdot\frac{\sqrt{2-r}}{3}, &\text{ on } A_2,\\
\frac{2}{3}+\frac{\varepsilon r}{3\pi}, &\text{ on } D_1,\\
\frac{1+r}{3}, & \text{ on }D_2,\\
\frac{2-r}{3}\cdot \frac{S-\alpha}{\delta(\varepsilon, r)}+\left(\frac{2}{3}+\frac{\varepsilon r}{3\pi}\right)\left(1-\frac{S-\alpha}{\delta(\varepsilon, r)}\right), & \text{ on } B,\\
\sqrt{\frac{\pi-2\varepsilon}{\pi-\varepsilon}}\cdot\frac{\sqrt{2-r}}{3}\cdot \frac{S-\alpha}{\delta(\varepsilon, r)} +\frac{1+r}{3}\left(1-\frac{S-\alpha}{\delta(\varepsilon, r)}\right), &\text{ on } C,\\
\end{cases}
$$
Note that $R_{\epsilon}$ is continuous and the values on boundaries between regions (like for $r=r_{\epsilon}^1$) agree. 
To define $T_{\varepsilon}$ we need two additional auxiliary functions. 
The first one $\xi_\varepsilon\in[0,1]$ measures how close we are to the critical strip between the blue line $W$ and the red line $\tilde{W}$ and is equal to $0$ exactly on the strip:
$$
\xi_\varepsilon  (r,\alpha)=
\begin{cases}
1-\frac{\alpha}{\tilde{S}_\varepsilon},& \text{ on }A_1\cup A_2 \text{ (i.e., on $\alpha<\tilde{S}_\varepsilon$)},\\
0,& \text{ on } B\cup C \text{ (i.e., on $\tilde{S}_\varepsilon\leq\alpha\leq S_\varepsilon$)},\\
1-\frac{\pi-\alpha}{\pi-S_\varepsilon},&\text{ on } D_1\cup D_2 \text{ (i.e., on $S_\varepsilon<\alpha$)}.
\end{cases}
$$
 
Let us define $r_\varepsilon^0=\frac{\varepsilon-2\varepsilon^2}{1-\varepsilon^2}\approx \varepsilon$ so that functions $\frac{r}{\varepsilon}$ and $1-(2-r)\varepsilon$ are equal at this point. 
The second one (recall that $p\in(\frac{1}{2},1)$ and $\lambda>2$ were chosen in \eqref{choosep} and \eqref{chooselambda}) 
$$
\psi(\varepsilon, r)=\begin{cases}
\frac{r}{\varepsilon}&\text{ for }r\in [0,r_\varepsilon^0],\\
1-\varepsilon(2-r) &\text{ for }r\in [r_\varepsilon^0,r_\varepsilon^1],\\
1-\left(\frac{\pi-2\varepsilon}{\pi-\varepsilon}\right)^{1-\lambda p}\varepsilon (2-r)^{\lambda p} &\text{ for }r\in [r_\varepsilon^1,2],\\
\end{cases}
$$
is influencing the shape of the ``horseshoes'' (see Fig. \ref{fm}). 
For $\psi(\varepsilon, r)=1$ the horseshoe is coming up to the point $[0,0,0]$ so we want $\lim_{\varepsilon\to 0+}\psi(\varepsilon, r)=1$, but $\psi(\varepsilon, r)<1$ (to have injectivity). Moreover, the definition $\psi(\varepsilon, 0)=0$ and  $\psi(\varepsilon, r)$ small for $r$ small ensures that for really small $r$ our horseshoes are small so that $f_{\varepsilon}$ is continuous at the origin. 

We set
$$
T_\varepsilon=\frac\pi2(1-\xi_\varepsilon^{p})\psi(\varepsilon, r)\in\Bigl[0,\frac{\pi}{2}\Bigr]. 
$$
Note that for $\xi_\varepsilon$ close to $0$ (i.e. close to blue-red strip) and for $\psi(\varepsilon, r)$ close to $1$ we have $T_\varepsilon$ close to $\frac{\pi}{2}$ and thus by \eqref{defzobr} we obtain that $\tilde{r}$ is close to $0$, i.e. the image of our point is close to $[0,0,0]$. Note that our $f_\varepsilon$ is continuous up to the boundary of  $\overline{B(0,2)}$. For simplicity of notation we sometimes omit the subscript $\epsilon$ and we write only $R$, $\xi$ and $T$ and not $R_{\epsilon}$, $\xi_\varepsilon$ and $T_{\epsilon}$. 


{\bf Step 3. Continuity and injectivity:} It is easy to check that our $f_{\varepsilon}$ is continuous on all regions. 
Moreover, it is not difficult to check that on boundaries between two regions the values are the same from both sides and hence our $f_{\varepsilon}$ is continuous. 

It is also not difficult to check that $f_{\varepsilon}$ restricted to each boundary $\partial A_1$, $\partial A_2$, $\partial D_1$, $\partial D_2$, $\partial B$ and $\partial C$ is a homeomorphism. For that purpose we will extend $R_\varepsilon$ and $T_\varepsilon$ on $S(0,2)\cup \{[0,0,0]\}$:
\begin{align*}
R_\varepsilon &= 1, \psi(\varepsilon, r)=1, \xi_\varepsilon(r, \alpha)=\frac{\alpha}{\pi} \text{ on } S(0,2)\setminus \{[0,0,2]\}\\
R_\varepsilon &= 0, T_\varepsilon = 0 \text{ on } \{[0,0,2]\}\\
R_\varepsilon &= \frac{2}{3}, T_\varepsilon = 0 \text{ on } \{[0,0,0]\}.
\end{align*}
There are two points where the extension of $\xi$ is not defined, points $[0,0,0]$ and $[0,0,2]$. Apart \textcolor{black}{from} them we have continuous functions.

Let us now prove the injectivity on the boundaries, firstly in the planar setting. Assume we have $(R^1, T^1)=(R^2, T^2)$, we want to show that the preimages are the same.

{\underline{1. $R^1=0$ or $R^1=2/3$:}} The only possible preimages in those cases are the points $[0,0,2]$ or $[0,0,0]$, respectively.

{\underline{2. $R\in (0,1]\setminus\{2/3\}$:}} In this case we have $R$ and $T$ determined by the previously used formulas. Also we can uniquely describe the preimage by its polar coordinates $(r,\alpha)\in (0,2]\times [0,\pi]$.
\begin{itemize}
\item $\partial A_1, \partial A_2, \partial D_1, \partial D_2$: Since $R$ is independent of $\alpha$ and injective with respect to $r$, $R^1=R^2$ implies $r^1=r^2$. That gives $\psi^1=\psi^2\neq 0$. From that and $T^1=T^2$ we have $\xi^1=\xi^2$. Again, for fixed $r$ on each of those domains we have that $\xi$ is injective with respect to $\alpha$. Together this gives $\alpha^1=\alpha^2$.
\item $\partial B, \partial C$: Here we have that $T$ is independent of $\alpha$ and injective, as $\psi$ is injective with respect to $r$. So we know that $r^1=r^2$. Since $R$ for fixed $r$ is a convex combination of two distinct numbers, it is therefore injective with respect to $\alpha$ and we are done.
\end{itemize}
Now we address the mapping $(R,T)\mapsto (\tilde{r}, \tilde{\alpha})=(R\cos T, RT)$. We claim that it is injective for $(R,T)\in ((0,1]\times[0,\pi/2])\cup \{(0,0)\}$. Let us have $(R^1, T^1)$, $(R^2, T^2)$ such that $(\tilde{r}^1, \tilde{\alpha}^1)=(\tilde{r}^2, \tilde{\alpha}^2)$. If $(\tilde{r}^1, \tilde{\alpha}^1)=(0,0)$, we know that $R^1=R^2=0$, $T^1=T^2=0$ and the result follows. If $\tilde{\alpha}^1=0$ and $\tilde{r}^1$ is positive, it follows that $T^1=T^2=0$ and $R^1=R^2=\tilde{r}^1$. Otherwise since $\cos T$ is decreasing and $T$ is increasing, we have that $\cos T/T: (0,\pi/2]\to [0, \infty)$ is strictly monotone. Since 
$$
\frac{\cos(T^1)}{T^1}=\frac{\tilde{r}^1}{\tilde{\alpha}^1}=\frac{\tilde{r}^2}{\tilde{\alpha}^2}=\frac{\cos(T^2)}{T^2},
$$ 
we obtain $T^1=T^2$, and so $R^1=R^2$.


When we add the third dimension and rotate, the injectivity does not change and hence our $f_{\epsilon}$ is a homeomorphism on boundaries of different regions. 
Below we estimate the integrability of $J_{f_{\varepsilon}}^{-a}$ and in those estimates we show (as a by-product) that $J_{f_{\varepsilon}}\neq 0$ in all the regions. By Inverse Mapping Theorem it follows that $f_{\varepsilon}$ is locally a homeomorphism and since it is a homeomorphism on the boundaries we obtain that it is a homeomorphism in each of the regions (see e.g. \cite{Kr}). Moreover, it is a homeomorphism on $\partial B(0,2)$ and thus a homeomorphism on $\overline{B(0,2)}$. 






{\bf Step 4. Integrability of $|Df\textcolor{black}{_\varepsilon}|^2$ and $J_{f\textcolor{black}{_\varepsilon}}^{-a}$ on $A_1\cup A_2\cup D_1\cup D_2$:} 

{\underline{Estimate from spherical to spherical coordinates:} }
For mappings from spherical to spherical coordinates that are rotationally symmetric with respect to $\beta$, we have
\eqn{derivace}
$$
\begin{aligned}
&\int_{B(0,2)} \|Df\textcolor{black}{_\varepsilon}\|^2 = 2\pi \int_0^2 \int_0^{\pi} \Bigl[(\partial_r \tilde{r})^2 + (\tilde{r} \partial_r \tilde{\alpha})^2 + \left(\frac{\partial_\alpha \tilde{r}}{r}\right)^2+\left(\frac{\tilde{r} \partial_\alpha \tilde{\alpha}}{r}\right)^2 + \left(\frac{\tilde{r} \sin(\tilde{\alpha})}{r \sin\alpha}\right)^2 \Bigr]\cdot\\
&\phantom{aaaaaaaaaaaaaaaaaaaaaaaaaaaaaaaaaaaaaaaaaaaaaaaaaaaaaaaaaaaaa}\cdot r^2 \sin\alpha \textcolor{black}{\; d\alpha}\;dr  \\
&\approx \int_0^2 \int_0^{\pi} \Bigl{[}r^2 \alpha (\pi-\alpha) \left[(\partial_r (R\cos T))^2 + (R \cos T \partial_r (RT))^2\right]\\
& \phantom{aaa}+ \alpha (\pi-\alpha)\left[(\partial_\alpha (R\cos T))^2 + (R \cos T \partial_\alpha (RT))^2\right]  +  \frac{(R\cos T \sin(RT))^2}{\alpha(\pi-\alpha)}\Bigr{]}  \textcolor{black}{\; d\alpha}\;dr\\
\end{aligned}
$$
and
\eqn{jakobian}
$$
\begin{aligned}
&\int_{B(0,2)} \left|J_{f\textcolor{black}{_\varepsilon}}\right|^{-a}=2\pi \int_0^2 \int_0^{\pi}\left| \partial_r \tilde{r}\cdot \partial_\alpha \tilde{\alpha}-\partial_r \tilde{\alpha}\cdot \partial_\alpha \tilde{r}\right|^{-a}|\tilde{r}^2 \sin(\tilde{\alpha})|^{-a}|r^2 \sin\alpha|^{1+a}\textcolor{black}{\; d\alpha}\; dr\\
&=2\pi \int_0^2 \int_0^{\pi}\left| \partial_r R\cdot \partial_\alpha T-\partial_r T\cdot \partial_\alpha R\right|^{-a} R^{-a} |\cos T + T\sin T|^{-a}\cdot\\
& \phantom{aaaaaaaaaaaaaaaaaaaaaaaaaaaaaaa}\cdot|R^2 ( \cos T)^2 \sin(RT)|^{-a}|r^2 \sin\alpha|^{1+a}\textcolor{black}{\; d\alpha}\;dr\\
&\approx \int_0^2 \int_0^{\pi}\left| \partial_r R\cdot \partial_\alpha T-\partial_r T\cdot \partial_\alpha R\right|^{-a} R^{-3a} |( \cos T)^2 \sin(RT)|^{-a}|r^2 \sin\alpha|^{1+a}\textcolor{black}{\; d\alpha}\;dr.\\
\end{aligned}
$$
Note that the term $|\cos T + T\sin T|$ is bounded both from below and above for $T\in [0,\frac{\pi}{2}]$ so we can estimate it by a constant.

{\underline{Estimate on $A_1\cap \{r>r_\varepsilon^0\}$:} } On this set we have $0<\alpha<\pi-\varepsilon r-\varepsilon^{\frac{1}{p}}r$,
$$
R_{\varepsilon}=\frac{2-r}{3}\text{ and }T_{\varepsilon}=\frac{\pi}{2}\left(1-\left(1-\frac{\alpha}{\pi-\varepsilon r-\varepsilon^{\frac{1}{p}}r}\right)^p\right)(1-(2-r)\varepsilon). 
$$
Let us first estimate 
\eqn{dertalpha}
$$
|\partial_{\alpha} T_{\varepsilon}|=\frac{\pi}{2}p\left(1-\frac{\alpha}{\pi-\varepsilon r-\varepsilon^{\frac{1}{p}}r}\right)^{p-1}\frac{1}{\pi-\varepsilon r-\varepsilon^{\frac{1}{p}}r}(1-(2-r)\varepsilon)
\approx (\pi-\varepsilon r-\varepsilon^{\frac{1}{p}}r-\alpha)^{p-1} 
$$
and 
$$
\begin{aligned}
|\partial_r T_{\varepsilon}|&=\frac{\pi}{2}\left(1-\left(1-\frac{\alpha}{\pi-\varepsilon r-\varepsilon^{\frac{1}{p}}r}\right)^{p}\right)\varepsilon+\\
\phantom{=}&+
\frac{\pi}{2}p\left(1-\frac{\alpha}{\pi-\varepsilon r-\varepsilon^{\frac{1}{p}}r}\right)^{p-1}\frac{\alpha}{(\pi-\varepsilon r-\varepsilon^{\frac{1}{p}}r)^2}(\varepsilon+\varepsilon^{1/p})(1-(2-r)\varepsilon)\\
&\lesssim \varepsilon (\pi-\varepsilon r-\varepsilon^{\frac{1}{p}}r-\alpha)^{p-1}.
\end{aligned}
$$
Using $\cos(\frac{\pi}{2}-y)\approx y$ we get
\eqn{cost}
$$
\begin{aligned}
\cos T &=\cos\left[ \frac{\pi}{2}\left(1-\left(1-\frac{\alpha}{\pi-\varepsilon r-\varepsilon^{\frac{1}{p}}r}\right)^p\right)(1-(2-r)\varepsilon)\right]\\
&\approx \left(1-\frac{\alpha}{\pi-\varepsilon r-\varepsilon^{\frac{1}{p}}r}\right)^p+(2-r)\varepsilon. \\
\end{aligned}
$$
Now we use $R\approx 1$, \textcolor{black}{the} previous line, \eqref{gonio} and $RT\leq \pi/2$, \eqref{elementary} and $p>\frac{1}{2}$ to estimate
$$
\begin{aligned}
\frac{(R\cos T \sin(RT))^2}{\alpha(\pi-\alpha)}
&\approx \frac{\Bigl(\Bigl(1-\frac{\alpha}{\pi-\varepsilon r-\varepsilon^{\frac{1}{p}}r}\Bigr)^p+(2-r)\varepsilon\Bigr)^2  T^2}{\alpha(\pi-\alpha)}\\
&\lesssim\frac{\Bigl(\Bigl(1-\frac{\alpha}{\pi-\varepsilon r-\varepsilon^{\frac{1}{p}}r}\Bigr)^{2p}+\varepsilon^2\Bigr)}{(\pi-\alpha)}
\frac{\Bigl(1-\Bigl(1-\frac{\alpha}{\pi-\varepsilon r-\varepsilon^{\frac{1}{p}}r}\Bigr)^{p}\Bigr)^2}{\alpha}\\
&\lesssim\frac{(\pi-\alpha-\varepsilon r-\varepsilon^{\frac{1}{p}}r)^{2p}+\varepsilon^2}{\pi-\alpha}\frac{\alpha^2}{\alpha}
\leq 1+\frac{\varepsilon^2}{\pi-\alpha}.\\
\end{aligned}
$$
With the help of these estimates, \textcolor{black}{using \eqref{derivace} and $p>\frac{1}{2}$ we get} that 
$$
\begin{aligned}
\int_{A_1\cap\{r>r_\varepsilon^0\}} \|Df\textcolor{black}{_\varepsilon}\|^2
&\lesssim \int_{r_\varepsilon^0}^{r_\varepsilon^1} \int_0^{\pi-\varepsilon r-\varepsilon^{\frac{1}{p}}r} 
\Bigl[\alpha (\pi-\alpha) \left[|\partial_r R|^2+|\partial_r T|^2 + |\partial_{\alpha} T|^2\right]+\\
&\phantom{aaaaaaaaaaaaaaaaaaaaaaaaaaaaaaa} + \frac{(R\cos T \sin(RT))^2}{\alpha(\pi-\alpha)}\Bigr] \textcolor{black}{\; d\alpha} \; dr\\
&\lesssim\int_{0}^{r_\varepsilon^1} \int_0^{\pi-\varepsilon r-\varepsilon^{\frac{1}{p}}r} \left(1+(\pi-\varepsilon r-\varepsilon^{\frac{1}{p}}r-\alpha)^{2p-2}
+\frac{\varepsilon^2}{(\pi-\alpha)}\right) \; d\alpha\; dr  \\
&\lesssim 1+\varepsilon^2\int_0^{r_\varepsilon^1}-\log(\varepsilon r)\; dr\lesssim 1. \\
\end{aligned}
$$
It remains to estimate the Jacobian on $A_1$ using \eqref{jakobian}, $\partial_{\alpha}R=0$ and $R\approx 1$
$$
\int_{A_1\cap\{r>r_\varepsilon^0\}}\left|J_{f\textcolor{black}{_\varepsilon}}\right|^{-a}
\lesssim \int_{r_\varepsilon^0}^{r_\varepsilon^1} \int_0^{\pi-\varepsilon r-\varepsilon^{\frac{1}{p}}r}\left|\partial_{\alpha}T\right|^{-a}\left|(\cos T)^2\sin(RT)\right|^{-a}\left|r^2\sin \alpha\right|^{1+a}\; d\alpha\; dr. 
$$ 
We estimate using \eqref{elementary} and $R\approx 1$
$$
\left|\frac{\sin \alpha}{\sin RT}\right|^a\lesssim \left|\frac{\alpha}{1-\bigl(1-\frac{\alpha}{\pi-\varepsilon r-\varepsilon^{\frac{1}{p}}r}\bigr)^p}\right|^a\lesssim 1. 
$$
Further using \eqref{cost} we obtain
$$
\frac{1}{\cos T}\lesssim \left(1-\frac{\alpha}{\pi-\varepsilon r-\varepsilon^{\frac{1}{p}}r}\right)^{-p}\approx (\pi-\varepsilon r-\varepsilon^{\frac{1}{p}}r-\alpha)^{-p}. 
$$
Together with \eqref{dertalpha} these estimates give us
$$
\begin{aligned}
\int_{A_1\cap\{r>r_\varepsilon^0\}}\left|J_{f\textcolor{black}{_\varepsilon}}\right|^{-a}
&\lesssim \int_{r_\varepsilon^0}^{r_\varepsilon^1} \int_0^{\pi-\varepsilon r-\varepsilon^{\frac{1}{p}}r}|\partial_{\alpha}T|^{-a}\bigl|\cos T\bigr|^{-2a}\; d\alpha\; dr\\
&\lesssim \int_0^{r_\varepsilon^1} \int_0^{\pi-\varepsilon r-\varepsilon^{\frac{1}{p}}r}|(\pi-\varepsilon r-\varepsilon^{\frac{1}{p}}r-\alpha)^{p-1}|^{-a}\bigl|
(\pi-\varepsilon r-\varepsilon^{\frac{1}{p}}r-\alpha)^{p}\bigr|^{-2a}\; d\alpha\; dr
\end{aligned}
$$ 
and our choice of $p$ in \eqref{choosep} implies that this integral is finite. 

{\underline{Estimate on $A_1\cap \{r<r_\varepsilon^0\}$:} }
On this set we have $0<\alpha<\pi-\varepsilon r-\varepsilon^{\frac{1}{p}}r$, $0<r<r_\varepsilon^0\approx \varepsilon$ and 
$$
R_{\varepsilon}=\frac{2-r}{3}\text{ and }T_{\varepsilon}=\frac{\pi}{2}\left(1-\left(1-\frac{\alpha}{\pi-\varepsilon r-\varepsilon^{\frac{1}{p}}r}\right)^p\right)\frac{r}{\varepsilon}. 
$$
Again we first estimate
$$
|\partial_{\alpha} T_{\varepsilon}|=\frac{\pi}{2}p\left(1-\frac{\alpha}{\pi-\varepsilon r-\varepsilon^{\frac{1}{p}}r}\right)^{p-1}\frac{1}{\pi-\varepsilon r-\varepsilon^{\frac{1}{p}}r}\cdot\frac{r}{\varepsilon}
\approx \frac{r}{\varepsilon}(\pi-\varepsilon r-\varepsilon^{\frac{1}{p}}r-\alpha)^{p-1} 
$$
and using \eqref{elementary}
\eqn{qqqq}
$$
\begin{aligned}
|\partial_r T_{\varepsilon}|&=\frac{\pi}{2}\left(1-\left(1-\frac{\alpha}{\pi-\varepsilon r-\varepsilon^{\frac{1}{p}}r}\right)^{p}\right)\frac{1}{\varepsilon}+\\
&\phantom{aaaaaa}+\frac{\pi}{2}p\left(1-\frac{\alpha}{\pi-\varepsilon r-\varepsilon^{\frac{1}{p}}r}\right)^{p-1}\frac{\alpha}{(\pi-\varepsilon r-\varepsilon^{\frac{1}{p}}r)^2}(\varepsilon+\varepsilon^{\frac{1}{p}})\frac{r}{\varepsilon}\\
&\approx \frac{\alpha}{\varepsilon}+\left(\pi-\varepsilon r-\varepsilon^{\frac{1}{p}}r-\alpha\right)^{p-1}r\alpha\lesssim \frac{1}{\varepsilon}\left(\pi-\varepsilon r-\varepsilon^{\frac{1}{p}}r-\alpha\right)^{p-1}.
\end{aligned}
$$
Using \eqref{gonio}, $R\approx 1$ and \eqref{elementary}  
we estimate the last term of the derivative 
$$
\frac{\left(R\cos T \sin(RT)\right)^2}{\alpha(\pi-\alpha)}
\lesssim \frac{R^2 T^2}{\alpha(\pi-\alpha)}
\lesssim
\frac{\left(1-\left(1-\frac{\alpha}{\pi-\varepsilon r-\varepsilon^{\frac{1}{p}}r}\right)^{p}\right)^2\frac{r^2}{\varepsilon^2}}{\alpha(\pi-\alpha)}
\lesssim \frac{r^2}{\varepsilon^2(\pi-\alpha)}.
$$
With the help of these estimates, \textcolor{black}{using \eqref{derivace} and $\partial_{\alpha}R=0$ we get} 
$$
\begin{aligned}
\int_{A_1\cap\{r<r_\varepsilon^0\}} \|Df\textcolor{black}{_\varepsilon}\|^2&\lesssim\int_0^{r_\varepsilon^0} \int_0^{\pi-\varepsilon r-\varepsilon^{\frac{1}{p}}r} 
\left[r^2 \left[|\partial_r R|^2+|\partial_r T|^2 \right]+ |\partial_{\alpha} T|^2+  
\frac{(R\cos T \sin(RT))^2}{\alpha(\pi-\alpha)}\right]  \textcolor{black}{\; d\alpha}\; dr\\
&\lesssim\int_0^{r_\varepsilon^0} \int_0^{\pi-\varepsilon r-\varepsilon^{\frac{1}{p}}r} 
\left[1+\frac{r^2(\pi-\varepsilon r-\varepsilon^{\frac{1}{p}}r-\alpha)^{2p-2}}{\varepsilon^2} +
\frac{r^2}{\varepsilon^2(\pi-\alpha)}\right]
\textcolor{black}{\; d\alpha}\; dr \\
\end{aligned}
$$
and the first part of the integral is finite since $p>\frac{1}{2}$ and $r\lesssim \varepsilon$. The second one we can estimate  as
$$
\frac{1}{\varepsilon^2}\int_0^{r_\varepsilon^0} r^2\int_0^{\pi-\varepsilon r-\varepsilon^{\frac{1}{p}}r} \frac{1}{\pi-\alpha}\; d\alpha\; dr
\lesssim\frac{1}{\varepsilon^2}\int_0^{r_\varepsilon^0} \varepsilon^2 (-\log(\varepsilon r))\; dr\lesssim 1. 
$$

It remains to estimate the Jacobian on $A_2$ using \eqref{jakobian}, $\partial_{\alpha}R=0$ and $R\approx 1$ 
$$
\int_{A_1\cap\{r<r_\varepsilon^0\}}\left|J_{f\textcolor{black}{_\varepsilon}}\right|^{-a}\lesssim \int_0^{r_{\varepsilon}} \int_0^{\pi-\varepsilon r-\varepsilon^{\frac{1}{p}}r}|\partial_{\alpha}T|^{-a}\bigl|(\cos T)^2\sin(RT)\bigr|^{-a}|r^2\sin \alpha|^{1+a}\; d\alpha\; dr. 
$$ 
Using \eqref{qqqq} we obtain 
$$
|\partial_{\alpha}T|\gtrsim (\pi-\varepsilon r-\varepsilon^{\frac{1}{p}}r-\alpha)^{p-1}\alpha r
$$
and using $R\approx 1$ and \eqref{elementary} we have 
$$
\sin(RT)\approx RT\approx T\approx \alpha\frac{r}{\varepsilon}. 
$$
Moreover, using again $\cos(\frac{\pi}{2}-y)\approx y$ we get
$$
\begin{aligned}
\cos T &\approx \frac{\pi}{2}-T\approx  \left(1-\frac{\alpha}{\pi-\varepsilon r-\varepsilon^{\frac{1}{p}}r}\right)^p+\left(1-\frac{r}{\varepsilon}\right)\\
&\geq \left(1-\frac{\alpha}{\pi-\varepsilon r-\varepsilon^{\frac{1}{p}}r}\right)^p
\gtrsim (\pi-\varepsilon r-\varepsilon^{\frac{1}{p}}r-\alpha)^p. 
\end{aligned}	
$$
Combining these estimates we obtain
$$
\begin{aligned}
&\int_{A_1\cap\{r<r_\varepsilon^0\}}\left|J_{f\textcolor{black}{_\varepsilon}}\right|^{-a}\lesssim \\
\phantom{=}&\lesssim \int_0^{r_\varepsilon^0} \int_0^{\pi-\varepsilon r-\varepsilon^{\frac{1}{p}}r} \frac{1}{(\pi-\varepsilon r-\varepsilon^{\frac{1}{p}}r-\alpha)^{ap-a}\alpha^a r^a}
\cdot\frac{|r^2\sin \alpha|^{1+a}}{(\pi-\varepsilon r-\varepsilon^{\frac{1}{p}}r-\alpha)^{2ap}}
\cdot\frac{\varepsilon^a}{\alpha^a r^a}\; d\alpha\; dr. \\
\end{aligned}
$$
As before (see \eqref{choosep}) the power of $(\pi-\varepsilon r-\varepsilon^{\frac{1}{p}}r-\alpha)$ is bigger than $-1$ and this term is integrable. Using $\sin \alpha\leq \alpha$ we obtain that the power of $\alpha$ is $-a-a+1+a=1-a>-1$ and this term is also integrable. The power of $r$ is $-a-a+2+2a$ and the power of $\varepsilon$ is positive so the whole integral is bounded.

{\underline{Estimate on $A_2$}: } We have $0<\alpha<\tilde{S}=(2-r)\pi-c_0\varepsilon^{1/p}(2-r)^\lambda$, $1<r_\varepsilon^1<r<2$,
$$
R_{\varepsilon}=\sqrt{\frac{\pi-2\varepsilon}{\pi-\varepsilon}}\cdot\frac{\sqrt{2-r}}{3}\text{ and }T_{\varepsilon}=\frac{\pi}{2}\Bigl(1-\Bigl(1-\frac{\alpha}{\tilde{S}}\Bigr)^p\Bigr)\psi,
$$
where $\psi=1-\left(\frac{\pi-2\varepsilon}{\pi-\varepsilon}\right)^{1-\lambda p}\varepsilon (2-r)^{\lambda p}$ and $\lambda =\frac{2}{1+a-3ap}>2$.

Since
$$
\partial_r \Bigl(\frac{\alpha}{\tilde{S}}\Bigr)=\frac{\alpha \left(\pi-\lambda c_0\varepsilon^{1/p}(2-r)^{\lambda-1}\right)}{\tilde{S}^2}\approx \frac{\alpha}{\tilde{S}^2},
$$
and $\lambda p-1>0$, we have
\begin{align*}
|\partial_r T|&=\frac{\pi}{2} p\Bigl(1-\frac{\alpha}{\tilde{S}}\Bigr)^{p-1}\frac{\alpha \left(\pi-c_0\varepsilon^{1/p}(2-r)^{\lambda-1}\right)}{\tilde{S}^2} \psi + \frac{\pi}{2}\Bigl(1-\Bigl(1-\frac{\alpha}{\tilde{S}}\Bigr)^p\Bigr)\partial_r\psi\\
& \lesssim \Bigl(\frac{\tilde{S}-\alpha}{\tilde{S}}\Bigr)^{p-1}\frac{\alpha}{\tilde{S}^2}+1\lesssim \Bigl(\frac{\tilde{S}-\alpha}{\tilde{S}}\Bigr)^{p-1}\frac{1}{\tilde{S}}+1
\end{align*}
and
$$
|\partial_\alpha T|=\frac{\pi}{2} \left| -p\left(1-\frac{\alpha}{\tilde{S}}\right)^{p-1}\frac{1}{\tilde{S}} \psi \right| \approx\left(\frac{\tilde{S}-\alpha}{\tilde{S}}\right)^{p-1}\frac{1}{\tilde{S}}.
$$
Using \eqref{elementary} we know that 
$$
T\approx 1-\left(1-\frac{\alpha}{\tilde{S}}\right)^p \approx \frac{\alpha}{\tilde{S}}
$$
and
\eqn{cosA2}
$$
\frac{\pi}{2}-T\approx  (1-\psi) + \left(1-\frac{\alpha}{\tilde{S}}\right)^{p}\psi\approx (1-\psi) + \left(1-\frac{\alpha}{\tilde{S}}\right)^{p},
$$
so together with \eqref{trojka}, $R\leq 1$ and $\alpha<\tilde{S}<\tilde{S}+c_0\varepsilon^{1/p}(2-r)^\lambda\leq \pi$ we get 
\begin{align*}
\frac{(R\cos T \sin(RT))^2}{\alpha(\pi-\alpha)} &\lesssim  \frac{ T^2 (\frac{\pi}{2}-T)^2}{\alpha(\pi-\alpha)}  
\approx \frac{ \frac{\alpha^2}{\tilde{S}^2} \Bigl((1-\psi) + \left(1-\frac{\alpha}{\tilde{S}}\right)^{p}\Bigr)^2}{\alpha(\pi-\alpha)}\\
&\lesssim \frac{\alpha^2}{\tilde{S}^2 \alpha} + \frac{(1-\psi)^2 + \Bigl(\frac{\tilde{S}-\alpha}{\tilde{S}}\Bigr)^{2p}}{\pi-\alpha} \lesssim \frac{1}{\tilde{S}} + \frac{\varepsilon^2(2-r)^{2\lambda p}}{\pi-\alpha} + \frac{(\tilde{S}-\alpha)^{2p}}{\tilde{S}^{2p}(\pi-\alpha)}\\
&\lesssim  \frac{1}{\tilde{S}} + \frac{\varepsilon^2(2-r)^{2\lambda p}}{\varepsilon^{1/p}(2-r)^\lambda} + \frac{(\tilde{S}-\alpha)^{2p}}{\tilde{S}^{2p}(\tilde{S}-\alpha)} \lesssim  \frac{1}{\tilde{S}}. \\
\end{align*}
Therefore using \eqref{derivace}, $\partial_{\alpha}R=0$, $\alpha<\tilde{S}<S\approx (2-r)$ and $p>1/2$ gives
\begin{align*}
\int_{A_2} \|Df\textcolor{black}{_\varepsilon}\|^2 &\lesssim \int_{r_\varepsilon^1}^2 \int_0^{\tilde{S}}\left[  \alpha (\pi-\alpha) \left[ (\partial_r R)^2 +  (\partial_r T)^2 + (\partial_\alpha T)^2 \right]  + \frac{(R\cos T \sin(RT))^2}{\alpha(\pi-\alpha)}\right]  d\alpha \; dr\\
&\lesssim \int_1^2 \int_0^{\tilde{S}} \left[  \alpha (\pi-\alpha) \left[ \frac{1}{2-r} + \frac{(\tilde{S}-\alpha)^{2p-2}}{\tilde{S}^{2p}} +1 + \frac{(\tilde{S}-\alpha)^{2p-2}}{\tilde{S}^{2p}} \right] +\frac{1}{\tilde{S}}\right] d\alpha\; dr \\
&\lesssim \int_1^2 \left[1+ \frac{\tilde{S}}{S} + \int_0^{\tilde{S}}  \alpha\frac{(\tilde{S}-\alpha)^{2p-2}}{\tilde{S}^{2p}} d\alpha\right] dr \lesssim \int_1^2 \left[1+ \tilde{S}\frac{\tilde{S}^{2p-1}}{\tilde{S}^{2p}} \right] dr \lesssim 1.
\end{align*}
Considering the Jacobian estimate, due to the fact that $\partial_\alpha R = 0$ we can rewrite \eqref{jakobian} as
$$
\int_{A_2} \left|J_{f\textcolor{black}{_\varepsilon}}\right|^{-a}\approx \int_{r_\varepsilon^1}^2 \int_0^{\tilde{S}}\left|\partial_r R\cdot \partial_\alpha T \right|^{-a} R^{-3a} |( \cos T)^2 \sin(RT)|^{-a}|r^2 \sin\alpha|^{1+a} d\alpha dr.\\
$$
We estimate (using again \eqref{gonio} and \eqref{elementary})
$$
\sin(RT)\approx RT\approx\sqrt{2-r} \Bigl(1-\Bigl(1-\frac{\alpha}{\tilde{S}}\Bigr)^p\Bigr)\approx \frac{\alpha}{\tilde{S}}\sqrt{2-r}
$$
and from \eqref{cosA2}
$$
\cos(T)\approx \frac{\pi}{2}-T\approx \Bigl(1-\frac{\alpha}{\tilde{S}}\Bigr)^p + (1-\psi)\gtrsim \Bigl(1-\frac{\alpha}{\tilde{S}}\Bigr)^p.
$$
Together this gives
\begin{align*}
&\int_{A_2} \left|J_{f\textcolor{black}{_\varepsilon}}\right|^{-a}\lesssim\\
&\lesssim  \int_{r_\varepsilon^1}^2 \int_0^{\tilde{S}} \Bigl|\frac{1}{\sqrt{2-r}} \Bigl(\frac{(\tilde{S}-\alpha)^{p-1}}{\tilde{S}^p}\Bigr) \Bigr|^{-a} \frac{1}{\sqrt{2-r}^{3a}}\Bigl(1-\frac{\alpha}{\tilde{S}}\Bigr)^{-2ap}\Bigl( \frac{\alpha}{\tilde{S}}\sqrt{2-r}\Bigr)^{-a} \alpha^{1+a}   d\alpha \; dr\\
&\lesssim \int_1^2 \int_0^{\tilde{S}} (\tilde{S}-\alpha)^{a-3ap} \tilde{S}^{3ap+a} (2-r)^{\frac{-3a}{2}}  d\alpha \; dr \\
&= \int_1^2 \tilde{S}^{3ap+a} (2-r)^{\frac{-3a}{2}}\int_0^{\tilde{S}} (\tilde{S}-\alpha)^{a-3ap}  d\alpha \; dr\lesssim 1,\\
\end{align*}
since $a-3ap>-1$ and $\tilde{S}<S\approx (2-r)$.

{\underline{Estimate on $D_1\cap \{r>r_\varepsilon^0\}$:} } On this set we have $S=\pi-\varepsilon r<\alpha<\pi$ and 
$$
R_{\varepsilon}=\frac{2}{3}+\frac{\varepsilon r}{3\pi}\text{ and }T_{\varepsilon}=\frac{\pi}{2}\left(1-\left(1-\frac{\pi-\alpha}{\varepsilon r}\right)^p\right)\left(1-(2-r)\varepsilon\right). 
$$
Let us first estimate 
\eqn{est}
$$
|\partial_{\alpha} T_{\varepsilon}|=\frac{\pi}{2}p\left(1-\frac{\pi-\alpha}{\varepsilon r}\right)^{p-1}\frac{1}{\varepsilon r}(1-(2-r)\varepsilon)
\approx (\alpha-\pi+\varepsilon r)^{p-1}\frac{1}{\varepsilon^p r^p} 
$$
and using $\pi-\alpha<\varepsilon r$
$$
\begin{aligned}
|\partial_r T_{\varepsilon}|&=\left|\frac{\pi}{2}\left(1-\left(1-\frac{\pi-\alpha}{\varepsilon r}\right)^p\right)\varepsilon -
\frac{\pi}{2}p\left(1-\frac{\pi-\alpha}{\varepsilon r}\right)^{p-1}\frac{\pi-\alpha}{\varepsilon r^2}(1-(2-r)\varepsilon)\right| \\
&\lesssim \frac{\varepsilon^{1-p}(\alpha-\pi+\varepsilon r)^{p-1}}{ r^{p}}.
\end{aligned}
$$
Now using $\sin RT\leq RT$, \eqref{elementary} and $\pi-\alpha<\varepsilon r$ we have 
$$
\frac{(R\cos T \sin(RT))^2}{\alpha(\pi-\alpha)}
\lesssim \frac{R^2 T^2}{\alpha(\pi-\alpha)}
\lesssim \frac{\left(1-\left(1-\frac{\pi-\alpha}{\varepsilon r}\right)^p\right)^2}{\alpha(\pi-\alpha)}
\lesssim \frac{\frac{(\pi-\alpha)^2}{\varepsilon^2 r^2}}{\alpha(\pi-\alpha)}
\leq \frac{1}{\alpha \varepsilon r}
\lesssim \frac{1}{\varepsilon r}.
$$
With the help of these estimates we use \eqref{derivace}, $\pi-\alpha< \varepsilon r$, 
$p>\frac{1}{2}$ and elementary integration to obtain 
$$
\begin{aligned}
\int_{D_1\cap\{r>r_\varepsilon^0\}} \|Df\textcolor{black}{_\varepsilon}\|^2&\lesssim\int_{r_{\epsilon}^0}^{r_\varepsilon^1} \int_{\pi-\varepsilon r}^{\pi}  
\Bigl[\alpha (\pi-\alpha) \left[|\partial_r R|^2+|\partial_r T|^2 + |\partial_{\alpha} T|^2\right]+\\
&\phantom{aaaaaaaaaaaaaaaaaaaaaaaaaaa} +  
\frac{(R\cos T \sin(RT))^2}{\alpha(\pi-\alpha)}\Bigr] \textcolor{black}{\; d\alpha} \; dr\\
&\lesssim\int_0^{r_\varepsilon^1} \int_{\pi-\varepsilon r}^{\pi}  
\left[(\pi-\alpha)(\alpha-\pi+\varepsilon r)^{2p-2}\frac{1}{\varepsilon^{2p} r^{2p}}+\frac{1}{\varepsilon r}\right] \; d\alpha\; dr\\
&\lesssim \int_0^{r_\varepsilon^1} \frac{1}{\varepsilon^{2p-1} r^{2p-1}} \int_{\pi-\varepsilon r}^{\pi} (\alpha-\pi+\varepsilon r)^{2p-2}\; d\alpha \;dr+\int_0^{r_\varepsilon^1} \frac{1}{\varepsilon r} (\varepsilon r)\; dr \\
&\lesssim \int_0^{r_\varepsilon^1} \frac{1}{\varepsilon^{2p-1} r^{2p-1}}  (\varepsilon r)^{2(p-1)+1} \;dr+1\approx 1.
\end{aligned}
$$

Now we estimate the Jacobian on $D_1$ using \eqref{jakobian}, $\partial_{\alpha}R=0$ and $R\approx 1$ as 
$$
\int_{D_1\cap\{r>r_\varepsilon^0\}}|J_{f\textcolor{black}{_\varepsilon}}|^{-a}\lesssim \int_0^{r_\varepsilon^1} \int_{\pi-\varepsilon r}^{\pi}|\partial_{\alpha}T|^{-a}\bigl|(\cos T)^2\sin(RT)\bigr|^{-a}|r^2\sin \alpha|^{1+a}\; d\alpha\; dr. 
$$ 
Using \eqref{est} we estimate $|\partial_{\alpha}T|$, further using \eqref{elementary}
$$
\sin(RT)\approx T\approx 1-\left(1-\frac{\pi-\alpha}{\varepsilon r}\right)^p\approx \frac{\pi-\alpha}{\varepsilon r}.
$$
As usual we estimate using $\cos(\frac{\pi}{2}-y)\approx y$ that
$$
\cos T\approx \left(1-\frac{\pi-\alpha}{\varepsilon r}\right)^p+(2-r)\varepsilon\gtrsim \left(\alpha-\pi+\varepsilon r\right)^p\frac{1}{\varepsilon^p r^p}.
$$
Combining these estimates with $|\frac{\sin \alpha}{\pi-\alpha}|\leq 1$ and $|\sin\alpha|\leq \varepsilon r$ we get
$$
\begin{aligned}
\int_{D_1\cap\{r>r_\varepsilon^0\}}|J_{f\textcolor{black}{_\varepsilon}}|^{-a}&\lesssim \int_{r_{\epsilon}^0}^{r_\varepsilon^1} \int_{\pi-\varepsilon r}^{\pi} \frac{\varepsilon^{ap} r^{ap}}{(\alpha-\pi+\varepsilon r)^{ap-a}} 
\frac{\varepsilon^{2ap} r^{2ap}}{\left(\alpha-\pi+\varepsilon r\right)^{2ap}}
\frac{(\varepsilon r)^a}{(\pi-\alpha)^a}
r^{2+2a}|\sin \alpha|^a \varepsilon r\; d\alpha\; dr\\
&\lesssim \varepsilon^{3ap+a+1}\int_0^{r_\varepsilon^1} r^{3ap+a+2+2a} \int_{\pi-\varepsilon r}^{\pi} \left(\alpha-\pi+\varepsilon r\right)^{a-3ap}\; d\alpha\; dr
\end{aligned}
$$
and this integral is bounded using \eqref{choosep}.

{\underline{Estimate on $D_1\cap \{r<r_\varepsilon^0\}$:} } On this set we have $\pi-\varepsilon r<\alpha<\pi$ and $0<r<r_\varepsilon^0\approx \varepsilon$ and 
$$
R_{\varepsilon}=\frac{2}{3}+\frac{\varepsilon r}{3\pi}\text{ and }T_{\varepsilon}=\frac{\pi}{2}\left(1-\left(1-\frac{\pi-\alpha}{\varepsilon r}\right)^p\right)\frac{r}{\varepsilon}. 
$$
Let us first estimate 
\eqn{est2}
$$
|\partial_{\alpha} T_{\varepsilon}|=\frac{\pi}{2}p\left(1-\frac{\pi-\alpha}{\varepsilon r}\right)^{p-1}\frac{1}{\varepsilon r}\cdot\frac{r}{\varepsilon}
\approx (\alpha-\pi+\varepsilon r)^{p-1}\frac{r^{1-p}}{\varepsilon^{1+p}}. 
$$
and using $\pi-\alpha<\varepsilon r$
$$
\begin{aligned}
|\partial_r T_{\varepsilon}|&=\left|\frac{\pi}{2}\left(1-\left(1-\frac{\pi-\alpha}{\varepsilon r}\right)^p\right)\frac{1}{\varepsilon} -
\frac{\pi}{2}p\left(1-\frac{\pi-\alpha}{\varepsilon r}\right)^{p-1}\frac{\pi-\alpha}{\varepsilon r^2}\cdot\frac{r}{\varepsilon}\right|\\
&\lesssim \frac{1}{\varepsilon} + (\alpha-\pi+\varepsilon r)^{p-1}\varepsilon^{1-p} r^{1-p}\frac{1}{\varepsilon}=\frac{1}{\varepsilon} + (\alpha-\pi+\varepsilon r)^{p-1} \frac{r^{1-p}}{\varepsilon^p}.
\end{aligned}
$$
Now using \eqref{elementary}, $\sin RT\leq RT$, $\alpha\approx 1$ and $\pi-\alpha< \varepsilon r$ we have
$$
\frac{(R\cos T \sin(RT))^2}{\alpha(\pi-\alpha)}
\lesssim \frac{R^2 T^2}{\alpha(\pi-\alpha)}
\lesssim \frac{\left(1-\left(1-\frac{\pi-\alpha}{\varepsilon r}\right)^p\right)^2\frac{r^2}{\varepsilon^2}}{\alpha(\pi-\alpha)}
\lesssim \frac{\frac{(\pi-\alpha)^2}{\varepsilon^2 r^2}\cdot \frac{r^2}{\varepsilon^2}}{\alpha(\pi-\alpha)}
\lesssim \frac{r}{\varepsilon^3}.
$$
With the help of these estimates we use \eqref{derivace}, $\pi-\alpha< \varepsilon r$, $p>\frac{1}{2}$, $r<r_\varepsilon^0\approx \varepsilon$ and elementary integration to obtain 
$$
\begin{aligned}
\int_{D_1\cap\{r<r_\varepsilon^0\}} \|Df\textcolor{black}{_\varepsilon}\|^2&\lesssim\int_0^{r_\varepsilon^0} \int_{\pi-\varepsilon r}^{\pi} 
\Bigl[\alpha (\pi-\alpha) \left[|\partial_r R|^2+|\partial_r T|^2 + |\partial_{\alpha} T|^2\right]+ \\
&\phantom{aaaaaaaaaaaaaaaaaaaaaaaaaaa} +  \frac{(R\cos T \sin(RT))^2}{\alpha(\pi-\alpha)}\Bigr]  \textcolor{black}{\; d\alpha}\; dra\\
&\lesssim\int_0^{r_\varepsilon^0} \int_{\pi-\varepsilon r}^{\pi} 
\left[1+\frac{\pi-\alpha}{\varepsilon^2} + (\pi-\alpha)(\alpha-\pi+\varepsilon r)^{2p-2}\frac{r^{2-2p}}{\varepsilon^{2+2p}}+\frac{r}{\varepsilon^3}\right] \; d\alpha\; dr\\
&\lesssim 1+ \int_0^{r_\varepsilon^0}\int_{\pi-\varepsilon r}^{\pi}\left[\frac{\varepsilon r}{\varepsilon^2} + (\alpha-\pi+\varepsilon r)^{2p-2}\frac{r^{3-2p}}{\varepsilon^{1+2p}}\right]\; d\alpha \;dr+\int_0^{r_\varepsilon^0}\frac{r}{\varepsilon^3} (\varepsilon r)\; dr
\\
&\lesssim  1+ \int_0^{r_\varepsilon^0}\left[\frac{\varepsilon^2 r^2}{\varepsilon^2}+ \frac{r^{3-2p}}{\varepsilon^{1+2p}}  (\varepsilon r)^{2p-2+1}\right] \;dr+1\approx 1.
\end{aligned}
$$
Now we estimate the Jacobian on $D_1$ using \eqref{jakobian}, $\partial_{\alpha}R=0$ and $R\approx 1$ as 
$$
\int_{D_1\cap\{r<r_\varepsilon^0\}}|J_{f\textcolor{black}{_\varepsilon}}|^{-a}\lesssim \int_0^{r_\varepsilon^0} \int_{\pi-\varepsilon r}^{\pi}|\partial_{\alpha}T|^{-a}\bigl|(\cos T)^2\sin(RT)\bigr|^{-a}|r^2\sin \alpha|^{1+a}\; d\alpha\; dr. 
$$ 
Using \eqref{est2} we estimate $|\partial_{\alpha}T|$, further using \eqref{elementary} we get
$$
\sin(RT)\approx T\approx \left[1-\left(1-\frac{\pi-\alpha}{\varepsilon r}\right)^p\right]\frac{r}{\varepsilon}\approx \frac{\pi-\alpha}{\varepsilon r}\cdot\frac{r}{\varepsilon}=\frac{\pi-\alpha}{\varepsilon^2}.
$$
As usual we estimate using $\cos(\frac{\pi}{2}-y)\approx y$ that
$$
\cos T\approx \left(1-\frac{\pi-\alpha}{\varepsilon r}\right)^p+\left(1-\frac{\varepsilon}{r}\right)\gtrsim \frac{\left(\alpha-\pi+\varepsilon r\right)^p}{\varepsilon^p r^p}.
$$
Combining these estimates with $|\frac{\sin \alpha}{\pi-\alpha}|\leq 1$ and $|\sin\alpha|\leq \varepsilon r$ we get
$$
\begin{aligned}
\int_{D_1\cap\{r<r_\varepsilon^0\}}|J_{f\textcolor{black}{_\varepsilon}}|^{-a}&\lesssim \int_0^{r_\varepsilon^0} \int_{\pi-\varepsilon r}^{\pi} 
\frac{\varepsilon^{ap+a} r^{ap-a}}{(\alpha-\pi-\varepsilon r)^{ap-a}}
\frac{\varepsilon^{2ap} r^{2ap}}{\left(\alpha-\pi-\varepsilon r\right)^{2ap}}
\frac{\varepsilon^{2a}}{(\pi-\alpha)^a}
|r^2|^{1+a}|\sin \alpha|^a \varepsilon r\; d\alpha\; dr\\
&\lesssim \varepsilon^{3ap+3a+1}\int_0^{r_\varepsilon^0} r^{3ap-a+2+2a+1} \int_{\pi-\varepsilon r}^{\pi} \left(\alpha-\pi-\varepsilon r\right)^{a-3ap}\; d\alpha\; dr
\end{aligned}
$$
and this integral is bounded using \eqref{choosep}.

{\underline{Estimate on $D_2$}: }  We have $S=(2-r)\pi<\alpha<\pi$, $r_\varepsilon^1<r<2$ and
$$
R_{\varepsilon}=\frac{1+r}{3}\text{ and }T_{\varepsilon}=\frac{\pi}{2}\left(1-\left(1-\frac{\pi-\alpha}{\pi-(2-r)\pi}\right)^p\right)\psi,
$$
where $\psi=1-\left(\frac{\pi-2\varepsilon}{\pi-\varepsilon}\right)^{1-\lambda p}\varepsilon (2-r)^{\lambda p}$ and $\lambda =\frac{2}{1+a-3ap}>2$.

First notice that similarly as before, since $\pi-\alpha<\pi-(2-r)\pi=(r-1)\pi$,
\begin{align*}
|\partial_r T|&=\frac{\pi}{2} \Bigl| -p\Bigl(1-\frac{\pi-\alpha}{\pi-(2-r)\pi}\Bigr)^{p-1}\frac{\pi(\pi-\alpha)}{(\pi-(2-r)\pi)^2} \psi + \Bigl(1-\Bigl(1-\frac{\pi-\alpha}{\pi-(2-r)\pi}\Bigr)^p\Bigr)\partial_r\psi \Bigr| \\
&\lesssim \frac{(\alpha-(2-r)\pi)^{p-1}(\pi-\alpha)}{(r-1)^{p+1}}+1 \lesssim \frac{(\alpha-(2-r)\pi)^{p-1}}{(r-1)^{p}}+1,\\
|\partial_\alpha T|&=\frac{\pi}{2} \Bigl| -p\Bigl(1-\frac{\pi-\alpha}{\pi-(2-r)\pi}\Bigr)^{p-1}\frac{1}{\pi-(2-r)\pi} \psi \Bigr| \approx\frac{(\alpha-(2-r)\pi)^{p-1}}{(r-1)^{p}}.
\end{align*}
To estimate the following term, we again use \eqref{gonio}, \eqref{trojka} and \eqref{elementary} and $0<\alpha-(2-r)\pi<\pi-(2-r)\pi=(r-1)\pi$:
\begin{align*}
\frac{(R\cos T \sin(RT))^2}{\alpha(\pi-\alpha)} &\lesssim  \frac{T^2 (\frac{\pi}{2}-T)^2  }{\alpha(\pi-\alpha)} 
\approx \frac{ \Bigl(1-\left(1-\frac{\pi-\alpha}{\pi-(2-r)\pi}\right)^p\Bigr)^2 \left((1-\psi) +\left(1-\frac{\pi-\alpha}{\pi-(2-r)\pi}\right)^{p}\psi\right)^2}{\alpha(\pi-\alpha)}\\
&\approx \frac{ \left(1-\left(1-\frac{\pi-\alpha}{\pi-(2-r)\pi}\right)\right)^2 }{\pi-\alpha}+\frac{\Bigl((1-\psi)^2 
+\Bigl(1-\frac{\pi-\alpha}{\pi-(2-r)\pi}\Bigr)^{2p}\Bigr)}{\alpha}\\
&\lesssim \frac{\pi-\alpha}{(\pi-(2-r)\pi)^2 } + \frac{(1-\psi)^2 + \left(\frac{\alpha-(2-r)\pi}{\pi-(2-r)\pi}\right)^{2p}}{\alpha}\\
&\lesssim \frac{1}{r-1} + \frac{\varepsilon^2(2-r)^{2\lambda p}}{2-r} + \frac{(\alpha-(2-r)\pi)^{2p-1}}{\left(\pi-(2-r)\pi\right)^{2p}} \lesssim   \frac{1}{r-1}.
\end{align*}
Now we can use all those estimates to integrate \eqref{derivace}
\begin{align*}
\int_{D_2} \|Df\textcolor{black}{_\varepsilon}\|^2 &\lesssim \int_{r_\varepsilon^1}^2 \int_{(2-r)\pi}^\pi \Bigl[  \alpha (\pi-\alpha) \left[ (\partial_r R)^2 + (\partial_r T)^2 + (\partial_\alpha T)^2 \right] +\\
&\phantom{aaaaaaaaaaaaaaaaaaaaaaaaaaa}  + \frac{(R\cos T \sin(RT))^2}{\alpha(\pi-\alpha)}\Bigr] d\alpha \; dr\\
&\lesssim \int_{r_\varepsilon^1}^2 \int_{(2-r)\pi}^\pi \left[  \alpha (\pi-\alpha) \left[ 1 + \frac{(\alpha-(2-r)\pi)^{2p-2}}{(r-1)^{2p}} \right]  +  \frac{1}{r-1} \right] d\alpha \; dr\\
&\lesssim \int_{r_\varepsilon^1}^2  \left[ 1+\frac{\pi-(2-r)\pi}{r-1} + \int_{(2-r)\pi}^\pi \frac{(\alpha-(2-r)\pi)^{2p-2}(\pi-\alpha)}{(r-1)^{2p}} d\alpha\right] dr\\
&\lesssim \int_{r_\varepsilon^1}^2  \left[ 1 + \frac{(\pi-(2-r)\pi)^{2p}}{(r-1)^{2p}} \right] dr\lesssim 1.\\
\end{align*}
To integrate the Jacobian, we estimate (using again \eqref{gonio} and \eqref{elementary})
$$
\sin(RT)\approx RT\approx\Bigl(1-\Bigl(1-\frac{\pi-\alpha}{\pi-(2-r)\pi}\Bigr)^p\Bigr)\approx \frac{\pi-\alpha}{r-1}
$$
and
$$
\cos(T)\approx \frac{\pi}{2}-T\gtrsim \Bigl(1-\frac{\pi-\alpha}{\pi-(2-r)\pi}\Bigr)^p.
$$
Since $\partial_\alpha R = 0$ and $R\approx 1$ we have using \eqref{jakobian}
\begin{align*}
\int_{D_2} |J_{f\textcolor{black}{_\varepsilon}}|^{-a} &\approx \int_{r_\varepsilon^1}^2 \int_{(2-r)\pi}^\pi\left| \partial_\alpha T \right|^{-a} |( \cos T)^2 \sin(RT)|^{-a}|r^2 \sin\alpha|^{1+a}\; d\alpha \; dr\\
&\lesssim\int_{r_\varepsilon^1}^2 \int_{(2-r)\pi}^\pi\left| \frac{(\alpha-(2-r)\pi)^{p-1}}{(r-1)^{p}}\right|^{-a} \Bigl(1-\frac{\pi-\alpha}{\pi-(2-r)\pi}\Bigr)^{-2ap} \Bigl(\frac{\pi-\alpha}{r-1}\Bigr)^{-a}(\pi-\alpha)^{1+a} \; d\alpha \; dr\\
&\lesssim \int_{r_\varepsilon^1}^2 \int_{(2-r)\pi}^\pi (\alpha-(2-r)\pi)^{a-3ap}(r-1)^{a+3ap} \; d\alpha \; dr \lesssim \int_{r_\varepsilon^1}^2 (\pi-(2-r)\pi)^{1+a-3ap} \;  dr \lesssim 1.\\
\end{align*}

{\bf Step 5. Integrability of $|Df\textcolor{black}{_\varepsilon}|^2$ and $J_{f\textcolor{black}{_\varepsilon}}^{-a}$ on $B\cup C$:}

{\underline{Estimate on $B\cap \{r>r_\varepsilon^0\}$:} } On this set we have 
$\tilde{S}_{\varepsilon}=\pi-\varepsilon r-\varepsilon^{\frac{1}{p}}r<\alpha<\pi-\varepsilon r=S_{\varepsilon}$ and 
$$
R_{\varepsilon}=\left(\frac{2}{3}+\frac{\varepsilon r}{3\pi}\right)
\frac{\alpha-(\pi-\varepsilon r-\varepsilon^{\frac{1}{p}} r)}{\varepsilon^{\frac{1}{p}} r}
+\frac{2-r}{3}\cdot
\frac{\pi-\varepsilon r-\alpha}{\varepsilon^{\frac{1}{p}} r}
\text{ and }T_{\varepsilon}=\frac{\pi}{2}(1-(2-r)\varepsilon). 
$$
Let us first estimate 
\eqn{aha1}
$$
|\partial_{\alpha} R_{\varepsilon}|=\left(\frac{2}{3}+\frac{\varepsilon r}{3\pi}\right)\frac{1}{\varepsilon^{\frac{1}{p}} r}
+\frac{2-r}{3}\cdot\frac{-1}{\varepsilon^{\frac{1}{p}} r}=\frac{\varepsilon^{1-\frac{1}{p}}}{3\pi}+\frac{1}{3\varepsilon^{\frac{1}{p}}}\approx\frac{1}{\varepsilon^{\frac{1}{p}}} . 
$$
and with the help of $\pi-\varepsilon r-\varepsilon^{\frac{1}{p}}r<\alpha<\pi-\varepsilon r$
\eqn{aha2}
$$
\begin{aligned}
|\partial_r R_{\varepsilon}|&=\Bigl|\frac{\varepsilon}{3\pi}\cdot\frac{\alpha-(\pi-\varepsilon r-\varepsilon^{\frac{1}{p}} r)}{\varepsilon^{\frac{1}{p}} r}
+\left(\frac{2}{3}+\frac{\varepsilon r}{3\pi}\right)\frac{\pi-\alpha}{\varepsilon^{\frac{1}{p}} r^2}\\
&\phantom{=|}+\frac{-1}{3}\cdot\frac{\pi-\varepsilon r-\alpha}{\varepsilon^{\frac{1}{p}} r}
+\frac{2-r}{3}\cdot\frac{\alpha-\pi}{\varepsilon^{\frac{1}{p}} r^2}\Bigr|\lesssim\frac{\varepsilon}{\varepsilon^{\frac{1}{p}}r}.
\\
\end{aligned}
$$
Further we estimate for $\pi-\varepsilon r-\varepsilon^{\frac{1}{p}}r<\alpha<\pi-\varepsilon r$
$$
\frac{(R\cos T \sin(RT))^2}{\alpha(\pi-\alpha)}\lesssim\frac{(\frac{\pi}{2}-T)^2}{\pi-\alpha}
\lesssim\frac{\varepsilon^2}{\varepsilon r}.
$$
Now using \eqref{derivace} (note that each term with $\partial_{\alpha} R$ always contains also $\cos T\approx \frac{\pi}{2}-T\approx\varepsilon$) we obtain using 
$\pi-\alpha\approx \varepsilon r$ and $\frac{1}{2}<p<1$
$$
\begin{aligned}
\int_{B\cap \{r>r_\varepsilon^0\}} \|Df\textcolor{black}{_\varepsilon}\|^2 &\lesssim\int_{r_{\epsilon}^0}^{r_\varepsilon^1} \int_{\pi-\varepsilon r-\varepsilon^{\frac{1}{p}} r}^{\pi-\varepsilon r} 
\Bigl[\alpha (\pi-\alpha) 
\left[|\partial_r R|^2+|\partial_r T|^2 + |\cos T\cdot \partial_{\alpha} R|^2\right]+ \\
\phantom{\leq\leq}&+ \frac{(R\cos T \sin(RT))^2}{\alpha(\pi-\alpha)}\Bigr]  \; dr\;  d\alpha\\
&\lesssim \int_0^{r_\varepsilon^1} (\varepsilon^{\frac{1}{p}} r)(\varepsilon r)\left[\frac{\varepsilon^2}{\varepsilon^{\frac{2}{p}} r^2}+1+\frac{\varepsilon^2}{\varepsilon^{\frac{2}{p}}}\right]\; dr
+\int_0^{r_\varepsilon^1} \int_{\pi-\varepsilon r-\varepsilon^{\frac{1}{p}}r}^{\pi-\varepsilon r}\frac{\varepsilon}{r}\; dr\; d\alpha \\
&\lesssim 1+\int_0^{r_\varepsilon^1} \varepsilon^{3-\frac{1}{p}}\; dr + \int_0^{r_\varepsilon^1} \varepsilon^{\frac{1}{p}} r\cdot\frac{\varepsilon}{r}\; dr \approx 1.\\
\end{aligned}
$$
Now we estimate the Jacobian on $B$ using \eqref{jakobian}, $\partial_{\alpha}T=0$, $\cos T\approx \frac{\pi}{2}-T\approx \varepsilon$, $\sin RT\approx 1$ and \eqref{choosep} as 
$$
\begin{aligned}
\int_{B\cap \{r>r_\varepsilon^0\}}|J_{f\textcolor{black}{_\varepsilon}}|^{-a}
&\lesssim \int_{r_{\epsilon}^0}^{r_\varepsilon^1} 
\int_{\pi-\varepsilon r-\varepsilon^{\frac{1}{p}}r}^{\pi-\varepsilon r}|\partial_{\alpha} R\cdot \partial_{r}T|^{-a}
\bigl|(\cos T)^2\sin(RT)\bigr|^{-a}|r^2\sin \alpha|^{1+a}\; d\alpha\; dr\\
&\lesssim \int_0^{r_\varepsilon^1} 
\int_{\pi-\varepsilon r-\varepsilon^{\frac{1}{p}}r}^{\pi-\varepsilon r}\Bigl|\frac{1}{\varepsilon^{\frac{1}{p}}}\ \varepsilon\Bigr|^{-a}
\varepsilon^{-2a} \; d\alpha\; dr = \int_0^{r_\varepsilon^1}
(\varepsilon^{\frac{1}{p}} r)\varepsilon^{\frac{a}{p}-a} \varepsilon^{-2a}\; dr \\
&\lesssim \int_0^{r_\varepsilon^1}
\varepsilon^{\frac{1+a}{p}-3a} \; dr \lesssim 1.
\\
\end{aligned}
$$

{\underline{Estimate on $B\cap \{r<r_\varepsilon^0\}$:} } On this set we have 
$\tilde{S}_{\varepsilon}=\pi-\varepsilon r-\varepsilon^{\frac{1}{p}}r<\alpha<\pi-\varepsilon r=S_{\varepsilon}$ and 
$$
R_{\varepsilon}=\left(\frac{2}{3}+\frac{\varepsilon r}{3\pi}\right)
\frac{\alpha-(\pi-\varepsilon r-\varepsilon^{\frac{1}{p}} r)}{\varepsilon^{\frac{1}{p}} r}
+\frac{2-r}{3}\cdot
\frac{\pi-\varepsilon r-\alpha}{\varepsilon^{\frac{1}{p}} r}
\text{ and }T_{\varepsilon}=\frac{\pi}{2}\cdot\frac{r}{\varepsilon}. 
$$
We can use the same estimates \eqref{aha1} and \eqref{aha2} as before. 
Further we estimate for $\pi-\varepsilon r-\varepsilon^{\frac{1}{p}}r<\alpha<\pi-\varepsilon r$
$$
\frac{(R\cos T \sin(RT))^2}{\alpha(\pi-\alpha)}\lesssim\frac{T^2}{\pi-\alpha}
\lesssim\frac{\frac{r^2}{\varepsilon^2}}{\varepsilon r}=\frac{r}{\varepsilon^3}.
$$
Now using \eqref{derivace}, $\pi-\alpha\approx \varepsilon r$ and $r_\varepsilon^0\approx\varepsilon$ we obtain 
$$
\begin{aligned}
\int_{B\cap \{r<r_\varepsilon^0\}} \|Df\textcolor{black}{_\varepsilon}\|^2 &\lesssim\int_0^{r_\varepsilon^0} \int_{\pi-\varepsilon r-\varepsilon^{\frac{1}{p}}r}^{\pi-\varepsilon r} 
\Bigl[\alpha (\pi-\alpha) 
\left[|\partial_r R|^2+|\partial_r T|^2 + |\partial_{\alpha} R|^2\right]+  \\
&\phantom{aaaaaaaaaaaaaaaaaaaaaaaaaaa} +  
\frac{(R\cos T \sin(RT))^2}{\alpha(\pi-\alpha)}\Bigr]  \;  d\alpha\; dr\\
&\lesssim \int_0^{r_\varepsilon^0} (\varepsilon^{\frac{1}{p}} r)(\varepsilon r)\left[\frac{\varepsilon^2}{\varepsilon^{\frac{2}{p}}r^2}+\frac{1}{\varepsilon^2}+\frac{1}{\varepsilon^{\frac{2}{p}}}\right]\; dr
+\int_0^{r_\varepsilon^0} \int_{\pi-\varepsilon r-\varepsilon^{\frac{1}{p}}r}^{\pi-\varepsilon r}\;\frac{r}{\varepsilon^3}\; d\alpha \; dr\\
&\lesssim 1+ \varepsilon^{1-\frac{1}{p}}\int_0^{r_\varepsilon^0}r^2\; dr+\int_0^{r_\varepsilon^0} (\varepsilon^{\frac{1}{p}} r)\frac{r}{\varepsilon^3}\; dr\approx 1.\\
\end{aligned}
$$
Now we estimate the Jacobian on $B$ using \eqref{jakobian}, $\partial_{\alpha}T=0$, 
$\sin RT\approx T\approx \frac{r}{\varepsilon}$ and
$$ 
\cos T\approx\frac{\pi}{2}-T\approx\frac{\varepsilon-r}{\varepsilon}
\text{ and }\varepsilon-r_\varepsilon^0=\varepsilon-\frac{\varepsilon-2\varepsilon^2}{1-\varepsilon^2}\approx \varepsilon^2
$$ 
as 
$$
\begin{aligned}
\int_{B\cap \{r<r_\varepsilon^0\}}|J_{f\textcolor{black}{_\varepsilon}}|^{-a}
&\lesssim \int_0^{r_\varepsilon^0} 
\int_{\pi-\varepsilon r-\varepsilon^{\frac{1}{p}}r}^{\pi-\varepsilon r}|\partial_{\alpha} R\cdot \partial_{r}T|^{-a}
\bigl|(\cos T)^2\sin(RT)\bigr|^{-a}|r^2\sin \alpha|^{1+a}\; d\alpha\; dr\\
&\lesssim \int_0^{r_\varepsilon^0} 
\int_{\pi-\varepsilon r-\varepsilon^{\frac{1}{p}}r}^{\pi-\varepsilon r}\left|\frac{1}{\varepsilon^{\frac{1}{p}}}\cdot \frac{1}{\varepsilon}\right|^{-a}
\frac{\varepsilon^{2a}}{(\varepsilon-r)^{2a}}\cdot \frac{\varepsilon^a}{r^a}|r^2|^{1+a}\; d\alpha\; dr\\
&\lesssim \varepsilon^{\frac{a}{p}+4a}\int_0^{r_\varepsilon^0} \frac{1}{(\varepsilon-r)^{2a}} dr
\lesssim \varepsilon^{\frac{a}{p}+4a} (\varepsilon-r_\varepsilon^0)^{1-2a}\approx 
 \varepsilon^{\frac{a}{p}+4a} \varepsilon^{2-4a}\lesssim 1. 
\\
\end{aligned}
$$

{\underline{Estimate on $C$}: } We have $(2-r)\pi-\delta(\varepsilon, r)=\tilde{S}<\alpha<S=(2-r)\pi$, $r_\varepsilon^1<r<2$,
$$
R_{\varepsilon}=\sqrt{\frac{\pi-2\varepsilon}{\pi-\varepsilon}}\cdot\frac{\sqrt{2-r}}{3}\cdot \frac{(2-r)\pi-\alpha}{\delta(\varepsilon, r)} +\frac{1+r}{3}\left(\frac{\alpha-(2-r)\pi + \delta(\varepsilon,r)}{\delta(\varepsilon, r)}\right)
$$
and
$$
T_{\varepsilon}=\frac{\pi}{2}\psi=\frac{\pi}{2}\Bigl(1-\Bigl(\frac{\pi-2\varepsilon}{\pi-\varepsilon}\Bigr)^{1-\lambda p}\varepsilon (2-r)^{\lambda p}\Bigr),
$$
where $\lambda =\frac{2}{1+a-3ap}$ and $\delta(\varepsilon, r)=c_0\varepsilon^{1/p}(2-r)^\lambda$. Here it is crucial that 
\eqn{lambda_p}
$$
\lambda p  \geq \frac{2}{1+a-3a/2}\cdot\frac{1}{2}\geq 1
$$
and  
\eqn{delta_psi}
$$
\cos T\approx \frac{\pi}{2}-T\approx 1-\psi\approx \varepsilon (2-r)^{\lambda p}\approx \delta^p.
$$
We first estimate
$$
\partial_r\delta = -\lambda c_0\varepsilon^{1/p}(2-r)^{\lambda-1} = \frac{-\lambda}{2-r} \delta,
$$
so
$$
\partial_r\left(\frac{(2-r)\pi-\alpha}{\delta}\right) =  \frac{-\pi\delta + ((2-r)\pi-\alpha)  \frac{-\lambda\pi}{(2-r)\pi} \delta}{\delta^2}=-\pi\frac{1+\lambda\frac{(2-r)\pi-\alpha}{(2-r)\pi}}{\delta}.
$$
Then we can use it to estimate
\begin{align*}
|\partial_r R|& = \Bigl| \sqrt{\frac{\pi-2\varepsilon}{\pi-\varepsilon}}\cdot\frac{1}{6\sqrt{(2-r)}}\cdot\frac{(2-r)\pi-\alpha}{\delta}+\sqrt{\frac{\pi-2\varepsilon}{\pi-\varepsilon}}\cdot\frac{\sqrt{2-r}}{3}\Bigl(-\pi\frac{1+\lambda\frac{(2-r)\pi-\alpha}{(2-r)\pi}}{\delta}\Bigr)\\
&+\frac{\alpha-(2-r)\pi+\delta}{3\delta}+\frac{1+r}{3}\cdot\pi\frac{1+\lambda\frac{(2-r)\pi-\alpha}{(2-r)\pi}}{\delta}\Bigr|\lesssim \frac{1}{\sqrt{2-r}}+\frac{1}{\delta}+1+\frac{1}{\delta}\lesssim \frac{1}{\sqrt{2-r}}+\frac{1}{\delta}\\
\intertext{ and }
|\partial_\alpha R|&=\frac{-\sqrt{\frac{\pi-2\varepsilon}{\pi-\varepsilon}}\sqrt{(2-r)}+1+r}{3\delta}\approx \frac{1}{\delta}.
\end{align*}
Since $|\partial_\alpha T|=0$, using \eqref{derivace}, \eqref{delta_psi}, \eqref{lambda_p} and \eqref{trojka} we obtain
\begin{align*}
\int_{C} \|Df\textcolor{black}{_\varepsilon}\|^2 &\lesssim \int_{r_\varepsilon^1}^2 \int_{(2-r)\pi-\delta}^{(2-r)\pi}  \Bigl[ |\cos T\cdot\partial_r R|^2 + |\partial_r T|^2  + |\cos T\cdot\partial_\alpha R|^2 + \\
&\phantom{aaaaaaaaaaaaaaaaaaaaaaaaaaa} +  \frac{(R\cos T \sin(RT))^2}{\alpha(\pi-\alpha)}\Bigr]  d\alpha \; dr\\
&\lesssim \int_{r_\varepsilon^1}^2 \int_{(2-r)\pi-\delta}^{(2-r)\pi}  \left[ \left(\frac{1}{2-r}+\frac{1}{\delta^2}\right) \delta^{2p}+\varepsilon^2 (2-r)^{2\lambda p-2} +\frac{\delta^{2p}}{\delta^2}  + \frac{ 1}{\alpha} +\frac{1}{\pi-\alpha} \right] d\alpha \; dr\\
&\lesssim \int_{r_\varepsilon^1}^2 \int_{(2-r)\pi-\delta}^{(2-r)\pi}  \left[ \frac{\delta^{2p}}{2-r}+\frac{\delta^{2p}}{\delta^2}+1  + \frac{ 1 }{\alpha} +\frac{1}{\pi-\alpha} \right]d\alpha \; dr\\
&\approx 1+ \int_{r_\varepsilon^1}^2 \left[\frac{\delta^{2p+1}}{2-r}+\delta^{2p-1}  + \log\left(\frac{(2-r)\pi}{(2-r)\pi-\delta}\right) +\log\left(\frac{\pi-(2-r)\pi+\delta}{\pi-(2-r)\pi}\right)\right] dr, \\
\intertext{now since $\delta\leq (2-r)\pi/2$ and $\varepsilon<\varepsilon r<\pi-(2-r)\pi$ (recall that $\pi-\epsilon r>(2-r)\pi$ for $r>r_{\epsilon}^1$)}
&\lesssim 1+ \int_{r_\varepsilon^1}^2 \frac{(2-r)^{2\lambda p+\lambda}}{2-r}+1 + \log\Bigl(\frac{(2-r)\pi}{(2-r)\pi/2}\Bigr) +\log\Bigr(1+\frac{\delta}{\varepsilon r}\Bigr) dr \\
&\lesssim 1+ \int_{r_\varepsilon^1}^2 (2-r)^{2\lambda p+\lambda-1}  + \log\left(2\right) +\log\Bigl(1+\frac{\varepsilon^{1/p}}{\varepsilon }\Bigr)  dr\lesssim 1. \\
\end{align*}

Considering the Jacobian estimate, due to the fact that $\partial_\alpha T = 0$ we can rewrite \eqref{jakobian} as
$$
\int_C |J_{f\textcolor{black}{_\varepsilon}}|^{-a}\approx \int_{(r, \alpha)}\left| \partial_\alpha R\cdot \partial_r T \right|^{-a} R^{-3a} |( \cos T)^2 \sin(RT)|^{-a}|r^2 \sin(\alpha)|^{1+a}dr d\alpha.
$$
Also we need
\begin{align*}
\sin(RT)&\approx RT \approx R = \sqrt{\frac{\pi-2\varepsilon}{\pi-\varepsilon}}\cdot\frac{\sqrt{2-r}}{3}\cdot \frac{(2-r)\pi-\alpha}{\delta} +\frac{1+r}{3}\Bigl(\frac{\alpha-(2-r)\pi + \delta}{\delta}\Bigr)\\
&\geq  C\Bigl[\sqrt{2-r}\cdot \frac{(2-r)\pi-\alpha}{\delta} +\sqrt{2-r}\Bigl(\frac{\alpha-(2-r)\pi + \delta}{\delta}\Bigr) \Bigr]=C\sqrt{2-r}.
\end{align*}
Together with \eqref{delta_psi}, $\delta\approx \varepsilon^{1/p}(2-r)^\lambda$ and \eqref{choosep} it yields
\begin{align*}
&\int_C |J_{f\textcolor{black}{_\varepsilon}}|^{-a}\approx \int_{r_\varepsilon^1}^2\int_{(2-r)\pi-\delta}^{(2-r)\pi}\left|\frac{\varepsilon(2-r)^{\lambda p-1}}{\delta}\right|^{-a} R^{-4a}\delta^{-2ap} 
\sin^{1+a} \alpha \;d\alpha \; dr\\
&\lesssim \int_{r_\varepsilon^1}^2 \int_{(2-r)\pi-\delta}^{(2-r)\pi}\frac{\delta^{a-2ap}}{\varepsilon^a} (2-r)^{a-\lambda ap}\left(\sqrt{2-r}\right)^{-4a}  d\alpha \; dr \approx \int_{r_\varepsilon^1}^2 \frac{\delta^{1+a-2ap}}{\varepsilon^a (2-r)^{a+\lambda ap}}\;dr \\
&\approx \int_{r_\varepsilon^1}^2  \frac{\varepsilon^{(1+a-2ap)/p}(2-r)^{\lambda(1+a-2ap)}}{\varepsilon^a (2-r)^{a+\lambda ap}}\;dr =\int_{r_\varepsilon^1}^2 \varepsilon^{(1+a-3ap)/p}(2-r)^{\lambda(1+a-3ap)-a} \;dr \\
& \lesssim \int_{r_\varepsilon^1}^2(2-r)^{2-a}dr\lesssim 1,
\end{align*}
since $\lambda =\frac{2}{1+a-3ap}$.

{\bf Step 6. Extension to $\overline{B(0,10)}\setminus B(0,2)$:} 
Our mapping $f\textcolor{black}{_\varepsilon}:\overline{B(0,2)}\to\er^3$ is defined on the sphere $\partial B(0,2)$ in polar coordinates as 
$$
f(2,\alpha,\beta)=(\cos T,T,\beta)
$$
where 
$$
T=T(\varepsilon,\alpha)=\frac{\pi}{2}\left(1-\left(\frac{\alpha}{\pi}\right)^p\right). 
$$

We define it on $B(0,10)\setminus B(0,2)$ as a simple interpolation between $(r,\pi-\alpha,\beta)$ on $\partial B(0,10)$ and this mapping on $\partial B(0,2)$, i.e. for $r\in[2,10]$ we set
$$
f\textcolor{black}{_\varepsilon}(r,\alpha,\beta)=\left(\frac{r-2}{8}10+\frac{10-r}{8}\cos T,\frac{r-2}{8}(\pi-\alpha)+\frac{10-r}{8}T,\beta\right).
$$
Note \textcolor{black}{that this is independent of $\varepsilon$ and} that the mapping $(10,\pi-\alpha,\beta)$ on $\partial B(0,10)$ is actually the identity mapping up to reflection $(x,y,z)\to(x,y,-z)$\textcolor{black}{, so after we compose it with this reflection as mentioned in Step 1 we obtain the identity on the boundary}. 
Our mapping $f\textcolor{black}{_\varepsilon}$ is a homeomorphism on  $\partial B(0,10)$ and $\partial B(0,2)$, so similarly to Step 3 it is enough to show that $J_{f\textcolor{black}{_\varepsilon}}\neq 0$ to obtain that it is a homeomorphism on $B(0,10)\setminus B(0,2)$. Using first line of \eqref{jakobian} it is enough to show that
$$
\begin{aligned}
0\neq& \partial_r\tilde{r} \partial_{\alpha}\tilde{\alpha}-\partial_{\alpha}\tilde{r} \partial_{r}\tilde{\alpha}\\
=&\frac{1}{8^2}\Bigl[(10-\cos T)\left(2-r+(10-r)\partial_{\alpha}T\right)-\left((10-r)(-\sin T)\partial_{\alpha}T\right)(\pi-\alpha-T)\Bigr]\\
=&\frac{1}{8^2}\Bigl[(10-\cos T)+(\pi-\alpha-T)\sin T\Bigr](10-r)\partial_{\alpha}T+\frac{1}{8^2}(10-\cos T)(2-r).\\
\end{aligned}
$$
Since $\partial_{\alpha}T<0$ and $-\partial_{\alpha}T>\frac{p}{2}\geq \frac{1}{4}$ for $p\in[\frac{1}{2},1]$ (and $\alpha\in[0,\pi]$) we obtain 
$$
\begin{aligned}
-8^2\left(\partial_r\tilde{r} \partial_{\alpha}\tilde{\alpha}-\partial_{\alpha}\tilde{r} \partial_{r}\tilde{\alpha}\right)
\geq \Bigl[(10-1)+(\pi-\pi-\frac{\pi}{2})\Bigr](10-r)\frac{1}{4}+(10-1)(r-2)\geq 1.\\
\end{aligned}
$$
To obtain integrability of $J_{f\textcolor{black}{_\varepsilon}}^{-a}$ it is thus enough to use first line of \eqref{jakobian} and show the finiteness of 
$$
\int_{B(0,10)\setminus B(0,2)}\frac{\bigl|r^2\sin \alpha\bigr|^{1+a}}
{\bigl|\tilde{r}\sin\tilde{\alpha}\bigr|^a}\; dr\; d\alpha\; d\beta. 
$$
Since 
$$
\tilde{r}\geq \cos T\geq C \alpha^p
$$ 
and using \eqref{elementary}
\eqn{tildealpha}
$$
\begin{aligned}
\tilde{\alpha}&=\frac{r-2}{8}(\pi-\alpha)+\frac{10-r}{8}\frac{\pi}{2}\left(1-\left(\frac{\alpha}{\pi}\right)^p\right)\\
&\approx \frac{r-2}{8}(\pi-\alpha)+\frac{10-r}{8}\Bigl(1-\frac{\alpha}{\pi}\Bigr)\approx \pi-\alpha\\
\end{aligned}
$$  
we obtain the convergence easily as $a<2$ and $\frac{1}{2}<p<1$. 

It remains to show the finiteness of $\int |Df\textcolor{black}{_\varepsilon}|^2$ on $B(0,10)\setminus B(0,2)$ using first line of \eqref{derivace}. 
Since $|\partial_r\tilde{r}|\leq C$ and $|\partial_r\tilde{\alpha}|\leq C$ it is easy to estimate first two terms. Further $r\approx 1$ and 
$$
|\partial_{\alpha}\tilde{r}|+|\partial_{\alpha}\tilde{\alpha}|\leq C+C|\partial_{\alpha}T|\leq C+C \alpha^{p-1}
$$
give us boundedness of
$$
\int_2^{10}\int_0^{\pi}\Bigl[\left(\frac{\partial_\alpha \tilde{r}}{r}\right)^2+\left(\frac{\tilde{r} \partial_\alpha \tilde{\alpha}}{r}\right)^2 \Bigr]r^2\sin\alpha\textcolor{black}{\; d\alpha}\; dr
\leq C \int_0^{\pi}(\alpha)^{2p-2}\alpha\; d\alpha.
$$
Using \eqref{tildealpha} it is easy to show convergence of the last term 
$$
\int_2^{10}\int_0^{\pi}\Bigl(\frac{\tilde{r}\sin\tilde{\alpha}}{r\sin\alpha}\Bigr)^2 r^2\sin\alpha\textcolor{black}{\; d\alpha}\; dr. 
$$
\textcolor{black}{We thus showed that when we extend $f_\varepsilon$, the energy \ref{energy_def} stays uniformly bounded for the whole sequence.}
 
\textcolor{black}{
{\bf Step 7. Violation of the (INV) condition:} 
Our limit mapping $f$ violates the (INV) condition on $B(0,r)$ for every $r\in(0,1)$. This can be easily seen as the mapping is continuous on $S(0,r)$, so we can consider the classical topological degree. 
We have
$$
R = \frac{2-r}{3}\text{ and }T=\frac{\pi}{2}\left(1-\left(\frac{\pi-\alpha}{\pi}\right)^p\right)
$$
on $B(0,1)\setminus\{0\}$.
The image of $S(0,r)$ is only the inner drop, so its topological image is the inside of this drop. However, when we take $0<r_1<r_2<1$, we can see that the smaller sphere is mapped onto a bigger drop, which contains the smaller drop (= image of the bigger ball). This shows that the material from $S(0,r_2)$ is ejected outside of $\im_T(f,B(0,r_2))$, which itself is enough to break the (INV) condition. Moreover, the material which is mapped into $\im_T(f,B(0,r_2))$ was originally outside of $B(0,r_2)$.
}

\end{proof}

\textcolor{black}
{
\begin{proof}[Proof of Theorem \ref{notthesame}] 
The mapping here is the same as mapping $f$ from Theorem \ref{example} which is a weak limit of homeomorphism $f_m$ from that statement. However, it does not satisfy the (INV) condition and hence Theorem \ref{distthm} (b) implies that it cannot be obtained as strong limit of homeomorphisms $h_m\in W^{1,2}$. 
\end{proof}
}

\textcolor{black}{
\section*{Acknowledgement}
We would like to thank the referees for carefully reading the manuscript and for many valuable comments which improved its readibility.
}

\end{document}